\documentclass [reqno] 
{amsart}

\copyrightinfo{2010}{American Mathematical Society}

\usepackage{amsfonts}
\usepackage{amssymb}

\usepackage{color}
\usepackage{hyperref}

\newtheorem{theorem}{Theorem}[section]
\newtheorem{lemma}[theorem]{Lemma}
\newtheorem{corollary}[theorem]{Corollary}

\theoremstyle{definition}

\theoremstyle{remark}
\newtheorem{remark}[theorem]{Remark}

\numberwithin{equation}{section}

\begin{document}

\title[An Open Mapping Theorem for the Navier-Stokes Equations]
      {An Open Mapping Theorem for \\ the Navier-Stokes Equations\footnote{This is a preprint version of the paper published in Advances and Applications in Fluid Mechanics, 21:2 (2018), 127-246}}

\author[A. Shlapunov]{Alexander Shlapunov}

\address{Siberian Federal University,
         Institute of Mathematics and Computer Science,
         pr. Svobodnyi 79,
         660041 Krasnoyarsk,
         Russia}

\email{ashlapunov@sfu-kras.ru}


\author[N. Tarkhanov]{Nikolai Tarkhanov}

\address{Institute of Mathematics,
         University of Potsdam,
         Karl-Liebknecht-Str. 24/25,
         14476 Potsdam OT Golm,
         Germany}

\email{tarkhanov@math.uni-potsdam.de}


\date{April 30, 2016}


\subjclass [2010] {Primary 76D05; Secondary 76N10, 35Q30}

\keywords{Navier-Stokes equations,
          weighted H\"older spaces,
          integral representation method}

\begin{abstract}
We consider the Navier-Stokes equations in the layer ${\mathbb R}^n \times [0,T]$ over $\mathbb{R}^n$
with finite $T > 0$.
Using the standard fundamental solutions of the Laplace operator and the heat operator, we reduce
the Navier-Stokes equations to a nonlinear Fredholm equation of the form $(I+K) u = f$, where $K$
is a compact continuous operator in anisotropic normed H\"older spaces weighted at the point at
infinity with respect to the space variables.
Actually, the weight function is included to provide a finite energy estimate for solutions to the
Navier-Stokes equations for all $t \in [0,T]$.
On using the particular properties of the de Rham complex we conclude that the Fr\'echet derivative
$(I+K)'$ is continuously invertible at each point of the Banach space under consideration and the
map $I+K$ is open and injective in the space.
In this way the Navier-Stokes equations prove to induce an open one-to-one mapping in the scale of
H\"older spaces.
\end{abstract}

\maketitle

\tableofcontents

\section*{Introduction}
\label{s.0}

The problem of describing the dynamics of incompressible viscous fluid is of great importance in
applications.
The dynamics is described by the Navier-Stokes equations and the problem consists in finding a
classical solution to the equations.
By a classical solution we mean here a solution of a class which is well motivated by applications
and for which a uniqueness theorem is available.
Essential contributions have been published in the research articles
   \cite{Lera34a,Lera34b},
   \cite{Kolm42},
   \cite{Hopf51},
   \cite{LadySolo60},
   \cite{Tao15}
as well as surveys and books
   \cite{Lady70,Lady03}),
   \cite{Lion61,Lion69},
   \cite{Tema79},
   \cite{FursVish80},
etc.

In physics by the Navier-Stokes equations is meant the impulse equation for the flow.
In the computational fluid dynamics the impulse equation is enlarged by the continuity and energy
equations.

The impulse equation of dynamics of (compressible) viscous fluid was formulated in differential form
independently by Claude Navier (1827) and George Stokes (1845).
This is
\begin{equation}
\label{eq.impulse}
   \rho (\partial_{t} v + (v \cdot \nabla) v)
 = \mu \varDelta v + (\lambda + \mu)\, \nabla\, \mathrm{div}\, v - \nabla p + f,
\end{equation}
where
   $v : \mathcal{X} \times (0,T) \to \mathbb{R}^3$ and
   $p : \mathcal{X} \times (0,T) \to \mathbb{R}$
are the search-for velocity vector field and pressure of a particle in the flow, respectively, and
   $\mathcal{X}$ is a domain in the Euclidean space $\mathbb{R}^3$,
   $(0,T)$ is an interval of the time axis.
Furthermore,
   the number $\rho$ stands for the mass density,
   $\lambda$ and $\mu$ are the first Lam\'{e} constant and the dynamical viscosity of the fluid under
   consideration, respectively,
   $\varDelta  =  \partial^2_{x^1 x^1} + \partial^2_{x^2 x^2} + \partial^2_{x^3 x^3}$ is the Laplace operator in
   $\mathbb{R}^3$,
   $\nabla$ and $\mathrm{div}$ are the gradient operator and the divergence operator in $\mathbb{R}^3$,
   respectively,
and
   $f$ is the density vector of outer forces, such as gravitation and so on,
see formulas (15.5) and (15.6) in \cite[\S~15]{LaLi}, \cite{Tema79} and elsewhere.

Usually the impulse equation is supplemented by the continuity equation
$$
   \partial_t \rho + \mbox{div} (\rho v) = g,
$$
see \cite[\S~15]{LaLi}.

In order to specify a particular solution of equations (\ref{eq.impulse}), one usually considers the first
mixed problem in the cylinder $\mathcal{X} \times (0,T)$ by posing the initial conditions on the lower basis
of the cylinder and a Dirichlet condition on the lateral surface.
To wit,
\begin{equation}
\label{eq.mbvp}
\begin{array}{rclcl}
   v (x,0)
 & =
 & v_0 (x),
 & \mbox{for}
 & x \in \mathcal{X},
\\
   v (x,t)
 & =
 & v_l (x,t),
 & \mbox{for}
 & (x,t) \in \partial \mathcal{X} \times (0,T).
\end{array}
\end{equation}
It is worth pointing out that the pressure $p$ is determined solely from the impulse equation up to an
additive constant.
To fix this constant it suffices to put a moment condition on $p$.

If the density $\rho$ does not change along the trajectories of particles, the flow is said to be
incompressible.
It is the assumption that is most often used in applications.
For incompressible fluid the continuity equation takes the especially simple form $\mathrm{div}\, v = 0$
in $\mathcal{X} \times (0,T)$, i.e., the vector field $v$ should be divergence free (solenoidal).
In many practical problems the flow is not only incompressible but it has even a constant density.
In this case one can just set $\rho = 1$ in (\ref{eq.impulse}) which reduces the impulse equation to
\begin{equation}
\label{eq.NS.cyl}
\begin{array}{rcl}
   \partial_{t} v + (v \cdot \nabla) v - \mu \varDelta v + \nabla p
 & = &
   f,
\\
   \mathrm{div}\, v
 & = &
  0
\end{array}
\end{equation}
in $\mathcal{X} \times (0,T)$.
In this way we obtain what is referred to as but the Navier-Stokes equations.

After J. Leray \cite{Lera34a,Lera34b}, a great attention was paid to weak solutions to (\ref{eq.NS.cyl})
with boundary conditions (\ref{eq.mbvp}).
Hopf \cite{Hopf51} proved that equations (\ref{eq.NS.cyl}) under homogeneous data (\ref{eq.mbvp}) have a
weak solution satisfying reasonable estimates.
However, in this full generality no uniqueness theorem for a weak solution has been known.
On the other hand, under stronger conditions on the solution it is unique, see \cite{Lady70,Lady03}, who
proved the existence of a smooth solution for the two-dimensional version of the problem.

Traditionally two main directions have been formed in the study of the Navier-Stokes equations.
The first direction is concerned with improvement of the regularity of the weak solution of Hopf \cite{Hopf51}
using a priori estimates or more refined methods.
The second one is based on the fixed point theorems like those by Banach or Schauder or mapping degree theory
which allow one to attack the nonlinear problem directly.

On these ways one usually looks first for a suitable uniqueness class of solutions to the problem.
When working in Lebesgue and Sobolev spaces of positive or negative smoothness, one is aimed at obtaining
solvability theorems of certain linearised versions of the problem in the spaces of generalised functions
and establishing a priori estimates for weak solutions (see for instance the monographs
   \cite{Lady70},
   \cite{Tema79}
and the references given there).
However, this does not lead to any breakthrough in the original nonlinear problem, as elementary examples
like $y = \exp x$ show.
Furthermore, on passing to the nonlinear problem in spaces of distributions one encounters the additional
problem on multiplication of distributions.
Hence, there is strong feeling that weak solutions do not completely fit to handle the nonlinearity of the
Navier-Stokes equations.

Actually this observation has motivated clearly the investigation of the linear and quasilinear systems
of parabolic equations in (possibly, weighted) H\"older spaces over cylindrical domains, see for instance
    \cite{LadSoUr67},
    \cite{Sol64},
    \cite{Sol65},
    \cite{Bel79},
    \cite{BiSol93},
    \cite{Sol06},
etc.
Although the use of H\"{o}lder spaces guarantees an uniqueness theorem for the Navier-Stokes equations, one
is still not able to derive an existence theorem in this function scale.
What is usually lacking is the compactness or smallness of the nonlinear term with respect to the parabolic
linear part of the equations.
The weighted H\"{o}lder spaces we introduce in the present article serve mainly to get rid of this drawback.

We now specify the contribution of our paper to a huge amount of works on the Navier-Stokes equations.
In the sequel we consider an initial problem for the Navier-Stokes equations in the case of non-compressible
fluid corresponding to $\rho = 1$.
Namely,
\begin{equation}
\label{eq.NS}
\begin{array}{rclcl}
   \partial_t u -  \mu \varDelta u + (u \cdot \nabla) u + \nabla p
 & =
 & f,
 & \mbox{if}
 & (x,t) \in \mathcal{X} \times (0,T),
\\
   \mbox{div}\, u
 & =
 & 0,
 & \mbox{if}
 & (x,t) \in \mathcal{X} \times (0,T),
\\[.2cm]
   u (x,0)
 & =
 & u_0 (x),
 & \mbox{if}
 & x \in \mathcal{X},
\end{array}
\end{equation}
where
   $\mathcal{X} = \mathbb{R}^n$ with $n \geq 2$,
   $T > 0$ is finite,
and
   $\mu > 0$ is a viscosity constant.
As usual, the boundary conditions in this case are replaced by proper asymptotic behaviour of the
solution at the point of infinity 
for the space of smooth functions).

We develop an operator theoretic approach to the Navier-Stokes equations.
The focus is on elaborating a scale of weighted H\"{o}lder spaces which provides the openness of the
map induced by the Navier-Stokes equations, and the compactness of the nonlinear term.

Basically an open mapping theorem just amounts to an existence theorem for all data which are close to
any element of the range of the operator.
To the best of our knowledge there have been known no such results for finite energy solutions in the
scale of H\"{o}lder spaces on a cylinder with unbounded basis.
Theorems 10 and 11 of \cite[Ch.~4, \S~4]{Lady70} contain similar results for weak solutions in spaces of the Lebesgue type 
(see also \cite{Gal13} for global stability results in the spaces of infinitely differentiable functions).

On the other hand, when combined with the invertibility of the linear part in the Navier-Stokes equations,
the compactness of the nonlinear term enables one to invoke the mapping degree theory to get an existence
theorem, if there is any.
This is precisely on what the most of current investigations of the Navier-Stokes equations are focused.

Let us dwell on the content of the paper and the choice of function spaces in detail.
It should be first noted that since the gradient, rotation and divergence operators are of steady use in
the models of hydrodinamics we use the language of exterior differential forms and the de Rham complex
to treat the Navier-Stokes equations, see Section \ref{s.NS.deRham}.
This allows one to immediately specify the Navier-Stokes equations within the framework of global analysis
on smooth compact manifolds with boundary.
However, we have not been able to uniquely identify the nature of nonlinearity on forms of degree greater
than one.

It is well known that the operator $\varDelta \left( \partial_t - \mu \varDelta \right)$ admits a matrix
factorisation through the Stokes operator (the principal linear part of the Navier-Stokes equations), see
\cite[\S~15]{LaLi} or Section \ref{s.NS.deRham} below.
Hence, on aiming at investigation of the Navier-Stokes equations on functions vanishing at the infinitely
distant point, one should begin with the study of invertibility of the Laplace and heat operators in the
spaces in question.
The point at infinity in $\mathbb{R}^n$ is naturally thought of as a conical point of the one-point
compactificaion of $\mathbb{R}^n$.
The analysis close to this point is traditionally based on the use of weighted spaces, see
   \cite{Kond66}, \cite{HiMaw96} for parabolic problems,
   \cite{AmMeNe14} for Stokes-type equations or
   \cite{McOw79} for elliptic problems in weighted Sobolev spaces.
Still the Laplace operator acts properly in weighted Sobolev spaces of square integrable functions on
$\mathbb{R}^n$ merely for $n \geq 5$.
Apart from difficulties with multiplication and possible lack of uniqueness this is an evidence for us
to choose the scale of weighted H\"{o}lder spaces over $\mathbb{R}^n$ instead of Sobolev spaces, see
   Section \ref{s.HoelderSpaces}.

We first introduce weighted H\"{o}lder spaces over $\mathbb{R}^n$ as in elliptic theory and then
anisotropic H\"{o}lder spaces over the layer $\mathbb{R}^n \times [0,T]$ of finite width $T > 0$.
The explicit construction of these Banach spaces provides appropriate embedding theorems including those
on compact embedding, see Section \ref{s.HoelderSpaces}.
Taking the results by
   \cite{McOw79},
   \cite{HiMaw98} and
   \cite{Be11}
as a starting point, we investigate the Laplace operator and the heat operator in anisotropic weighted
H\"older spaces with weight functions prescribing a proper behaviour of its elements at the point at infinity,
see
   Sections \ref{s.deRham.HS} and \ref{s.thoitwHs}.
As usual, a set of prohibited weights appears, for which the action of the Laplace operator fails to be
Fredholm.
Note that the range of weight exponents, for which the Laplace operator is continuously invertible, is rather
narrow.

As but one useful tool we derive theorems on the invertibility of the differential of the de Rham complex
in both elliptic and parabolic scales of weighted \textcolor{red}{H\"older} spaces, see Section \ref{s.deRham.HS}.
They actually constitute certain versions of the classical Hodge theory on the one-point compactification
of $\mathbb{R}^n$, cf. also the particular decompositions of
   \cite[Ch.~1, \S~2]{Lady70},  \cite[Ch.~1, \S~1.4]{Tema79} and 
	{\cite[\S 1.8]{BerMaj02}
related to the rotation operator.

The actions of the Laplace and heat operators in the weighted H\"{o}lder spaces under study prove to be not
fully coherent.
To wit, in order to achieve the invertibility of the Laplace operator for lower dimensions $2 \leq n \leq 4$
or higher smoothness, we deal with parabolic H\"{o}lder spaces, where the dilation principle is partially
neglected with regard to the weight.
As a result we lose some regularity and weight in the Cauchy problem for the heat equation in the scale of
parabolic H\"{o}lder spaces under consideration.
More precisely, the loss of regularity occurs with respect to the smoothness of solution in the standard
H\"{o}lder spaces over  a cylinder domain with bounded base.

Using a familiar trick excluding the pressure, and the standard fundamental solutions of the Laplace and heat
operators, we study linearisations of the Navier-Stokes equations which can be reduced to Fredholm equations
of the form $(I + K') g = g_0$, where $K'$ is a compact pseudodifferential mapping of anisotropic H\"{o}lder
spaces.
As a consequence we get an existence theorem for linearisations of the Navier-Stokes equations, see
   Section \ref{s.NS.lin}.
Note that there is a loss of regularity in the existence theorem for the linearised Navier-Stokes equations
comparing with that for the reduced equations.

Further development allows one to reduce the Navier-Stokes equations to an operator equation for the Fredholm
type operator $I+K$ with a nonlinear compact continuous mapping $K$ in anisotropic normed H\"ol\-der spaces
weighted at the point at infinity with respect to the space variables.
Actually, the weight function is chosen in such a way that the finite energy estimate be fulfilled for
solutions of the Navier-Stokes equations for all $t \in [0,T]$.
Next, using the properties of the de Rham complex we conclude that the Fr\'echet derivative $(I+K)'$ of the
map is continuously invertible at every point of the Banach space under the consideration and the map $I+K$
is open and injective in the space.
This implies readily that the reduced Navier-Stokes equations induce an open mapping on the scale of H\"older
spaces.
Again, a loss of regularity occurs in the open mapping theorem for the Navier-Stokes equations comparing with
that for the reduced equations.
However, the soft formulation of the open mapping theorem is strengthened to a rigorous result is the Navier-Stokes
equations are given a domain being a metric space, see Section \ref{s.NS.OpenMap}.

\part{Preliminaries}
\label{p.preliminaries}

\section{The Navier-Stokes equations and the de Rham complex}
\label{s.NS.deRham}

Let
   $\mathbb{Z}_{\geq 0}$ be the set of all natural numbers including zero,
and let $\mathbb{R}^n$ be the Euclidean space of dimension $n \geq 2$ with coordinates $x=(x^1, \ldots, x^n)$.
For a domain $\mathcal{X}$ in $\mathbb{R}^n$, we often consider the open cylinder
   $\mathcal{C}_T (\mathcal{X}) := \mathcal{X} \times (0,T)$
over $\mathcal{X}$ in the space $\mathbb{R}^{n+1}$ of variables $(x,t)$.
If $\mathcal{X} = \mathbb{R}^n$, we write
   $\mathcal{C}_T = \mathbb{R}^n \times (0,T)$
for short and we sometimes refer to $\mathcal{C}_T$ as a layer over $\mathbb{R}^n$.

We are going to rewrite the nonlinear Navier-Stokes equations (\ref{eq.NS}) in a more convenient form (the so-called Lamb form).
To this end, denote by $\varLambda^q$ the bundle of exterior forms of degree $0 \leq q \leq n$ over
$\mathbb{R}^n$.
Given a domain $\mathcal{X}$ in $\mathbb{R}^n$, we write $\varOmega^q (\mathcal{X})$ for the space of
all differential forms of degree $q$ with $C^\infty$ coefficients on $\mathcal{X}$.
These space constitute the so-called de Rham complex $\varOmega^\cdot (\mathcal{X})$ on $\mathcal{X}$
whose differential is given by the exterior derivative $d$.
To display $d$ acting on $q\,$-forms one uses the designation
   $du := d^q u$
for $u \in \varOmega^q (\mathcal{X})$ (see for instance \cite{Tark95a}).
For the space of differential forms of degree $q$ with coefficients of a class $\mathcal{F}$ in $\mathcal{X}$
we have to use more cumbersome designation $\mathcal{F} (\mathcal{X},\varLambda^q)$.

By $\mathcal{F} (\mathcal{X} \times (0,T),\varLambda^q)$ is meant the space of all differential forms of
degree $q$ in $x$ whose coefficients are functions of a class $\mathcal{F} (\mathcal{X} \times (0,T))$.
They have the form
$$
  u (x,t) = \sum_{1 \leq i_1 < \ldots < i_q \leq n} u_I (x,t) dx^I
$$
where
   the sum is over all increasing multi-indices $I = (i_1, \ldots, i_q)$ of the numbers $1, \ldots, n$,
   $dx^I$ stands for the exterior product of the differentials  $dx^{i_1},  \ldots, dx^{i_q}$ after each other,
and
   $u_I (x,t)$ are functions of class $\mathcal{F} (\mathcal{X} \times (0,T))$.
The parameter $t$ is included only in the coefficients.

Consider the de Rham complex $\varOmega^\cdot (\mathbb{R}^{n+1}_{t \geq 0})$ on the closed half-space
$\mathbb{R}^{n+1}_{t \geq 0}$, i.e.,
\begin{equation*}
   0
 \rightarrow
   \varOmega^0 (\mathbb{R}^{n+1}_{t \geq 0})
 \stackrel{d}{\rightarrow}
   \varOmega^1 (\mathbb{R}^{n+1}_{t \geq 0})
 \stackrel{d}{\rightarrow}
   \ldots
 \stackrel{d}{\rightarrow}
   \varOmega^n (\mathbb{R}^{n+1}_{t \geq 0})
 \rightarrow
   0,
\end{equation*}
where
   $t$ is thought of as a parameter and the coefficients of differential forms are smooth both in $x$ and $t$.
We denote by $d^\ast$ the formal adjoint operator for $d$, more precisely,
   $d^\ast g = (d^{q-1})^\ast g$
for $g \in \varOmega^{q} (\mathbb{R}^{n+1}_{t \geq 0})$.
The Laplacian of $\varOmega^\cdot (\mathbb{R}^{n+1}_{t \geq 0})$ evaluated on $q\,$-forms reduces to
\begin{equation}
\label{eq.deRham}
   \varDelta^q
 := d^\ast d + d d^\ast
  = - E_{k_q}\, \varDelta,
\end{equation}
where
   $k_q = \left( {n \atop k} \right)$,
   $E_{k_q}$ is the unit $(k_q \times k_q)\,$-matrix
and
   $\varDelta$ the Laplace operator applied componentwise in the space variable $x$.
Equality (\ref{eq.deRham}) means that
$$
  (\varDelta^q u) (x,t) = - \sum_{1 \leq i_1 < \ldots < i_q \leq n} (\varDelta u_I) (x,t) dx^I.
$$
For the most important case $n=3$, the differential forms of degree $0$ and $3$ can be identified with 
 functions and the differential forms of degree $1$ and $2$ can be identified with vector-valued functions 
 having three components.
 Then one readily gets
$$
\begin{array}{cccccc}
   d^0 = \nabla, 
 & d^0{}^* = - \mathrm{div}, 
 & d^1 = \mathrm{rot}, 
 & d^1{}^* = \mathrm{rot}, 
 & d_2 = \mathrm{div}.
\end{array}
$$

Write $H_{\mu} = \partial_t - \mu \varDelta$ for the heat operator in $\mathbb{R}^{n+1}_{t \geq 0}$ with
a constant $\mu > 0$.
When extended componentwise to differential forms of degree $q$ it can be written in the form
$$
   H_\mu^q
 = E_{k_q}\, \partial_t + \mu \varDelta^q.
$$
By abuse of notation we will omit the index $q$ and write it simply $H_\mu$ if it causes no confusion.

Since scalar partial differential operators with constant coefficients commute, we obviously deduce
that
\begin{equation}
\label{eq.commute}
\begin{array}{rcl}
   d H_\mu
 & =
 & H_\mu d,
\\
   d^\ast H_\mu
 & =
 & H_\mu d^\ast,
\end{array}
\end{equation}
etc.

If we identify a function $u (x,t) = (u_1 (x,t), \ldots, u_n (x,t))$ with values in $\mathbb{R}^n$ with the
differential form
$$
   u (x,t)= \sum_{i=1}^n u_i (x,t) dx^i,
$$
then using the de Rham complex above we introduce the block matrix of partial differential operators
\begin{equation*}
   \Big( \begin{array}{cc}
           H_\mu
         & d
\\
           d^\ast
         & 0
         \end{array}
   \Big) :\
   \begin{array}{c}
   \varOmega^1 (\mathbb{R}^{n+1}_{t \geq 0})
\\
   \oplus
\\
   \varOmega^0 (\mathbb{R}^{n+1}_{t \geq 0})
   \end{array}
 \to
   \begin{array}{c}
   \varOmega^1 (\mathbb{R}^{n+1}_{t \geq 0})
\\
   \oplus
\\
   \varOmega^0 (\mathbb{R}^{n+1}_{t \geq 0})
   \end{array}
\end{equation*}
for the linear part of the Navier-Stokes operator of (\ref{eq.NS.cyl}).
The following factorisation was implicitly used in \cite[\S~15]{LaLi}.

\begin{lemma}
\label{l.factorization}
We have
\begin{eqnarray*}
\Big( \begin{array}{cc}
        H_\mu^1
      & d^0
\\
        d^0{}^\ast
      & 0
      \end{array}
\Big)
\Big( \begin{array}{cc}
        d^1{}^\ast d^1
      & H_\mu^1 d^0
\\
        d^0{}^\ast H_\mu^1
      & - (H_\mu^0)^2
      \end{array}
\Big)
 & = &
\Big( \begin{array}{cc}
        \varDelta^1 H_\mu^1
      & 0
\\
        0
      & \varDelta^0 H_\mu^0
      \end {array}
\Big),
\\
\Big( \begin{array}{cc}
        d^1{}^\ast d^1
      & H_\mu^1 d^0
\\
        d^0{}^\ast H_\mu^1
      & - (H_\mu^0)^2
      \end {array}
\Big)
\Big( \begin{array}{cc}
        H_\mu^1
      & d^0
\\
        d^0{}^\ast
      & 0
\end{array}
\Big)
 & = &
\Big( \begin{array}{cc}
        \varDelta^1 H_\mu^1
      & 0
\\
        0
      & \varDelta^0 H_\mu^0
      \end{array}
\Big).
\end{eqnarray*}
\end{lemma}

\begin{proof}
On using (\ref{eq.deRham}), (\ref{eq.commute}) and the rules of multiplication of block-matrices we see
that
$$
\begin{array}{rclcrcl}
   H_\mu d^1{}^\ast d^1 + d^0 d^0{}^\ast H_\mu
 & =
 & \varDelta^1 H_\mu^1,
 &
 & H_\mu H_\mu d^0 - d^0 (H_\mu)^2
 & =
 & 0,
\\
   d^0{}^\ast d^1{}^\ast d^1 + 0\, d^0{}^\ast H_\mu
 & =
 & 0,
 &
 & d^0{}^\ast H_\mu d^0 - 0\, (H_\mu)^2
 & =
 & \varDelta^0 H_\mu^0,
\end{array}
$$
which proves the first formula.
On the other hand,
$$
\begin{array}{rclcrcl}
   d^1{}^\ast d^1 H_\mu + H_\mu d^0 d^0{}^\ast
 & =
 & \varDelta^1 H_\mu^1,
 &
 & d^1{}^\ast d^1 d^0 + H_\mu d^0\, 0
 & =
 & 0,
\\
   d^0{}^\ast H_\mu H_\mu - (H_\mu)^2 d^0{}^\ast
 & =
 & 0,
 &
 & d^0{}^\ast H_\mu d^0 - (H_\mu)^2\, 0
 & =
 & \varDelta^0 H_\mu^0.
\end{array}
$$
This proves the second formula.
\end{proof}

The lemma shows that the investigation of the Navier-Stokes equations is closely related to the
behaviour of the Laplace and the heat operators on functions over $\mathbb{R}^{n}$ and
$\mathbb{R}^{n+1}_{t \geq 0}$, respectively.
For weighted Sobolev spaces this behaviour has been studied in the papers
   \cite{McOw79},
   \cite{HiMaw98},
see also \cite{Be11}, \cite{Mar02}.

Our next objective is to rewrite the nonlinear part of the Navier-Stokes equations in terms of the
de Rham complex.
Given a smooth vector field $v$ on $\mathbb{R}^n$, the derivative of $v$ in the direction of $v$ is
called the substantial derivative of $v$ and denoted by $\mathbf{D}^1 v := (v \cdot \nabla) v$.
To obtain a useful description we make use of the so-called Hodge star operator on $\mathbb{R}^n$
$$
   \ast : \varOmega^q (\mathbb{R}^n) \to \varOmega^{n-q} (\mathbb{R}^n)
$$
defined by linearity from
$
   dx^I \wedge (\ast dx^I) = dx
$
for all multi-indices $I = (i_1, \ldots, i_q)$ with $1 \leq i_1 < \ldots < i_q \leq n$.
As is known,
\begin{equation}
\label{eq.Hodge}
\begin{array}{rcl}
   \ast (\ast dx^I)
 & =
 & (-1)^{(n-q) q} dx^I,
\\
   u \wedge \ast v
 & =
 & \displaystyle
   \Big( \sum_{1 \leq i_1 < \ldots < i_q \leq n} u_I v_I \Big) dx
\end{array}
\end{equation}
and so $\ast (u \wedge \ast u) = |u|^2$ holds for all differential forms on $\mathbb{R}^n$ with
real-valued coefficients.

The following lemma is well known, see for instance \cite[\S~15]{LaLi}.

\begin{lemma}
\label{l.D.deRham}
For any smooth differential forms $u$ and $v$ of degree one in $\mathbb{R}^n$ it follows that
\begin{equation}
\label{eq.D.deRham}
   (v \cdot \nabla) u + (u \cdot \nabla) v
 = d^0 \ast (\ast v \wedge u) +  \ast ((\ast d^1 u) \wedge v) +  \ast ((\ast d^1 v) \wedge u).
\end{equation}
\end{lemma}

\begin{proof}
Indeed,
\begin{equation}
\label{eq.D.deRham.1}
   d \ast (\ast v \wedge u)
 = \sum_{j=1}^n
   \sum_{i=1}^n \left( v_i \partial_j u_i + u_i \partial_j v_i \right)
   dx^j
\end{equation}
and
\begin{eqnarray*}
   \ast du
 & = &
   \sum_{i < j} \left( \partial_i u_j - \partial_j u_i \right) \ast (dx^i \wedge dx^j)
\\
 & = &
   \sum_{i < j} \left( \partial_i u_j - \partial_j u_i \right) (-1)^{i+j-1} dx [i,j],
\end{eqnarray*}
where $dx [i,j]$ is the exterior product of the differentials $dx^1, \ldots, dx^n$ after each other
among which $dx^i$ and $dx^j$ are omitted.
Then
\begin{eqnarray*}
   (\ast du) \wedge v
 \! & \! = \! & \!
   \sum_{k=1}^n
   \Big( \sum_{i < j} (\partial_i u_j - \partial_j u_i) v_k  (-1)^{i+j-1} dx [i,j] \wedge dx^k \Big)
\\
 \! & \! = \! & \!
   \sum_{i < j} (\partial_i u_j - \partial_j u_i) v_j  (-1)^{n+i-1} dx [i]
 + \sum_{i < j} (\partial_i u_j - \partial_j u_i) v_i  (-1)^{n+j} dx [j],
\end{eqnarray*}
and so using (\ref{eq.Hodge}) yields
\begin{eqnarray*}
   \ast ((\ast du) \wedge v)
 & = &
   \sum_{i < j} (\partial_i u_j - \partial_j u_i) v_i dx^j
 - \sum_{i < j} (\partial_i u_j - \partial_j u_i) v_j dx^i
\\
 & = &
   \sum_{i < j} (\partial_i u_j - \partial_j u_i) v_i dx^j
 - \sum_{i > j} (\partial_j u_i - \partial_i u_j) v_i dx^j
\\
 & = &
   \sum_{j=1}^n
   \Big( \sum_{i \neq j} v_i \partial_i u_j - \sum_{i \neq j} v_i \partial_j u_i \Big)
   dx^j.
\end{eqnarray*}
The same reasoning shows that
$$
   \ast ((\ast dv) \wedge u)
 = \sum_{j=1}^n
   \Big( \sum_{i \neq j} u_i \partial_i v_j - \sum_{i \neq j} u_i \partial_j v_i \Big)
   dx^j.
$$
Since, by (\ref{eq.D.deRham.1}),
$$
   d \ast (\ast v \wedge u)
 = \sum_{j=1}^n
   \Big( v_j \partial_j u_j + u_j \partial_j v_j
       + \sum_{i \neq j} \left( v_i \partial_j u_i + u_i \partial_j v_i \right)
   \Big)
   dx^j,
$$
on gathering the above equalities we arrive at (\ref{eq.D.deRham}), as desired.
\end{proof}

In particular, on choosing $u = v$ we get
$
   \mathbf{D}^1 u = d^0 \left( \ast (\ast u \wedge u)/2 \right) + \ast \left( (\ast d^1 u) \wedge u \right)
$
for all one-forms $u$.

\section{The H\"older spaces weighted at the point at infinity}
\label{s.HoelderSpaces}

Suppose that $\mathcal{X}$ is a (possibly, unbounded) domain in $\mathbb{R}^n$ with smooth boundary, where
$n \geq 1$.

For $s = 0, 1, \ldots$, denote by $C^{s,0} (\overline{\mathcal{X}})$ the space of all $s$ times continuously
differentiable functions on $\overline{\mathcal{X}}$ with finite norm
$$
   \| u \|_{C^{s,0} (\overline{\mathcal{X}})}
 = \sum_{|\alpha| \leq s}
   \sup_{x \in \overline{\mathcal{X}}}
   |\partial^\alpha u (x)|.
$$
Given any $0 < \lambda \leq 1$, we set
$$
   \langle u \rangle_{\lambda,\overline{\mathcal{X}}}
 = \sup_{x,y \in \overline{\mathcal{X}} \atop x \neq y}
   \frac{|u (x) - u (y)|}{|x-y|^\lambda}
$$
and write $C^{0,\lambda} (\overline{\mathcal{X}})$ for the space of all continuous functions on the closure
of $\mathcal{X}$ with finite norm
$$
   \| u \|_{C^{0,\lambda} (\overline{\mathcal{X}})}
 = \| u \|_{C^{0,0} (\overline{\mathcal{X}})} + \langle u \rangle_{\lambda,\overline{\mathcal{X}}},
$$
the so-called H\"{o}lder space.
More generally, for $s = 0, 1, \ldots$, let $C^{s,\lambda} (\overline{\mathcal{X}})$ stand for the space of
all $s$ times continuously differentiable functions on $\overline{\mathcal{X}}$ with finite norm
$$
   \| u \|_{C^{s,\lambda} (\overline{\mathcal{X}})}
 = \| u \|_{C^{s,0} (\overline{\mathcal{X}})}
 + \sum_{|\alpha| \leq s} \langle \partial^{\alpha} u \rangle_{\lambda, \overline{\mathcal{X}}}.
$$
The normed spaces $C^{s,\lambda} (\overline{\mathcal{X}})$ with two indices $s \in \mathbb{Z}_{\geq 0}$ and
$\lambda \in [0,1]$ are known to be Banach spaces.

We are next going to control the growth of functions on $\mathcal{X}$ at the point at infinity.
Set
$$
\begin{array}{rcl}
   w (x)
 & =
 & \sqrt{1 + |x|^2},
\\
   w (x,y)
 & =
 & \max \{ w (x), w (y) \} \sim \sqrt{1 + |x|^2 + |y|^2}
\end{array}
$$
for $x, y \in \mathbb{R}^n$.
Let $\delta \in \mathbb{R}$.
(Note that $\delta$ is tacitly assumed to be nonnegative.)
Denote by $C^{s,0,\delta} (\overline{\mathcal{X}})$ the space of all $s$ times continuously differentiable
functions on $\overline{\mathcal{X}}$ with finite norm
$$
   \| u \|_{C^{s,0,\delta} (\overline{\mathcal{X}})}
 = \sum_{|\alpha| \leq s}
   \sup_{x \in \overline{\mathcal{X}}}
   (w (x))^{\delta+|\alpha|}
   |\partial^\alpha u (x)|.
$$
For $0 < \lambda \leq 1$, we introduce
$$
   \langle u \rangle_{\lambda,\delta, \overline{\mathcal{X}}}
 = \sup_{{{x,y \in \overline{\mathcal{X}} \atop x \neq y} \atop |x-y| \leq |x|/2}}
   (w (x,y))^{\delta+\lambda} \frac{|u (x) - u (y)|}{|x-y|^\lambda}.
$$
If $\mathcal{X}$ does not contain the origin, we define $C^{0,\lambda,\delta} (\overline{\mathcal{X}})$ to
consist of all continuous functions on $\overline{\mathcal{X}}$, such that
$$
\begin{array}{rcl}
   \| u \|_{C^{0,\lambda,\delta} (\overline{\mathcal{X}})}
 & =
 & \| u \|_{C^{0,0,\delta} (\overline{\mathcal{X}})} + \langle u \rangle_{\lambda,\delta,\overline{\mathcal{X}}}
\\
 & <
 & \infty.
\end{array}
$$
For those domains $\mathcal{X}$ which contain the origin we have also to control the H\"{o}lder property close
to $0$.
Hence, we let
   $C^{0,\lambda,\delta} (\overline{\mathcal{X}})$
be the space of all continuous functions on $\overline{\mathcal{X}}$ with finite norm
$$
   \| u \|_{C^{0,\lambda,\delta} (\overline{\mathcal{X}})}
 = \| u \|_{C^{0,\lambda} (\overline{U})}
 + \| u \|_{C^{0,0,\delta} (\overline{\mathcal{X}})} + \langle u \rangle_{\lambda,\delta,\overline{\mathcal{X}}},
$$
where $U$ is a small neighbourhood of the origin in $\mathcal{X}$.
Finally, for $s \in \mathbb{Z}_{\geq 0}$, we introduce $C^{s,\lambda,\delta} (\overline{\mathcal{X}})$ to be
the space of all $s$ times continuously differentiable functions on $\overline{\mathcal{X}}$  with finite norm
$$
   \| u \|_{C^{s,\lambda,\delta} (\overline{\mathcal{X}})}
 = \sum_{|\alpha| \leq s}
   \| \partial^\alpha u \|_{C^{0,\lambda,\delta+|\alpha|} (\overline{\mathcal{X}})}.
$$

The normed spaces $C^{s,\lambda,\delta} (\overline{\mathcal{X}})$ constitute a scale of Banach spaces
parametrised by
   $s \in \mathbb{Z}_{\geq 0}$,
   $\lambda \in [0,1]$ and
   $\delta \in \mathbb{R}$.
We will mostly consider the case $\mathcal{X} = \mathbb{R}^n$ and $U$ being the unit open ball
   $B_1 = B (0,1)$
in $\mathbb{R}^n$.
We will write simply $C^{s,\lambda,\delta}$ for the corresponding space when no confusion can arise.
The properties of the scale are well known.
Actually this construction is natural if we think of $\mathbb{R}^n$ as a manifold with a singular point at
the point of infinity.
The first summand corresponds to a coordinate chart in the nonsingular part of the manifold while the last
two summands of the norm correspond to a coordinate chart of the point at infinity.
We need not glue together these summands into one norm by means of partition of unity on $\mathbb{R}^n$,
for there are global coordinates in all of $\mathbb{R}^{n}$ and the operators under consideration possess
the transmission property.

\begin{lemma}
\label{l.emb.loc}
The space $C^{s,\lambda,\delta} (\mathbb{R}^{n})$ is embedded continuously into the Fr\'{e}chet space
   $C^{s,\lambda}_{\mathrm{loc}} (\mathbb{R}^{n})$.
\end{lemma}

\begin{proof}
By definition, the space $C^{s,\lambda,\delta} (\mathbb{R}^{n})$ is  embedded continuously into
$C^{s,0}_{\mathrm{loc}} (\mathbb{R}^{n})$ for all $\lambda \in [0,1]$.
In addition, it is  embedded continuously into $C^{s,\lambda} (\overline{B_1})$.
Take $x_0 \in \mathbb{R}^n$ away from the closure of $B_1$ and choose any ball $B (x_0,r)$ around
$x_0$ with radius $0 < r < 1$.
Since $|x| \geq |x_0| - r \geq 1 - r$ for all $x$ in the closure of $B (x_0,r)$, it follows that
\begin{eqnarray*}
   \langle u \rangle_{\lambda,\overline{B (x_0,r)}}
 & \leq &
   \sup_{x,y \in \overline{B (x_0,r)} \atop |x-y| \leq |x|/2}
   \frac{|u (x) - u (y)|}{|x-y|^\lambda}
 + \sup_{{x,y \in \overline{B (x_0,r)} \atop x \neq y} \atop |x-y| > |x|/2}
   \frac{|u (x) - u (y)|}{|x-y|^\lambda}
\\
 & \leq &
   \langle u \rangle_{\lambda,\delta,\overline{B (x_0,r)}}
 + \frac{2^{\lambda + 1}}{(1-r)^\lambda}
   \sup_{x \in \overline{B (x_0,r)}}
   (w (x))^\delta |u (x)|
\\
 & \leq &
   \Big( 1 + \frac{2^{\lambda + 1}}{(1-r)^\lambda} \Big)\,
   \| u \|_{C^{0,\lambda,\delta} (\mathbb{R}^n)}.
\end{eqnarray*}
Therefore, $C^{0,\lambda,\delta} (\mathbb{R}^n)$ is embedded continuously into
   $C^{0,\lambda} (\overline{B (x_0,r)})$
for all $x_0 \in \mathbb{R}^n$ with $|x_0| > 1$ and for any $r \in (0,1)$.
By the Heine-Borel theorem, the space $C^{0,\lambda,\delta} (\mathbb{R}^n)$ is embedded continuously into
   $C^{0,\lambda} (\overline{\mathcal{X}})$
for any bounded domain $\mathcal{X}$ in $\mathbb{R}^n$, as desired.
\end{proof}

\begin{lemma}
\label{l.emb.L2}
If $\delta > n/2$, then there is a constant $c = c (\delta) > 0$ such that
$$
   \| u \|_{L^2 (\mathbb{R}^n)}
 \leq c (\delta)\, \| u \|_{C^{0,0,\delta} (\mathbb{R}^{n})}
$$
for all $u \in C^{0,0,\delta} (\mathbb{R}^{n})$.
\end{lemma}

\begin{proof}
The proof is similar to that of Lemma \ref{l.emb.L2t}.
\end{proof}

By the very definition of the spaces, any derivative $\partial_x^\alpha $ maps
   $C^{s,\lambda,\delta}$ continuously into
   $C^{s-|\alpha|,\lambda,\delta+|\alpha|}$
if
   $s \in \mathbb{Z}_{\geq 0}$,
   $\lambda \in [0,1]$ and
   $|\alpha| \leq s$.
The following embedding theorem is expectable.

\begin{theorem}
\label{t.emb.hoelder}
Suppose that
   $s, s' \in \mathbb{Z}_{\geq 0}$,
   $\delta, \delta' \in \mathbb{R}_{\geq 0}$ and
   $\lambda, \lambda' \in [0,1]$.
If
   $\delta \geq \delta'$ and
   $s + \lambda \geq s' + \lambda'$
then the space
   $C^{s,\lambda,\delta}$ is embedded continuously into the space
   $C^{s',\lambda',\delta'}$.
Moreover, the embedding is compact if
   $\delta > \delta'$ and
   $s + \lambda > s' + \lambda'$
is fulfilled.
\end{theorem}

\begin{proof}
It is similar to the proof of Theorem \ref{t.emb.hoelder.t}.
\end{proof}

Let us introduce  the anisotropic H\"older spaces (see
   \cite{LadSoUr67},
   \cite{Sol64},
   \cite{Sol06},
   \cite{Be11}
and elsewhere).

As usual, we first set
\begin{eqnarray*}
   \| v \|_{C^{0,0} [0,T]}
 & =
 & \sup_{t \in [0,T]} |v (t)|,
\\
   \langle v \rangle_{\lambda,[0,T]}
 & =
 & \sup_{t', t'' \in [0,T] \atop t' \neq t''} \frac{|v (t') - v (t'')|}{|t' - t''|^\lambda}
\end{eqnarray*}
and
$$
   \| v \|_{C^{0,\lambda} [0,T]} = \| v \|_{C^{0,0} [0,T]} + \langle v \rangle_{\lambda,[0,T]}
$$
for functions defined on $[0,T]$.
For $s \in \mathbb{Z}_{\geq 0}$ and $\lambda \in [0,1]$, the space $C^{s,\lambda} ([0,T])$ is the
usual H\"older space on the segment $[0,T]$ with norm
$$
   \| v \|_{C^{s,\lambda} [0,T]}
 = \sum_{j=0}^s \| (d/dt)^j v \|_{C^{0,\lambda} [0,T]}.
$$
As is well known, this is a scale of Banach spaces.

More generally, given a Banach space $\mathcal{B}$, we denote by $C^{s,0} ([0,T], \mathcal{B})$ the
Banach space of all mappings $v : [0,T] \to \mathcal{B}$ with finite norm
$$
   \| v \|_{C^{s,0} ([0,T], \mathcal{B})}
 = \sum_{j=0}^s \sup_{t \in [0,T]} \| (d/dt)^j v \|_{\mathcal B},
$$
where $s \in \mathbb{Z}_{\geq 0}$.
We also let
$$
   \langle v \rangle_{\lambda, [0,T], \mathcal{B}}
 = \sup_{t', t'' \in [0,T] \atop t' \neq t''}
   \frac{\|  v (t') - v (t'') \|_{\mathcal{B}}}{|t' - t''|^\lambda}
$$
and let $C^{s,\lambda} ([0,T], \mathcal{B})$ stand for the space of all functions
   $v \in C^{s,0} ([0,T], \mathcal{B})$
with finite norm
$$
   \| v \|_{C^{s,\lambda} ([0,T], \mathcal{B})}
 = \sum_{j=0}^s
   \Big( \sup_{t \in [0,T]} \| (d/dt)^j v \|_{\mathcal B}
       + \langle (d/dt)^j v \rangle_{\lambda, [0,T], \mathcal{B}} \Big).
$$

Weighted H\"older spaces which control the behaviour of functions with respect to the time variable $t$
have been well known, see for instance \cite{Sol06} and elsewhere.
We go to introduce anisotropic H\"older spaces which suit well to parabolic theory and are weighted at
$x = \infty$.

The H\"{o}lder spaces in question will be parametrised several parameters $s$, $\lambda$, $\delta$,
$\mathcal{X}$ and $T$.
By abuse of notation we introduce the special designation $\mathbf{s} (s,\lambda,\delta)$ for the quintuple
$$
   \mathbf{s} (s,\lambda,\delta) := \Big( 2s,\lambda,s,\frac{\lambda}{2},\delta \Big).
$$

Let
$
   C^{\mathbf{s} (0,0,\delta)} (\overline{\mathcal{C}_T (\mathcal{X})})
 = C^{0,0} ([0,T], C^{0,0,\delta} (\overline{\mathcal{X}}))
$
be the space of all continuous functions on $\overline{\mathcal{C}_T (\mathcal{X})}$ with finite norm
$$
   \| u \|_{C^{\mathbf{s} (0,0,\delta)} (\overline{\mathcal{C}_T (\mathcal{X})})}
 = \sup_{(x,t) \in \overline{\mathcal{C}_T (\mathcal{X})}} (w (x))^\delta |u (x,t)|,
$$
and, for $0 < \lambda \leq 1$,
$$
   C^{\mathbf{s} (0,\lambda,\delta)} (\overline{\mathcal{C}_T (\mathcal{X})})
 = C^{0,0} ([0,T], C^{0,\lambda,\delta} (\overline{\mathcal{X}}))\, \cap\,
   C^{0,\lambda/2} ([0,T], C^{0,0,\delta} (\overline{\mathcal{X}}))
$$
is the space of all continuous functions on $\overline{\mathcal{C}_T (\mathcal{X})}$ with finite norm
\begin{eqnarray*}
\lefteqn{
   \| u \|_{C^{\mathbf{s} (0,\lambda,\delta)} (\overline{\mathcal{C}_T (\mathcal{X})})}
}
\\
 & = &
   \sup_{t \in [0,T]} \| u (\cdot, t) \|_{C^{0,\lambda,\delta} (\overline{\mathcal X})}
 + \sup_{t', t'' \in [0,T] \atop t' \neq t''}
   \frac{\|u (\cdot,t') - u (\cdot,t'') \|_{C^{0,0,\delta} (\overline{\mathcal X})}}{|t'-t''|^{\lambda/2}}.
\end{eqnarray*}
Then
$$
   C^{\mathbf{s} (s,0,\delta)} (\overline{\mathcal{C}_T (\mathcal{X})})
 = \bigcap_{j=0}^s C^{j,0} ([0,T], C^{2 (s-j),0,\delta} (\overline{\mathcal X}))
$$
is the space of functions on the cylinder $\overline{\mathcal{C}_T (\mathcal{X})}$ with continuous derivatives
   $\partial^{\alpha}_x \partial^j_t u$,
for $|\alpha| + 2j \leq 2s$, and with finite norm
$$
   \| u \|_{C^{\mathbf{s} (s,0,\delta)} (\overline{\mathcal{C}_T (\mathcal{X})})}
 = \sum_{|\alpha| + 2j \leq 2s}
   \| \partial^{\alpha}_x \partial^j_t u
   \|_{C^{\mathbf{s} (0,0,\delta+|\alpha|)} (\overline{\mathcal{C}_T (\mathcal{X})})}.
$$
Similarly,
$$
   C^{\mathbf{s} (s,\lambda,\delta)} (\overline{\mathcal{C}_T (\mathcal{X})})
 = \bigcap_{j=0}^s
   \! \Big( \!
         C^{j,0} ([0,T], C^{2 (s\!-\!j),\lambda,\delta} (\overline{\mathcal X})) \cap
         C^{j,\lambda/2} ([0,T], C^{2 (s\!-\!j),0,\delta} (\overline{\mathcal X}))
   \! \Big) \!
$$
is the space functions on $\overline{\mathcal{C}_T (\mathcal{X})}$ with continuous partial derivatives
   $\partial^{\alpha}_x \partial^j_t u$, for $|\alpha| + 2j \leq 2s$,
and with finite norm
$$
   \| u \|_{C^{\mathbf{s} (s,\lambda,\delta)} (\overline{\mathcal{C}_T (\mathcal{X})})}
 = \sum_{|\alpha| + 2j \leq 2s}
   \| \partial^{\alpha}_x \partial^j_t u
   \|_{C^{\mathbf{s} (0,\lambda,\delta+|\alpha|)} (\overline{\mathcal{C}_T (\mathcal{X})})}.
$$

We also need a function space whose structure goes slightly beyond the scale of function spaces
$C^{\mathbf{s} (s,\lambda,\delta)} (\overline{\mathcal{C}_T (\mathcal{X})})$.
Namely, given any integral $k \geq 0$, we denote by
   $C^{k, \mathbf{s} (s,\lambda,\delta)} (\overline{\mathcal{C}_T (\mathcal{X})})$
the space of all continuous functions $u$ on $\overline{\mathcal{C}_T (\mathcal{X})}$ whose derivatives
   $\partial^\beta_x u$
belong to $C^{\mathbf{s} (s,\lambda,\delta+|\beta|)} (\overline{\mathcal{C}_T (\mathcal{X})})$ for all multi-indices
$\beta$ satisfying $|\beta| \leq k$, with finite norm
$$
   \| u \|_{C^{k, \mathbf{s} (s,\lambda,\delta)} (\overline{\mathcal{C}_T (\mathcal{X})})}
 = \sum_{|\beta| \leq k}
   \| \partial^{\beta}_x u \|_{C^{\mathbf{s} (s,\lambda,\delta+|\beta|)} (\overline{\mathcal{C}_T (\mathcal{X})})}.
$$
For $k = 0$, this space just amounts to $C^{\mathbf{s} (s,\lambda,\delta+|\beta|)} (\overline{\mathcal{C}_T (\mathcal{X})})$,
and so we omit the index $k=0$.

If $\mathcal{X} = \mathbb{R}^n$, we will simply write it $C^{k, \mathbf{s} (s,\lambda,\delta)}$ if it causes no
confusion.
The normed spaces $C^{k, \mathbf{s} (s,\lambda,\delta)}$ are obviously Banach spaces.
Let us briefly discuss their basic properties.

\begin{lemma}
\label{l.emb.loc.t}
The Banach space $C^{k, \mathbf{s} (s,\lambda,\delta)}$ is embedded continuously into the space
   $C^{k, \mathbf{s} (s,\lambda,0)}_{\mathrm{loc}} (\overline{\mathcal{C}_T})$.
\end{lemma}

\begin{proof}
It is similar to the proof of Lemma \ref{l.emb.loc.t}.
\end{proof}

We note that the function classes introduced above can be thought of as ``physically'' admissible solutions
to the Navier-Stokes equations (at least for proper numbers $\delta$).

\begin{lemma}
\label{l.emb.L2t}
If $\delta > n/2$ then there exists a constant $c > 0$ depending on $\delta$, such that
$$
   \| u (\cdot,t) \|_{L^2 (\mathbb{R}^n)}
 \leq  c\, \| u \|_{C^{\mathbf{s} (0,0,\delta)} (\overline{\mathcal{C}_T})}
$$
for all $t \in [0,T]$ and all $u \in C^{\mathbf{s} (0,0,\delta)}$.
\end{lemma}

\begin{proof}
Indeed, on passing to the spherical coordinates we obtain
\begin{eqnarray*}
   \| u (\cdot,t) \|^2_{L^2 (\mathbb{R}^n)}
 & \leq &
   \| u \|^2_{C^{\mathbf{s} (0,0,\delta)} (\overline{\mathcal{C}_T})}
   \int_{\mathbb{R}^n} (1 + |x|^2)^{-\delta} dx
\\
 & = &
   \| u \|^2_{C^{\mathbf{s} (0,0,\delta)} (\overline{\mathcal{C}_T})}
   \int_{0}^{+\infty} \frac{\sigma_n \, r^{n-1}}{(1 + r^2)^{\delta}}\, dr,
\end{eqnarray*}
where $\sigma_n$ is the surface area of the unit sphere in $\mathbb{R}^n$.
Hence it follows that
$$
c^2
 = \int_0^{+\infty} \frac{\sigma_n\, r^{n-1}}{(1+r^2)^{\delta}}\, dr,
$$
for this integral converges if $\delta > n/2$.
\end{proof}

The following embedding theorem is rather expectable.

\begin{theorem}
\label{t.emb.hoelder.t}
Let
   $s, s' \in \mathbb{Z}_{\geq 0}$,
   $\delta, \delta' \in \mathbb{R}_{\geq 0}$,
   $\lambda, \lambda' \in [0,1]$
and $k$ a nonnegative integer.
If
   $s+\lambda \geq s'+\lambda'$ and
   $\delta \geq \delta'$,
then the space
   $C^{k, \mathbf{s} (s,\lambda,\delta)}$ is embedded continuously into
   $C^{k, \mathbf{s} (s',\lambda',\delta')}$.
The embedding is compact if $s + \lambda > s' + \lambda'$ and $\delta > \delta'$.
\end{theorem}

\begin{proof}
We begin with the following lemma.

\begin{lemma}
\label{l.smooth.lipschitz.t}
The space $C^{k, \mathbf{s} (s,0,\delta)} (\overline{\mathcal{C}_T})$ is embedded continuously  into the space
$C^{k, \mathbf{s} (s-1,1,\delta)} (\overline{\mathcal{C}_T})$.
\end{lemma}

\begin{proof}
If $u \in C^{k, \mathbf{s} (s,0,\delta)}$ then by the mean value theorem of Lagrange there is $\vartheta \in (0,1)$
such that
\begin{eqnarray*}
   \frac{|\partial^\beta_x \partial_t^j u (x,t) - \partial^\beta_x \partial_t^j u (y,t)|}{|x-y|}
 & = &
   \frac{\displaystyle
         \Big| \sum_{i=1}^n \partial^{\beta + e_i}_x \partial_t^j u (x_\vartheta,t) (y^i - x^i) \Big|}{|x-y|}
\\
 & \leq &
   \Big( \sum_{i=1}^n |\partial^{\beta+e_i}_x \partial_t^j u (x_\vartheta,t)|^2 \Big)^{1/2}
\\
 & = &
   \frac{\displaystyle
         \Big( \sum_{i=1}^n (w (x_\vartheta))^{2 (\delta+|\beta|+1)}
               |\partial^{\beta+e_i}_x \partial_t^j  u (x_\vartheta,t)|^2
         \Big)^{1/2}}
         {(w (x_\vartheta))^{\delta+|\beta|+1}}
\end{eqnarray*}
for all admissible $\beta$ and $j$, where
   $e_i$ is the basis vector of the axis $x^i$ in $\mathbb{R}^n$ and
   $x_\vartheta = x + \vartheta (y-x)$.

Suppose $|x-y| \leq |x|/2$.
Then, by the triangle inequality, we get
\begin{equation}
\label{eq.z.half}
   |x|/2
 \leq
   | |x| - \vartheta |y-x| |
 \leq
   |x_\vartheta|
 \leq
   |x| + \vartheta |y-x|
 \leq (3/2)\, |x|
\end{equation}
whence
$
   |x-y| \leq w (x_\vartheta)
$
if $|x-y| \leq |x|/2$.
Moreover, as $||x| - |y|| \leq |x-y|$, it follows that
\begin{equation}
\label{eq.x.half}
   |x|/2 \leq |y| \leq (3/2)\, |x|,
\end{equation}
if $|x-y| \leq |x|/2$.
Now (\ref{eq.z.half}) and (\ref{eq.x.half}) imply that
\begin{equation}
\label{eq.half}
   \frac{2}{3}\, w (x_\vartheta)
 \leq w (x)
 \leq w (x,y)
 \leq \frac{\sqrt{13}}{2}\, w (x)
 \leq \sqrt{13}\, w (x_\vartheta)
\end{equation}
if $|x-y| \leq |x|/2$.

It follows from (\ref{eq.half}) that
\begin{eqnarray*}
\lefteqn{
   \sup_{t \in [0,T]}
   \sup_{|x-y| \leq |x|/2 \atop x \neq y}
   (w (x,y))^{\delta+|\beta|+1}\,
   \frac{| \partial^\beta_x \partial_t^j u (x,t) - \partial^\beta_x \partial_t^j u (y,t)|}{|x-y|}
}
\\
 & \leq &
   c\,
   \Big( \sum_{i=1}^n
         \| \partial^{\beta+e_i}_x \partial^j_t u
         \|^2_{C^{\mathbf{s} (0,0,\delta+|\beta|+1)} (\overline{\mathcal{C}_T})}
   \Big)^{1/2},
\end{eqnarray*}
with $c$ being a positive constant independent of $u \in C^{k, \mathbf{s} (s,0,\delta)} (\overline{\mathcal{C}_T})$.

Arguing in the same way, for $u \in C^{k, \mathbf{s} (s,0,\delta)} (\overline{\mathcal{C}_T})$, we conclude by
Lagrange's mean value theorem that there is $\vartheta \in (0,1)$ with the property that
\begin{eqnarray*}
\lefteqn{
   \sup_{x \in \mathbb{R}^n}
   \sup_{t', t'' \in [0,T] \atop t' \neq t''}
   (w (x))^{\delta+|\beta|}\,
   \frac{| \partial^\beta_x \partial_t^j u (x,t') - \partial^\beta_x \partial_t^j u (x,t'')|}
        {|t'-t''|^{1/2}}
}
\\
 & = &
   \sup_{x \in {\mathbb R}^n}
   \sup_{t', t'' \in [0,T] \atop t' \neq t''}
   (w (x))^{\delta+|\beta|}\,
   \frac{| \partial^\beta_x \partial^{j+1}_{t} u (x,t_\vartheta) (t' - t'')|}{|t' - t''|^{1/2}}
\\
 & = &
   \sup_{x \in {\mathbb R}^n}
   (w (x))^{\delta+|\beta|}
   \sup_{t', t'' \in [0,T] \atop t' \neq t''}
   |t' - t''|^{1/2}
   |\partial^\beta_x \partial_{t}^{j+1} u (x,t_\vartheta)|
\\
 & \leq &
   \sqrt{T}\,
   \| \partial^\beta_x \partial^{j+1}_t u \|_{C^{\mathbf{s} (0,0,\delta+|\beta|)} (\overline{\mathcal{C}_T})}
\end{eqnarray*}
for all admissible $\beta$ and $j$, where $t_{\vartheta} = t' + \vartheta (t''-t')$.

As is well known, the space
   $C^{k, \mathbf{s} (s,0,0)} (\overline{\mathcal{C}_T (B_1)})$
is embedded continuously into the space
   $C^{k, \mathbf{s} (s-1,1,0)} (\overline{\mathcal{C}_T (B_1)})$,
if $s$ is a nonnegative integer.
Thus, $C^{k, \mathbf{s} (s,0,\delta)} (\overline{\mathcal{C}_T})$ is embedded continuously into
      $C^{k, \mathbf{s} (s-1,1,\delta)} (\overline{\mathcal{C}_T})$,
as desired.
\end{proof}

Now we note that
\begin{eqnarray*}
\lefteqn{
   (w (x,y))^{\delta' + \lambda' + |\beta|}
   \frac{|\partial_x^\beta \partial_t^j u (x,t) - \partial_y^\beta \partial_t^j u (y,t)|}
        {|x-y|^{\lambda'}}
}
\\
 & = &
   (w (x,y))^{\delta + \lambda + |\beta|}
   \frac{|\partial_x^\beta \partial_t^j u (x,t) - \partial_y^\beta \partial_t^j u (y,t)|}
        {|x-y|^{\lambda}}\,
   \frac{1}{(w (x,y))^{\delta - \delta'}}
   \Big( \frac{|x -y|}{w (x,y)} \Big)^{\lambda - \lambda'}
\end{eqnarray*}
whence
$$
   \langle \partial^\beta \partial_t^j u (\cdot,t) \rangle_{\lambda',\delta',\mathbb{R}^n}
 \leq
   2^{\lambda'-\lambda}\,
   \langle \partial^\beta \partial_t^j u (\cdot,t) \rangle_{\lambda,\delta,\mathbb{R}^n},
$$
if
    $0 < \lambda' \leq \lambda \leq 1$ and
    $\delta' \leq \delta$.
Besides,
\begin{eqnarray*}
\lefteqn{
   (w (x))^{\delta' + |\beta|}
   \frac{|\partial^\beta_x \partial^j_t u (x,t') - \partial^\beta_x \partial^j_t u (x,t'')|}
        {|t' - t''|^{\lambda'/2}}
}
\\
 & = &
   (w (x))^{\delta + |\beta|}
   \frac{|\partial_x^\beta \partial_t^j u (x,t') - \partial_x^\beta \partial_t^j u (x,t'')|}
        {|t' - t''|^{\lambda/2}}\,
   \frac{1}{(w (x))^{\delta - \delta'}}
   |t' - t''|^{(\lambda-\lambda')/2},
\end{eqnarray*}
and so, for
   $0 < \lambda' \leq \lambda \leq 1$ and
   $\delta' \leq \delta$,
we obtain
$$
   (w (x))^{\delta' + |\beta|}
   \langle \partial_x^\beta \partial^j u (x,\cdot) \rangle_{\lambda'/2,[0,T]}
 \leq
   (2T)^{\frac{\scriptstyle \lambda-\lambda'}{\scriptstyle 2}}\,
   (w (x))^{\delta' + |\beta|} \langle \partial_x^\beta \partial^j u (x,\cdot) \rangle_{\lambda/2,[0,T]}.
$$

Hence, as
   $C^{k,\mathbf{s} (s,\lambda,0)} (\overline{\mathcal{C}_T (B_1)})$ is embedded continuously into
   $C^{k,\mathbf{s} (s,\lambda',0)} (\overline{\mathcal{C}_T (B_1)})$
for $\lambda \geq \lambda' > 0$, we see that
   $C^{k,\mathbf{s} (s,\lambda,\delta)} (\overline{\mathcal{C}_T})$ is embedded continuously into the space
   $C^{k,\mathbf{s} (s,\lambda',\delta')} (\overline{\mathcal{C}_T})$
provided
   $\lambda \geq \lambda' > 0$ and
   $\delta \geq \delta'$.
Now applying Lemma \ref{l.smooth.lipschitz.t} yields readily
$$
   C^{k,\mathbf{s} (s,\lambda,\delta)} (\overline{\mathcal{C}_T})
 \hookrightarrow
   C^{k,\mathbf{s} (s',\lambda',\delta')} (\overline{\mathcal{C}_T})
$$
if
   $s + \lambda \geq s' + \lambda'$ and
   $\delta \geq \delta'$.

To study the compactness of the embedding we note that
   $C^{\mathbf{s} (0,\lambda,0)} (\overline{B_1} \! \times \! [0,T])$ is embedded compactly into
   $C^{\mathbf{s} (0,\lambda',0)}  (\overline{\mathcal{C}_T (B_1)})$
for $\lambda > \lambda' > 0$.
Use the standard one point compactification of $\mathbb{R}^n$.
Namely, we embed $\overline{\mathcal{C}_T} $ into the compact cylinder $\mathcal{C}$ in $\mathbb{R}^{n+2}$
consisting of all $(z^0, z^1, \ldots, z^n, z^{n+1}) \in \mathbb{R} \times \mathbb{R}^{n} \times [0,T]$,
such that
$$
   (z^0)^2 + \sum_{j=1}^{n} (z^j)^2 = 1.
$$
To this end we consider the map
$$
   \iota : \hat{\mathbb{R}}^n \times [0,T] \to \mathbb{R}^{n+2}
$$
given by
\begin{equation}
\label{eq.M}
   \iota (x,t)
 = \left\{ \begin{array}{llll}
             \displaystyle
             \Big( \frac{|x|^2-1}{(w (x))^2}, \frac{2x}{(w (x))^{2}}, t \Big),
           & \mbox{if}
           & x \neq \infty,
 \\
             (1, 0, t),
           & \mbox{if}
           & x = \infty.
           \end{array}
           \right.
\end{equation}
Indeed,
$$
   \Big( \frac{|x|^2-1}{(w (x))^2} \Big)^2 + \Big| \frac{2x}{(w (x))^2} \Big|^2
 = \frac{|x|^4 - 2 |x|^2 + 1 + 4 |x|^2 }{(w (x))^4}
 = 1,
$$
and then, as the points at infinity of $\hat{\mathbb{R}}^n \times [0,T]$ correspond to the point
$(1, 0, t)$  of cylinder $\mathcal{C}$, the inverse map $\iota^{-1}$ is given by
$$
   \iota^{-1} (z)
 = \left\{ \begin{array}{lll}
             \displaystyle
             \Big( \frac{(z^1, \ldots, z^n)}{1 - z^0}, z^{n+1} \Big),
           & \mbox{if}
           & - 1 \leq z^0 < 1,
\\
             (\infty, z^{n+1}),
           & \mbox{if}
           & z = (1, 0, z^{n+1}).
\end{array}
\right.
$$

The map $\iota$ is obviously continuous on $\hat{\mathbb{R}}^n \times [0,T]$ and at least $C^1$
smooth and nonsingular on $\overline{\mathcal{C}_T} $, for
$$
\begin{array}{rclclrclcl}
   \partial_i \iota_0
 & \!\!\! = \!\!\!
 & \displaystyle
   \frac{4 x^i}{(w (x))^4},
 & \mbox{for}
 & 1 \leq i \leq n;
 & \partial_t \iota_0
 & \!\!\! = \!\!\!
 & 0;
 &
 &
\\
   \partial_i \iota_j
 & \!\!\! = \!\!\!
 & \displaystyle
   \frac{2 \delta_{ij}}{(w (x))^2} - \frac{4 x^i\, x^j}{(w (x))^4},
 & \mbox{for}
 & 1 \leq i, j \leq n;
 & \partial_t \iota_j
 & \!\!\! = \!\!\!
 & 0,
 & \mbox{for}
 & 1 \leq j \leq n;
\\
   \partial_i \iota_{n+1}
 & \!\!\! = \!\!\!
 & 0,
 & \mbox{for}
 & 1 \leq i \leq n;
 & \partial_t \iota_{n+1}
 & \!\!\! = \!\!\!
 & 1.
 &
 &
\end{array}
$$
The function
$$
   d ((x,t'),(y,t'')) = |\iota (x,t') - \iota (y,t'')|
$$
is obviously a metric on the set $\hat{\mathbb{R}}^n \times [0,T] \cong \mathcal{C}$.
We are now in a position to formulate a compactness criterion a l\`{a} Ascoli-Arzel\'{a} theorem.

\begin{lemma}
\label{l.Arzela.t}
Let $S$ be a subset of
   $C^{k,\mathbf{s} (s,0,\delta)} (\overline{\mathcal{C}_T})$
bearing the following properties:

1)
$S$ is bounded in $C^{k,\mathbf{s} (s,0,\delta)} (\overline{\mathcal{C}_T})$;

2)
for any $\varepsilon > 0$ there is $\delta (\varepsilon) > 0$ such that, for all
   $(x,t'), (y,t'') \in \overline{\mathcal{C}_T} $ with $d ((x,t'),(y,t'')) < \delta (\varepsilon)$
and for all $u \in S$, we get
$$
   | (w (x))^{\delta' + |\alpha + \beta|} \partial_x^{\alpha + \beta} \partial_t^j u (x,t')
   - (w (y))^{\delta' + |\alpha + \beta|} \partial_y^{\alpha + \beta} \partial_t^j u (y,t'') |
 < \varepsilon
$$
if
   $|\alpha| + 2j \leq 2s$ and
   $|\beta| \leq k$;

3)
for any $\varepsilon > 0$ there is $\delta (\varepsilon) > 0$ such that, for all
   $(x,t'), (y,t'') \in \overline{\mathcal{C}_T (B_1)}$ with $\sqrt{|x-y|^2 + |t'-t''|} < \delta (\varepsilon)$
and for all $u \in S$, we have
$$
   | \partial_x^{\alpha + \beta} \partial_t^j u (x,t')
   - \partial_y^{\alpha + \beta} \partial_t^j u (y,t'') |
 < \varepsilon
$$
if
   $|\alpha| + 2j \leq 2s$ and
   $|\beta| \leq k$.

Then the set $S$ is precompact in the weighted space
   $C^{k,\mathbf{s} (s,0,\delta')} (\overline{\mathcal{C}_T})$
for any $\delta' < \delta$.
\end{lemma}

\begin{proof}
Fix an arbitrary $\delta' < \delta$.
If $S$ is a bounded set in $C^{k,\mathbf{s} (s,0,\delta)} (\overline{\mathcal{C}_T})$ then, for $u \in S$,
the functions
\begin{eqnarray*}
\lefteqn{
   u^{(\alpha+\beta,j)} (z)
}
\\
 & = &
   \left\{
   \begin{array}{lll}
     \left( (w (x))^{\delta' + |\alpha + \beta|} \partial_x^{\alpha+\beta} \partial_t^j u \right)
     (\iota^{-1} (z)),
   & \mbox{if}
   & z \in \mathcal{C} \setminus (1, 0, z^{n+1}),
\\
     0,
   & \mbox{if}
   & z = (1, 0, z^{n+1}),
   \end{array}
   \right.
\end{eqnarray*}
are continuous on $\mathcal{C}$ for $\delta > \delta'$ because
$$
   |u^{(\alpha+\beta,j)} (z)|
 \leq
   \| \partial_x^{\alpha+\beta} \partial_t^j u
   \|_{C^{\mathbf{s} (0,0,\delta + |\alpha|)} (\overline{\mathcal{C}_T})}
   (w (\iota^{-1} (z))^{\delta' - \delta}
$$
for all $z \in \mathcal{C} \setminus (1, 0, z^{n+1})$, and
$$
   \iota^{-1} (z) \to (\infty, t)
$$
if $z \to (1,0,t)$ with some $t \in [0,T]$.
In particular, the set $\{ u^{(\alpha+\beta,j)} \}_{u \in S}$  satisfies the hypotheses of the
Ascoli-Arzel\'a theorem, and so it is precompact in $C (\mathcal{C})$.
In particular, this means that the set $S$ is precompact in
   $C^{k,\mathbf{s} (s,0,\delta')} (\overline{\mathcal{C}_T})$,
as desired.
\end{proof}

If $S$ is a bounded set in the space $C^{\mathbf{s} (0,\lambda,\delta)} (\overline{\mathcal{C}_T})$,
then it is obviously bounded in the spaces
   $C^{\mathbf{s} (0,0,\delta)} (\overline{\mathcal{C}_T})$ and
   $C^{\mathbf{s} (0,\lambda,0)} (\overline{\mathcal{C}_T (B_1)})$,
too.
For any $0 \leq \lambda' < \lambda \leq 1$, this set is precompact in
   $C^{\mathbf{s} (0,\lambda',0)} (\overline{\mathcal{C}_T (B_1)})$.
Moreover, we have
\begin{eqnarray*}
   (w (x,y))^{\delta}\, | u (x,t) - u (y,t) |
 & \leq &
   \| u \|_{C^{\mathbf{s} (0,\lambda,\delta)} (\overline{\mathcal{C}_T})}
   \Big( \frac{|x - y|}{w (x,y)} \Big)^{\lambda}
\\
 & \leq &
   2^{\lambda/2}\, \| u \|_{C^{\mathbf{s} (0,\lambda,\delta)} (\overline{\mathcal{C}_T})}
\end{eqnarray*}
for all
   $x, y \in \mathbb{R}^n$ and
   $t \in [0,T]$,
and
\begin{eqnarray*}
   (w (y))^{\delta}\, | u (y,t') - u (y,t'') |
 & \leq &
    \| u \|_{C^{\mathbf{s} (0,\lambda,\delta)} (\overline{\mathcal{C}_T})}
    |t' - t''|^{\lambda/2}
\\
 & \leq &
    T^{\lambda/2}\, \| u \|_{C^{\mathbf{s} (0,\lambda,\delta)} (\overline{\mathcal{C}_T})}
\end{eqnarray*}
for all $y \in \mathbb{R}^n$ and $t', t'' \in [0,T]$.

If $|x| \leq |y|$ then
\begin{eqnarray*}
   | (w (x))^{\delta'} u (x,t') - (w (y))^{\delta'} u (y,t'') |
 & \leq &
   (w (x))^{\delta'} | u (x,t') - u (y,t') |
\\
 & + &
   | (w (x))^{\delta'} - (w (y))^{\delta'} | | u (y,t') |
\\
 & + &
   (w (y))^{\delta'} | u (y,t') - u (y,t'') |,
\end{eqnarray*}
which is dominated by
\begin{eqnarray*}
 & \leq &
   \frac{1}{(w (x,y))^{\delta-\delta'}}\,
   (w (x,y))^{\delta} | u (x,t') - u (y,t') |
\\
 & + &
   \frac{(\sqrt{2})^{\delta}}{(w (x,y))^{\delta-\delta'}}
   \frac{| (w (x))^{\delta'} - (w (y))^{\delta'} |}{(w (x,y))^{\delta'}} (w (y))^{\delta} |u (y,t')|
\\
 & + &
   \Big( \frac{\sqrt{2}}{w (x,y)} \Big)^{\delta-\delta'} (w (y))^{\delta} | u (y,t') - u (y,t'')|,
\end{eqnarray*}
for in this case
   $w (x) \leq  w (x,y) \leq \sqrt{2} w (y)$
and $\delta \geq \delta' \geq 0$.
Hence, for all $x,y \in \mathbb{R}^n$ with $|x| \leq |y|$ and $t', t'' \in [0,T]$ we have
\begin{eqnarray}
\label{eq.emb.comp1.t}
\lefteqn{
   | (w (x))^{\delta'} u (x,t') - (w (y))^{\delta'} u (y,t'') |
}
\nonumber
\\
 & \leq &
   c\,
   \frac{\| u \|_{C^{\mathbf{s} (0,\lambda,\delta)} (\overline{\mathcal{C}_T})}}{(w (x,y))^{\delta-\delta'}}\,
   \Big( \Big( \frac{|x-y|}{w (x,y)} \Big)^{\lambda} + |t'-t''|^{\lambda/2} \Big)
\nonumber
\\
 & + &
   c\,
   \frac{\| u \|_{C^{\mathbf{s} (0,0,\delta)} (\overline{\mathcal{C}_T})}}{(w (x,y))^{\delta-\delta'}}\,
   \frac{| (w (x))^{\delta'} - (w (y))^{\delta'} |}{(w (x,y))^{\delta'}}
\nonumber
\\
\end{eqnarray}
with a positive constant $c$ depending on $T$ but not on $x$, $y$, $t'$, $t''$ and $u$.

Fix $\varepsilon > 0$.
As $\delta > \delta' \geq 0$ and
$$
   \frac{|x-y|}{w (x,y)} \leq \sqrt{2},\ \ \
   |t'-t''| \leq T,\ \ \
   \frac{| (w (x))^{\delta'} - (w (y))^{\delta'} |}{(w (x,y))^{\delta'}} \leq 2,
$$
we conclude that there is $R > 0$ such that
\begin{eqnarray*}
   \frac{\| u \|_{C^{\mathbf{s} (0,\lambda,\delta)} (\overline{\mathcal{C}_T})}}{(w (x,y))^{\delta-\delta'}}\,
   \Big( \Big( \frac{|x-y|}{w (x,y)} \Big)^{\lambda} + |t'-t''|^{\lambda/2} \Big)
 & < &
   \frac{\varepsilon}{4},
\\
   \frac{\| u \|_{C^{\mathbf{s} (0,0,\delta)} (\overline{\mathcal{C}_T})}}{(w (x,y))^{\delta-\delta'}}\,
   \frac{| (w (x))^{\delta'} - (w (y))^{\delta'} |}{(w (x,y))^{\delta'}}
 & < &
   \frac{\varepsilon}{4}
\end{eqnarray*}
for all $x$ and $y$ satisfying $|x|^2 + |y|^2 > R^2$, and for all $t', t'' \in [0,T]$.
Since the function $w (x)$ and the map $\iota (x,t)$ are continuous on $\overline{\mathcal{C}_T} $,
they are equicontinuous on the cylinder
   $\overline{\mathcal{C}_T (B_R)}$,
where $B_R$ stands for the ball of radius $R$ around the origin in $\mathbb{R}^n$.
On this cylinder the metric $d ((x,t'),(y,t''))$ defines the same topology as the standard metric
   $(|x-y|^2 +|t'-t''|^2)^{1/2}$
as well as the metric $(|x-y|^2 + |t'-t''|)^{1/2}$.
Hence, there is a positive number $\delta (\varepsilon)$ depending on $\varepsilon$, such that
\begin{eqnarray*}
   \frac{\| u \|_{C^{\mathbf{s} (0,\lambda,\delta)} (\overline{\mathcal{C}_T})}}{(w (x,y))^{\delta-\delta'}}\,
   \Big( \Big( \frac{|x-y|}{w (x,y)} \Big)^{\lambda} + |t'-t''|^{\lambda/2} \Big)
 & < &
   \frac{\varepsilon}{4},
\\
   \frac{\| u \|_{C^{\mathbf{s} (0,0,\delta)} (\overline{\mathcal{C}_T})}}{(w (x,y))^{\delta-\delta'}}\,
   \frac{| (w (x))^{\delta'} - (w (y))^{\delta'} |}{(w (x,y))^{\delta'}}
 & < &
   \frac{\varepsilon}{4}
\end{eqnarray*}
for all $(x,t'), (y,t'') \in \overline{\mathcal{C}_T (B_R)}$ satisfying
   $d ((x,t'), (y,t'')) < \delta (\varepsilon)$.

As for any $x, y \in \mathbb{R}^n$ one has either $|x| \leq |y|$ or $|y| \leq |x|$, on evaluating
the difference
   $| (w (x))^{\delta'} u (x,t') - (w (y))^{\delta'} u (y,t'') |$
we may confine ourselves to the case $|x| \leq |y|$.
In particular, (\ref{eq.emb.comp1.t}) means that $S$ satisfies the hypotheses of Lemma \ref{l.Arzela.t},
and so $S$ is precompact in $C^{\mathbf{s} (0,0,\delta')} (\overline{\mathcal{C}_T})$.
Hence, (any sequence in) $S$ contains a subsequence $\{ u_\nu \}$ convergent in
   $C^{\mathbf{s} (0,0,\delta')} (\overline{\mathcal{C}_T})$.
Without loss of the generality we may assume that it converges to zero in this space.
On the other hand,
\begin{eqnarray*}
\lefteqn{
   (w (x,y))^{\delta'+\lambda'} \frac{| u (x,t) - u (y,t) |}{|x-y|^{\lambda'}}
}
\\
 & \leq &
   \Big( (w (x,y))^{\delta+\lambda} \frac{|  u (x,t) -  u (y,t)|}{|x-y|^{\lambda}}
   \Big)^{\frac{\scriptstyle \lambda'}{\scriptstyle \lambda}}
   \frac{( (w (x,y))^{\delta'} |u (x,t) - u (y,t)| )^{1-\frac{\scriptstyle \lambda'}{\scriptstyle \lambda}}}
        { (w (x,y))^{(\delta-\delta') \frac{\scriptstyle \lambda'}{\scriptstyle \lambda}}}
\end{eqnarray*}
and
\begin{eqnarray*}
\lefteqn{
   (w (x))^{\delta'}
   \frac{| u (x,t') -  u (x,t'') |}{|t'-t''|^{\frac{\scriptstyle \lambda'}{\scriptstyle 2}}}
}
\\
 & \leq &
   \Big( (w (x))^{\delta}
         \frac{| u (x,t') -  u (x,t'')|}{|t'-t''|^{\frac{\scriptstyle \lambda}{\scriptstyle 2}}}
   \Big)^{\frac{\scriptstyle \lambda'}{\scriptstyle \lambda}}
   \frac{( (w (x))^{\delta'} |u (x,t') - u (x,t'')| )^{1-\frac{\scriptstyle \lambda'}{\scriptstyle \lambda}}}
        { (w (x))^{(\delta-\delta') \frac{\scriptstyle \lambda'}{\scriptstyle \lambda}}}.
\end{eqnarray*}
Therefore, by (\ref{eq.half}),
\begin{eqnarray*}
   \sup_{t \in [0,T]}
   \langle u_\nu \rangle_{\lambda',\delta',\mathbb{R}^n}
 & \leq &
   C\,
   \sup_{t \in [0,T]}
   \langle  u_\nu \rangle_{\lambda,\delta,\mathbb{R}^n}^{\lambda'/\lambda}\,
   \| u_\nu \|_{C^{\mathbf{s} (0,0,\delta')} (\overline{\mathcal{C}_T})}^{1-\lambda'/\lambda},
\\
   \langle u_\nu \rangle_{\lambda'/2,[0,T],C^{0,0,\delta'} (\mathbb{R}^n)}
 & \leq &
   C\,
   \langle  u_\nu \rangle_{\lambda/2,[0,T],C^{0,0,\delta} (\mathbb{R}^n)}^{\lambda'/\lambda}\,
   \| u_\nu \|_{C^{\mathbf{s} (0,0,\delta')} (\overline{\mathcal{C}_T})}^{1-\lambda'/\lambda}
\end{eqnarray*}
with $C > 0$ a constant independent of $\nu$, and so the sequence $\{ u_\nu \}$ converges to zero in
$C^{\mathbf{s} (0,\lambda',\delta')} (\overline{\mathcal{C}_T})$, too.
Hence it follows that $S$ is precompact in $C^{\mathbf{s} (0,\lambda',\delta')} (\overline{\mathcal{C}_T})$.
Thus,
   $C^{\mathbf{s} (0,\lambda,\delta)} (\overline{\mathcal{C}_T})$
is embedded compactly into
   $C^{\mathbf{s} (0,\lambda',\delta')} (\overline{\mathcal{C}_T})$
if $\lambda > \lambda'$ and $\delta > \delta'$.
By induction we conclude immediately that the embedding
$$
   C^{k, \mathbf{s} (s,\lambda,\delta)} (\overline{\mathcal{C}_T})
 \hookrightarrow
   C^{k, \mathbf{s} (s,\lambda',\delta')} (\overline{\mathcal{C}_T})
$$
is compact for any integral numbers $k, s \geq 0$, provided that
   $\lambda > \lambda'$ and
   $\delta > \delta'$.

Finally, on applying Lemma \ref{l.smooth.lipschitz.t} we see that
   $C^{k, \mathbf{s} (s,\lambda,\delta)} (\overline{\mathcal{C}_T})$
is embedded compactly into
   $C^{k, \mathbf{s} (s',\lambda',\delta')} (\overline{\mathcal{C}_T})$,
if
   $s + \lambda > s' + \lambda'$ and
   $\delta > \delta'$.
The proof is complete.
\end{proof}

We also need a standard lemma on the multiplication of functions.

\begin{lemma}
\label{l.product}
Let
   $s$, $k$ be nonnegative integers
and $\lambda \in [0,1]$.
If
   $u \in C^{k, \mathbf{s} (s,\lambda,\delta)} (\overline{\mathcal{C}_T})$ and
   $v \in C^{k, \mathbf{s} (s,\lambda,\delta')} (\overline{\mathcal{C}_T})$,
then the product $u v$ belongs to
   $C^{k, \mathbf{s} (s,\lambda,\delta+\delta')} (\overline{\mathcal{C}_T})$
and
\begin{equation}
\label{eq.product}
   \| u v \|_{C^{k, \mathbf{s} (s,\lambda,\delta+\delta')} (\overline{\mathcal{C}_T})}
 \leq
   c\,
   \| u \|_{C^{k, \mathbf{s} (s,\lambda,\delta)} (\overline{\mathcal{C}_T})}
   \| v \|_{C^{k, \mathbf{s} (s,\lambda,\delta')} (\overline{\mathcal{C}_T})}
\end{equation}
with $c > 0$ a constant independent of $u$ and $v$.
\end{lemma}

\begin{proof}
Indeed, for
   $|\alpha| + 2j \leq 2s$ and
   $|\beta| \leq k$,
we get
\begin{equation}
\label{eq.product.diff}
   \partial_x^{\alpha+\beta} \partial_t^j (u v)
 =
   \sum_{{\alpha' \leq \alpha \atop \beta' \leq \beta} \atop j' \leq j}
   \Big( \substack{\alpha \\ \alpha'} \Big)
   \Big( \substack{\beta \\ \beta'} \Big)
   \Big( \substack{j \\ j'} \Big)
   \partial_x^{\alpha'+\beta'} \partial_t^{j'} u\,
   \partial_x^{\alpha-\alpha'+\beta-\beta'} \partial_t^{j-j'} v
\end{equation}
whence
$$
   \frac{|\partial_x^{\alpha+\beta} \partial_t^j (u v)|}
        {(w (x))^{-\delta-\delta'-|\alpha+\beta|}}
 \leq
   c
   \sum_{{\alpha' \leq \alpha \atop \beta' \leq \beta} \atop j' \leq j}
   \frac{|\partial_x^{\alpha'+\beta'} \partial_t^{j'} u|}
        {(w (x))^{-\delta - |\alpha'+\beta'|}}\,
   \frac{|\partial_x^{\alpha''+\beta''} \partial_t^{j''} v|}
        {(w (x))^{-\delta' - |\alpha''+\beta''|}},
$$
where
   $\alpha'' = \alpha - \alpha'$,
   $\beta'' = \beta - \beta'$ and
   $j'' = j - j'$,
the constant $c$ depending only on $s$ and $k$.
Hence it follows that (\ref{eq.product}) is fulfilled for $\lambda = 0$ because
$
   \displaystyle
   \int f g \leq \int f \int g
$
if $f \geq 0$ and $g \geq 0$.

We now assume that
   $\lambda \in (0,1]$ and
   $|x-y| \leq |x|/2$.
Then using (\ref{eq.half}) and (\ref{eq.product.diff}) we deduce readily that
\begin{eqnarray*}
\lefteqn{
   \frac{1}{(w (x,y))^{- \delta - \delta' - |\alpha+\beta| - \lambda}}\,
   \frac{|\partial_x^{\alpha+\beta} \partial_t^j (u v) (x,t)
        - \partial_y^{\alpha+\beta} \partial_t^j (u v) (y,t)|}
        {|x-y|^{\lambda}}
}
\\
 & \leq &
   c
   \sum_{{\alpha' \leq \alpha \atop \beta' \leq \beta} \atop j' \leq j}
   \frac{|\partial_x^{\alpha'+\beta'} \partial_t^{j'} u (x,t)|}
        {(w (x))^{- \delta - |\alpha'+\beta'|)}}
   \frac{|\partial_x^{\alpha''+\beta''} \partial_t^{j''} v (x,t)
        - \partial_y^{\alpha''+\beta''} \partial_t^{j''} v (y,t)|}
        {(w (x,y))^{- \delta' - |\alpha''+\beta''| - \lambda} |x-y|^{\lambda}}
\\
 & + &
   c
   \sum_{{\alpha' \leq \alpha \atop \beta' \leq \beta} \atop j' \leq j}
   \frac{|\partial_x^{\alpha'+\beta'} \partial_t^{j'} u (x,t)
        - \partial_y^{\alpha'+\beta'} \partial_t^{j'} u (y,t)|}
        {(w (x,y))^{- \delta - |\alpha'+\beta'| - \lambda} |x-y|^{\lambda}}
   \frac{|\partial_y^{\alpha''+\beta''} \partial_t^{j''} v (y,t)|}
        {(w (y))^{- \delta' - |\alpha''+\beta''|)}}
\end{eqnarray*}
for
   $|\alpha| + 2j \leq 2s$ and
   $|\beta| \leq k$,
the constant $c$ depends neither on $x$ and $t$ nor on $u$ and $v$.
Therefore,
\begin{eqnarray*}
\lefteqn{
   \sup_{t \in [0,T]}
   \langle \partial^{\alpha+\beta} \partial_t^j (u v) (\cdot, t)
   \rangle_{\lambda,\delta+\delta'+|\alpha+\beta|,\mathbb{R}^n}
}
\\
 & \!\!\! \leq \!\!\! &
   c
   \sum_{{\alpha' \leq \alpha \atop \beta' \leq \beta} \atop j' \leq j}
   \| \partial^{\alpha'+\beta'}_x \partial_t^{j'} u
   \|_{C^{\mathbf{s} (0,0,\delta+|\alpha'+\beta'|)} (\overline{\mathcal{C}_T})}\,
   \sup_{t \in [0,T]}
   \langle \partial_x^{\alpha''+\beta''} \partial_t^{j''} v
   \rangle_{\lambda, \delta'+|\alpha''+\beta''|,\mathbb{R}^n}
\\
 & \!\!\! + \!\!\! &
   c
   \sum_{{\alpha' \leq \alpha \atop \beta' \leq \beta} \atop j' \leq j}
   \sup_{t \in [0,T]}
   \langle \partial_x^{\alpha'+\beta'} \partial_t^{j'} u
   \rangle_{\lambda, \delta+|\alpha'+\beta'|,\mathbb{R}^n}\,
   \| \partial^{\alpha''+\beta''}_x \partial_t^{j''} v
   \|_{C^{\mathbf{s} (0,0,\delta'+|\alpha''+\beta''|)} (\overline{\mathcal{C}_T})}.
\end{eqnarray*}

Similarly,
\begin{eqnarray*}
\lefteqn{
   \frac{1}{(w (x))^{- \delta - \delta' - |\alpha+\beta|}}\,
   \frac{|\partial_x^{\alpha+\beta} \partial_t^j (u v) (x,t')
        - \partial_x^{\alpha+\beta} \partial_t^j (u v) (x,t'')|}
        {|t'-t''|^{\lambda/2}}
}
\\
 & \leq &
   c
   \sum_{{\alpha' \leq \alpha \atop \beta' \leq \beta} \atop j' \leq j}
   \frac{|\partial_x^{\alpha'+\beta'} \partial_t^{j'} u (x,t')|}
        {(w (x))^{- \delta - |\alpha'+\beta'|)}}
   \frac{|\partial_x^{\alpha''+\beta''} \partial_t^{j''} v (x,t')
        - \partial_x^{\alpha''+\beta''} \partial_t^{j''} v (x,t'')|}
        {(w (x))^{- \delta' - |\alpha''+\beta''|} |t'-t''|^{\lambda/2}}
\\
 & + &
   c
   \sum_{{\alpha' \leq \alpha \atop \beta' \leq \beta} \atop j' \leq j}
   \frac{|\partial_x^{\alpha'+\beta'} \partial_t^{j'} u (x,t')
        - \partial_x^{\alpha'+\beta'} \partial_t^{j'} u (x,t'')|}
        {(w (x))^{- \delta - |\alpha'+\beta'|} |t'-t''|^{\lambda/2}}
   \frac{|\partial_x^{\alpha''+\beta''} \partial_t^{j''} v (x,t'')|}
        {(w (x))^{- \delta' - |\alpha''+\beta''|)}},
\end{eqnarray*}
the constant $c$ need not be the same in diverse applications.
Hence,
\begin{eqnarray*}
\lefteqn{
   \langle \partial_x^{\alpha+\beta} \partial_t^j (u v) (x,\cdot)
   \rangle_{\lambda/2, [0,T], C^{0,0,\delta + \delta' + |\alpha+\beta|} (\mathbb{R}^n)}
}
\\
 & \!\!\! \leq \!\!\! &
   \! c \!
   \sum_{{\alpha' \leq \alpha \atop \beta' \leq \beta} \atop j' \leq j}
   \| \partial_x^{\alpha'\!+\!\beta'}\! \partial_t^{j'}\! u
   \|_{C^{\mathbf{s} (0,0,\delta\!+\!|\alpha'\!+\!\beta'|)} (\overline{\mathcal{C}_T})}\,
   \langle \partial_x^{\alpha''\!+\!\beta''}\! \partial_t^{j''}\! v
   \rangle_{\lambda/2, [0,T], C^{0,0,\delta'\!+\!|\alpha''\!+\!\beta''|} (\mathbb{R}^n)}
\\
 & \!\!\! + \!\!\! &
   \! c \!
   \sum_{{\alpha' \leq \alpha \atop \beta' \leq \beta} \atop j' \leq j}
   \langle \partial_x^{\alpha'\!+\!\beta'} \partial_t^{j'} u
   \rangle_{\lambda/2, [0,T], C^{0,0,\delta\!+\!|\alpha'\!+\!\beta'|} (\mathbb{R}^n)}\,
   \| \partial_x^{\alpha''\!+\!\beta''}\! \partial_t^{j''}\! v
   \|_{C^{\mathbf{s} (0,0,\delta'\!+\!|\alpha''\!+\!\beta''|)} (\overline{\mathcal{C}_T})}.
\end{eqnarray*}

Applying the Cauchy-Schwarz inequality we conclude that (\ref{eq.product}) is actually fulfilled
for all $\lambda \in (0,1]$, for the estimates of the term
   $\| u v \|_{C^{k, \mathbf{s} (s,\lambda,0)} (\overline{\mathcal{C}_T (B_1)})}$
are similar.
\end{proof}

\begin{lemma}
\label{l.diff.oper}
Suppose that $s \geq 1$ and $k \geq 0$ are integers and $\lambda \in [0,1]$.
Then it follows that

1)
$\partial_x^\alpha$ maps
   $C^{k, \mathbf{s} (s,\lambda,\delta)} (\overline{\mathcal{C}_T})$
continuously into
   $C^{k\!-\!|\alpha|, \mathbf{s} (s,\lambda,\delta\!+\!|\alpha|)} (\overline{\mathcal{C}_T})$,
if $|\alpha| \leq k$;

2)
$\partial_x^\alpha$ maps
   $C^{\mathbf{s} (s,\lambda,\delta)} (\mathbb{R}^n \! \times \! [0,T])$
continuously into
   $C^{2s\!-\!|\alpha|,\lambda, s\!-\!1,\frac{\scriptstyle \lambda}{\scriptstyle 2}, \delta\!+\!|\alpha|}
    (\mathbb{R}^n \! \times \! [0,T])$,
if $1 \leq |\alpha| \leq 2$;

3)
$\partial_t^j$ maps
   $C^{k, \mathbf{s} (s,\lambda,\delta)} (\overline{\mathcal{C}_T})$
continuously into
   $C^{k, \mathbf{s} (s-j,\lambda,\delta)} (\overline{\mathcal{C}_T})$,
if $0 \leq j \leq s$;

4)
the heat operator $H_{\mu}$ maps
   $C^{k, \mathbf{s} (s,\lambda,\delta)} (\overline{\mathcal{C}_T})$
continuously into
   $C^{k, \mathbf{s} (s-1,\lambda,\delta)} (\overline{\mathcal{C}_T})$.
\end{lemma}

\begin{proof}
The first three assertions follow readily from the definition of the spaces.
To derive the last assertions from the first three ones it suffices to apply Theorem \ref{t.emb.hoelder.t}
according to which the space
   $C^{k, \mathbf{s} (s-1,\lambda,\delta+2)} (\overline{\mathcal{C}_T})$
is embedded continuously into
   $C^{k, \mathbf{s} (s-1,\lambda,\delta)} (\overline{\mathcal{C}_T})$.
\end{proof}

Aiming at the investigation of linearisations of the Navier-Stokes equations we now consider the action of
differential operators with variable coefficients in the scale
   $C^{k, \mathbf{s} (s,\lambda,\delta)} (\overline{\mathcal{C}_T})$.
Set
\begin{equation}
\label{eq.V0}
   Pu
 = \sum_{|\alpha| \leq 1} P_\alpha  (x,t) \partial^\alpha u
\end{equation}
for $u \in C^1 (\mathbb{R}^n, \varLambda^q)$, where
   $P_{\alpha}$ are $(k_q \times k_q)\,$-matrices of differentiable functions on $\overline{\mathcal{C}_T} $
and
   $k_q$ the rank of the bundle $\varLambda^q$.

\begin{lemma}
\label{l.map.hoelder.V0}
Let
   $s \geq 1$ and $k \geq 0$ be integers,
   $0 < \lambda \leq 1$
and
   $\delta, \delta' > 0$.
If the entries of $P_\alpha$, $|\alpha| \leq 1$, belong to
   $C^{k + \mathbf{s} (s-1,\lambda,\delta'-|\alpha|)} (\overline{\mathcal{C}_T})$,
then (\ref{eq.V0}) induces a bounded linear map
$$
   P :\, C^{k, \mathbf{s} (s,\lambda,\delta)} (\overline{\mathcal{C}_T} ,\varLambda^q) \to
         C^{k, \mathbf{s} (s-1,\lambda,\delta+\delta')} (\overline{\mathcal{C}_T} ,\varLambda^q).
$$
\end{lemma}

\begin{proof}
If $0 < \lambda \leq 1$, then according to
   Lemma \ref{l.product} and
   Theorems \ref{t.emb.hoelder.t}
we get immediately
\begin{eqnarray*}
\lefteqn{
   \| P_\alpha \partial^\alpha u \|_{C^{k, \mathbf{s} (s-1,\lambda,\delta+\delta')} (\overline{\mathcal{C}_T} ,\varLambda^q)}
}
\\
 & \leq &
   c'\,
   \| P_\alpha \|_{C^{k, \mathbf{s} (s-1,\lambda,\delta'-|\alpha|)} (\overline{\mathcal{C}_T} , \mathrm{Hom} (\varLambda^q))}
   \| \partial^\alpha u \|_{C^{k, \mathbf{s} (s-1,\lambda,\delta+|\alpha|)} (\overline{\mathcal{C}_T} ,\varLambda^q)}
\\
 & \leq &
   c''\,
   \| P_\alpha \|_{C^{k, \mathbf{s} (s-1,\lambda,\delta'-|\alpha|)} (\overline{\mathcal{C}_T} , \mathrm{Hom} (\varLambda^q))}
   \| u \|_{C^{k+|\alpha|, \mathbf{s} (s-1,\lambda,\delta)} (\overline{\mathcal{C}_T} ,\varLambda^q)}
\\
 & \leq &
   c'''\,
   \| P_\alpha \|_{C^{k, \mathbf{s} (s-1,\lambda,\delta'-|\alpha|)} (\overline{\mathcal{C}_T} , \mathrm{Hom} (\varLambda^q))}
   \| u \|_{C^{k, \mathbf{s} (s,\lambda,\delta)} (\overline{\mathcal{C}_T} ,\varLambda^q)}
\end{eqnarray*}
with some constants $c'$, $c''$ and $c'''$ independent of
   $u \in C^{k, \mathbf{s} (s,\lambda,\delta)} (\overline{\mathcal{C}_T} ,\varLambda^q)$,
because the space
   $C^{k, \mathbf{s} (s,\lambda,\delta)} (\overline{\mathcal{C}_T})$ is embedded continuously into
   $C^{k+1, \mathbf{s} (s-1,\lambda,\delta)} (\overline{\mathcal{C}_T})$.
\end{proof}

Lemma \ref{l.diff.oper} shows that the scale of weighted spaces
   $C^{k, \mathbf{s} (s,\lambda,\delta)} (\overline{\mathcal{C}_T})$
does not fully agree with the dilation principle for parabolic equations so far as it concerns the weight.
Hence, when solving the Cauchy problem for the heat equation, we should expect some loss of regularity with
respect to either the smoothness or the weight.

\section{The de Rham complex over weighted H\"older spaces}
\label{s.deRham.HS}

Motivated by the factorisation of linearised Navier-Stokes equations we are interested in describing
the behaviour of the Laplace operator and the de Rham complex in the scales
   $C^{s,\lambda,\delta}$ and
   $C^{k, \mathbf{s} (s,\lambda,\delta)} (\overline{\mathcal{C}_T})$.
Actually, it is  well known and similar to the behaviour of the Laplace operator in the scale of
weighted Sobolev spaces (see for instance \cite{McOw79}).

Let $H_{\leq m}$ stand for the space of all harmonic polynomials of degree $\leq m$ in the space
variable $x \in \mathbb{R}^n$.
Denote by $R^{s,\lambda,\delta+2} (\mathbb{R}^n)$ the range of the bounded linear operator
\begin{equation}
\label{eq.Laplace}
   \varDelta :\, C^{s+2,\lambda,\delta} (\mathbb{R}^n) \to
              C^{s,\lambda,\delta+2} (\mathbb{R}^n)
\end{equation}
induced by the Laplace operator $\varDelta$.

\begin{theorem}
\label{t.weight.Hoelder.Laplace}
Assume that
   $n \geq 2$,
   $s$ is a nonnegative integer and
   $0 < \lambda < 1$.
If moreover
   $\delta > 0$ and
   $\delta+2-n \not\in \mathbb{Z}_{\geq 0}$,
then the operator (\ref{eq.Laplace}) is Fredholm.
Moreover,

1)
(\ref{eq.Laplace}) is an isomorphism, if $0 < \delta < n-2$;

2)
(\ref{eq.Laplace}) is an injection, if $n-2+m < \delta < n-1+m$ for $m \in \mathbb{Z}_{\geq 0}$
and its (closed) range $R^{s,\lambda,\delta+2} (\mathbb{R}^n)$ consists of all
   $f \in  C^{s,\lambda,\delta+2} (\mathbb{R}^n)$
satisfying
$$
   \int_{\mathbb{R}^n} f (x) h (x) dx = 0
$$
whenever $h \in H_{\leq m}$.
%
\end{theorem}

\begin{proof}
See for instance \cite{Be11}, \cite{Mar02}.
The key tool in the proof is the Newton potential
\begin{equation}
\label{eq.volume.potential}
   \varPhi f\, (x) = \int_{\mathbb{R}^n} \phi (x-y)\, f (y)\, dy
\end{equation}
on all of $\mathbb{R}^n$ defined for functions $f$ over $\mathbb{R}^n$, where
$$
   \phi (x)
 = \left\{ \begin{array}{lcl}
             \displaystyle
             \frac{1}{\pi} \ln |x|,
           & \mbox{for}
           & n = 2,
\\
             \displaystyle
             \frac{1}{\sigma_n} \frac{|x|^{2-n}}{2-n},
           & \mbox{for}
           & n \geq 3,
           \end{array}
   \right.
$$
is the standard two-sided fundamental solution of the convolution type to the Laplace operator
in $\mathbb{R}^n$ and $\sigma_n$ the area of the unit sphere in $\mathbb{R}^n$.
Let us briefly sketch the proof.

The crucial role in the proof is played by the following a priori estimate of Schauder type for
the Laplace operator.

\begin{lemma}
\label{l.apriori.Delta}
Suppose $\delta > 0$.
If
   $f \in C^{0,\lambda,\delta+2} (\mathbb{R}^n)$ and
   $u \in C^{0,0,\delta} (\mathbb{R}^n)$
satisfies $\varDelta u = f$ in the sense of distributions in $\mathbb{R}^n$, then
   $u \in C^{2,\lambda,\delta} (\mathbb{R}^n)$
and
$$
   \| u \|_{C^{2,\lambda,\delta} (\mathbb{R}^n)}
 \leq
   c \left( \| f \|_{C^{0,\lambda,\delta+2} (\mathbb{R}^n)}
          + \| u \|_{C^{0,0,\delta} (\mathbb{R}^n)}
     \right)
$$
with $c$ a constant depending on $\lambda$ and $\delta$ but not on $u$.
\end{lemma}

\begin{proof}
The proof is based on a priori estimates of Schauder type for solutions of elliptic
equations, see for instance
   \cite{GiTru83} for H\"older spaces,
   \cite{NireWalk73}, \cite{McOw79} for weighted Sobolev spaces,
   \cite{MazRoss04} for weighted H\"older spaces on an infinite cone
and
   Proposition 2.7 of \cite{Be11} and Theorem 4.21 of \cite{Mar02} for weighted H\"older
   spaces on a manifold with conical points.

Indeed, by Lemma \ref{l.emb.loc} and elliptic regularity we conclude that any function $u$
satisfying the hypotheses of the lemma belongs to
   $C^{2,\lambda}_{\mathrm{loc}} (\mathbb{R}^n)$.
Using standard a priori estimates for the Laplace operator yields
\begin{equation}
\label{eq.GilTru}
\begin{array}{rcl}
   \displaystyle
   \| u \|_{C^{2,\lambda} (\overline{B_1})}
 & \leq &
   \displaystyle
   c \left( \| \varDelta u \|_{C^{0,\lambda} (\overline{B_2})}
          + \| u \|_{C^{0,0} (\overline{B_2})}
   \right),
\\
   \displaystyle
   \sum_{|\alpha| \leq 2} \| \partial^\alpha u \|_{C^{0,\lambda} (\overline{B_2})}
 & \leq &
   \displaystyle
   c \left( \| \varDelta u \|_{C^{0,\lambda} (\overline{B_4})}
          + \| u \|_{C^{0,0} (\overline{B_4})}
     \right)
\end{array}
\end{equation}
for all $u$ as in the statement of the lemma,
   with $c$ a constant depending on the ratio of the radii of the balls but not on $u$.
(See for instance
   Theorem 4.6 in \cite[\S~4.2, \S~4.3]{GiTru83}, cf. also
   Theorem 9.11 for Lebesgue spaces {\it ibid}.)
For $|x| \leq 4$ one verifies easily $1 \leq \sqrt{1+|x|^2} \leq \sqrt{17}$.
Therefore, using (\ref{eq.half}) we may write the last inequality as
\begin{equation}
\label{eq.apriori.1}
   \sum_{|\alpha| \leq 2} \| \partial^\alpha u \|_{C^{0,\lambda,\delta+|\alpha|} (\overline{B_2})}
 \leq
   c \left( \| \varDelta u \|_{C^{0,\lambda,\delta+2} (\overline{B_4})}
          + \| u \|_{C^{0,0,\delta} (\overline{B_4})}
     \right)
\end{equation}
for all $u$ as in the statement of the lemma,
   the constant $c$ depends on $\delta$ and need not be the same in diverse applications.

We next consider a spherical layer in $\mathbb{R}^n$ of the form $r < |x| < 8 r$, where $r \geq 1$
is a fixed constant, and we define the function
$$
   u_r (x) := u (r x)
$$
for $1 \leq |x| \leq 8$.
Then
$$
\begin{array}{rcl}
   \partial_i u_r (x)
 & =
 & r (\partial_i u) (r x),
\\
   \varDelta u_r (x)
 & =
 & r^2 (\varDelta u) (r x),
\end{array}
$$
that is $r^2 f (r x)$.
Again by standard a priori estimates we obtain
$$
   \| \partial^\alpha u_r \|_{C^{0,\lambda} (\overline{B_4} \setminus B_2)}
 \leq
   c
   \left( \| \varDelta u_r \|_{C^{0,\lambda} (\overline{B_8} \setminus B_1)}
        + \| u_r \|_{C^{0,0} (\overline{B_8} \setminus B_1)}
   \right)
$$
for $|\alpha| \leq 2$, see for instance
   Theorem 4.6 in \cite[\S~4.2, \S~4.3]{GiTru83} and
   Theorem 9.11 for Lebesgue spaces {\it ibid}.
For the original function $u$ this reduces to
\begin{eqnarray*}
\lefteqn{
   r^{\delta+|\alpha|}
   \| \partial^\alpha u \|_{C^{0,\lambda} (\overline{B_{4r}} \setminus B_{2r})}
}
\\
 & \leq &
   c
   \left( r^{\delta+2} \| \varDelta u \|_{C^{0,\lambda} (\overline{B_{8r}} \setminus B_r)}
        + r^\delta \| u \|_{C^{0,0} (\overline{B_{8r}} \setminus B_r)}
   \right)
\end{eqnarray*}
with $c$ a constant depending on the ratio of the radii of the balls but not on $u$.
Note that if $r \leq  |x| \leq 8r$ then $r < \sqrt{1+|x|^2} \leq 9r$.
Therefore, on applying estimate (\ref{eq.GilTru}) we get
\begin{eqnarray*}
\lefteqn{
   \| \partial^\alpha u \|_{C^{0,\lambda,\delta+|\beta|} (\overline{B_{4r}} \setminus B_{2r})}
}
\\
 & \leq &
   c
   \left( \| \varDelta u \|_{C^{0,\lambda,\delta+2} (\overline{B_{8r}} \setminus B_r)}
        + \| u \|_{C^{0,0,\delta} (\overline{B_{8r}} \setminus B_r)}
   \right)
\end{eqnarray*}
for $|\alpha| \leq 2$, where the constant $c$ depends on $\delta$ but not on $r \geq 1$ and $u$.

We now choose $r = 2^m$ with $m = 0, 1, \ldots$.
For any multi-index $\alpha$ satisfying $|\alpha| \leq 2$ it follows that
\begin{eqnarray*}
\lefteqn{
   \| \partial^\alpha u \|_{C^{0,\lambda,\delta+|\beta|} (\overline{B_{2^{m+2}}} \setminus B_{2^{m+1}})}
}
\\
 & \leq &
   c
   \left( \| \varDelta u \|_{C^{0,\lambda,\delta+2} (\overline{B_{2^{m+3}}} \setminus B_{2^m})}
        + \| u \|_{C^{0,0,\delta} (\overline{B_{2^{m+3}}} \setminus B_{2^m})}
   \right),
\end{eqnarray*}
where the constant $c$ depends neither on $m = 0, 1, \ldots$ nor on $u$.
Combining these sequence of inequalities in spherical layers with
   (\ref{eq.GilTru}) and
   (\ref{eq.apriori.1})
establishes the lemma.
\end{proof}

Let us continue with the proof of Theorem \ref{t.weight.Hoelder.Laplace}.
First we note that the Liouville theorem implies that the operator $\varDelta $ is injective on
   $C^{s,\lambda,\delta} (\mathbb{R}^n)$,
for any $s \geq 2$ and $\delta > 0$, and the kernel of the operator $\varDelta$ on
   $C^{s,\lambda,\delta} (\mathbb{R}^n)$
is equal to $H_{\leq m}$, if $-m-1 < \delta < -m$ with $m = 0, 1, \ldots$.

Second, let $H_m$ stand for the set of all homogeneous harmonic polynomials of degree $m$ with respect
to the space variable $x$.
Note that if $m$ is a nonnegative integer and $\delta > n-2+m$ then for
   $f \in C^{0,0,\delta+2} (\mathbb{R}^n)$
to be in the range of $\varDelta$ acting on
   $C^{1,0,\delta} (\mathbb{R}^n) \cap C^2 (\mathbb{R}^n)$
it is necessary that
$$
   \int_{\mathbb{R}^n} f h_j dx
 = \int_{\mathbb{R}^n} (\varDelta u) h_j dx
 = \lim_{R \to +\infty} \int_{|x|=R}
   \Big( \frac{\partial u}{\partial \nu} h_j - u \frac{\partial h_j}{\partial \nu} \Big) ds
 = 0
$$
for all $h_j \in H_j$ with $0 \leq j \leq m$, because
$$
R^{n-1-\delta+(j-1)} = R^{n-1-(\delta+1)+j} = R^{n-2+j-\delta} \to 0
$$
as $R \to +\infty$.
Thus for $n = 2$ the isomorphism described in the item $1)$ is impossible.

Clearly, the function $(w (x))^{-\delta-2}$ belongs to $C^{s,0,\delta+2} (\mathbb{R}^n)$ for any
$\delta \in \mathbb{R}$ and all $s \in \mathbb{Z}_{\geq 0}$.
Our next objective is to Let us construct a formal solution to the inhomogeneous equation
\begin{equation}
\label{eq.Delta.2}
   \varDelta F  = \frac{1}{(1+|x|^2)^{(\delta+2)/2}}
\end{equation}
in $\mathbb{R}^n \setminus \{ 0 \}$.
To this end we introduce
$$
   F (x) = \sum_{k=0}^\infty \frac{a_k}{(1+|x|^2)^{(\delta+2k)/2}}
$$
as a formal series.
Clearly, the coefficients $a_k$ are uniquely determined from equality (\ref{eq.Delta.2}).

\begin{lemma}
\label{l.F0}
Let
   $\delta$ be a real number different from $0, n-2, n-4, \ldots$.
The series $F$ converges uniformly along with all derivatives on compact subsets away from the origin
in $\mathbb{R}^n$.
Moreover the function $F$ belongs to
   $C^{s,0,\delta} (\mathbb{R}^n \setminus B_1)$
for any $s \in \mathbb{Z}_{\geq 0}$ and it satisfies (\ref{eq.Delta.2}).
\end{lemma}

\begin{proof}
It easy to verify that
\begin{eqnarray*}
   \partial_i F\, (x)
 & = &
   \sum_{k=0}^\infty \frac{-(\delta+2k) x_i\, a_k}{(1+|x|^2)^{(\delta+2k+2)/2}},
\\
   \partial_i^2 F\, (x)
 & = &
   \sum_{k=0}^\infty
   \Big( \frac{(\delta+2k) (\delta+2k+2)\, x_i^2\, a_k}{(1+|x|^2)^{(\delta+2k+4)/2}}
       - \frac{(\delta+2k)\, a_k}{(1+|x|^2)^{(\delta+2k+2)/2}} \Big),
\end{eqnarray*}
and so
\begin{eqnarray*}
   \varDelta F\, (x)
 & = &
   \sum_{k=0}^\infty
   \Big( \frac{(\delta+2k) (\delta+2k+2)\, |x|^2\, a_k}{(1+|x|^2)^{(\delta+2k+4)/2}}
       - \frac{(\delta+2k) n\, a_k}{(1+|x|^2)^{(\delta+2k+2)/2}} \Big)
\\
 & = &
 \sum_{k=0}^\infty
 \Big( \frac{(\delta+2k) (\delta+2k+2-n)\, a_k}{(1+|x|^2)^{(\delta+2k+2)/2}}
     - \frac{(\delta+2k) (\delta+2k+2)\, a_k}{(1+|x|^2)^{(\delta+2k+4)/2}} \Big)
\\
 & = &
   \frac{\delta (\delta\!+\!2\!-\!n)\, a_0}{(1+|x|^2)^{(\delta+2)/2}}
 + \sum_{k=1}^\infty
   \frac{(\delta\!+\!2k) \left( (\delta\!+\!2k\!+\!2\!-\!n) a_k - (\delta\!+\!2k\!-\!2) a_{k\!-\!1} \right)}
   {(1+|x|^2)^{(\delta+2k+2)/2}}
\end{eqnarray*}
as formal series.
In particular, if
\begin{equation}
\label{eq.Abel}
\begin{array}{rcl}
   a_0
 & =
 & \displaystyle
   \frac{1}{\delta (\delta+2-n)},
\\
   a_k
 & =
 & \displaystyle
   \frac{\delta+2k-2}{\delta+2k+2-n}\, a_{k-1}
\end{array}
\end{equation}
for $k \geq 1$, then
$$
   \varDelta F\, (x)
 = \frac{1}{(1+|x|^2)^{(\delta+2)/2}}
$$
as formal series.
We get
\begin{eqnarray*}
   a_1
 & = &
   \frac{1}{(\delta+2-n)(\delta+4-n)},
\\
   a_2
 & = &
   \frac{\delta+2}{(\delta+2-n) (\delta+4-n) (\delta+6-n)}
\end{eqnarray*}
and more generally
$$
   a_k
 = \frac{\displaystyle \prod_{j=1}^{k-1} (\delta+2j)}
        {\displaystyle \prod_{j=1}^{k+1} (\delta+2j-n))},
$$
for $k \geq 3$, provided that $\delta$ is different  from $0, n-2, n-4, \ldots$.
From (\ref{eq.Abel}) it follows that the convergence domain of the power series
$$
   \sum_{k=0}^\infty a_k z^k
$$
coincides with the unit disc in $\mathbb{C}$.
Hence, by the Abel theorem the series $F$ converges uniformly along with all derivatives
on compact subsets of $\mathbb{R}^n \setminus \{ 0 \}$.
Its sum belongs actually to
   $C^\infty_{\mathrm{loc}} (\mathbb{R}^n \setminus \{ 0 \}) \cap
    C^{0,0,\delta} (\mathbb{R}^n \setminus B_1)$,
for
$$
   (w (x))^\delta F (x)
 = \sum_{k=0}^\infty \frac{a_k (w (x))^\delta}{(1+|x|^2)^{(\delta+2k)/2}}
 = \sum_{k=0}^\infty \frac{a_k}{(1+|x|^2)^k}.
$$
Next,
\begin{eqnarray*}
   (w (x))^{\delta+1} | \partial_i F (x) |
 & \leq &
   \sum_{k=0}^\infty
   \frac{|\delta+2k| |a_k| |x_i| (w (x))^{\delta+1}}{(1+|x|^2)^{(\delta+2k+2)/2}}
\\
 & = &
   \frac{|x_i|}{(1+|x|^2)^{1/2}}
   \sum_{k=0}^\infty  \frac{|\delta+2k| |a_k|}{(1+|x|^2)^k}.
\end{eqnarray*}
Formula (\ref{eq.Abel}) implies that the convergence radius of the series
$$
   \sum_{k=0}^\infty |\delta+2k| |a_k| z^k
$$
equals $1$ whence
$$
   (w (x))^{\delta+1} |\partial_i F (x)|
 \leq
   \sum_{k=0}^\infty  \frac{|\delta+2k| |a_k|}{2^k}
 < \infty,
$$
i.e., $F \in C^{1,0,\delta} (\mathbb{R}^n \setminus B_1)$.

We now proceed by induction.
For any multi-index $\alpha \in \mathbb{Z}_{\geq 0}$ we readily verify that
$$
   (w (x))^{\delta+|\alpha|} \partial^\alpha F (x)
 = \sum_{k=0}^\infty
   P_{\alpha,k} \Big( \frac{x}{w (x)} \Big) \frac{a_k}{(1+|x|^2)^{k}},
$$
where $P_{\alpha,k}$ are polynomials of degree $\leq |\alpha|$ of $n$ variables, such
that
$$
   \Big| P_{\alpha,k} \Big( \frac{x}{w (x)} \Big) \Big|
 \leq
   c\, k^{|\alpha|}
$$
for all $x \in \mathbb{R}^n$, the constant $c$ depending on $\delta$ but not on $k$.
On arguing as above we deduce that
   $F \in C^{s,0,\delta} (\mathbb{R}^n \setminus B_1)$
for any nonnegative integer $s$.
By the construction, $F$ satisfies (\ref{eq.Delta.2}).
\end{proof}

\begin{lemma}
\label{l.phi.Hoelder}
If $\delta > 0$ then the potential $\varPhi f$ given by (\ref{eq.volume.potential}) satisfies
   $\varDelta (\varPhi f) = f$
in the sense of distributions on $\mathbb{R}^n$ for all
   $f \in C^{0,\lambda,\delta+2} (\mathbb{R}^n)$.
Moreover, if\, $0 < \delta < n-2$ then the potential (\ref{eq.volume.potential}) induces
the bounded map
$
   \varPhi : C^{0,\lambda,\delta+2} (\mathbb{R}^n) \to C^{2,\lambda,\delta} (\mathbb{R}^n)
$
for all $\lambda \in (0,1)$.
\end{lemma}

\begin{proof}
We begin with $n > 2$.

Recall that
   $C^{0,\lambda,\delta+2} (\mathbb{R}^n) \hookrightarrow
    C^{0,\lambda}_{\mathrm{loc}} (\mathbb{R}^n)$.
Fix $f \in C^{0,\lambda,\delta+2} (\mathbb{R}^n)$ with $\delta > 0$.
First, for $x = 0$, we get
\begin{eqnarray*}
   |(\varPhi f) (0)|
 & \leq &
   \frac{1}{\sigma_n (n-2)} \int_{\mathbb{R}^n} \frac{|f (y)|}{|y|^{n-2}} dy
\\
 & \leq &
   \frac{1}{\sigma_n (n-2)} \int_{|y| \leq 1} \frac{|f (y)|}{|y|^{n-2}} dy
 + \frac{1}{\sigma_n (n-2)} \int_{|y| \geq 1} \frac{|f (y)|}{|y|^{n-2}} dy
\\
 & \leq &
   \frac{\| f \|_{C^{0,0,\delta} (\mathbb{R}^n)}}{\sigma_n (n-2)}
   \Big( \int_{|y| \leq 1} \frac{(1+|y|^2)^{-(\delta+2)/2}}{|y|^{n-2}} dy
       + \int_{|y| \geq 1} |y|^{-(\delta+n)} dy
   \Big).
\end{eqnarray*}
The first integral in the parentheses converges because $n - 2 < n$, and the second integral
converges because $\delta + n > n$.

If $x \ne 0$, then
$$
   |x-y|
 = |y| \Big| \frac{y}{|y|} - \frac{x}{|y|} \Big|
 \geq |y| (1-1/2)
 = |y|/2
$$
for all $x \in \mathbb{R}^n$ satisfying $|x| \leq |y|/2$.
Hence we obtain
\begin{eqnarray}
\label{eq.uni}
\lefteqn{
   |(\varPhi f) (x)|
}
 \nonumber
 \\
 & \leq &
   \frac{1}{\sigma_n (n-2)} \int_{|y| \leq 2 |x|} \frac{|f (y)|}{|y-x|^{n-2}} dy
 + \frac{1}{\sigma_n (n-2)} \int_{|y| \geq 2 |x|} \frac{|f (y)|}{|y-x|^{n-2}} dy
\nonumber
\\
 & \leq &
   \frac{\| f \|_{C^{0,0,\delta} (\mathbb{R}^n)}}{\sigma_n (n-2)}
   \Big( \int_{|y| \leq 2 |x|} \frac{(1+|y|^2)^{-(\delta+2)/2}}{|x-y|^{n-2}} dy
       + \int_{|y| \geq 2 |x|} 2^{n-2} |y|^{-(\delta+n)} dy
   \Big).
\nonumber
\\
\end{eqnarray}
Again the first integral in the last line converges, for $n-2 < n$, and the second integral
in the last line converges, for $n + \delta > n$.
Thus, the potential $(\varPhi f) (x)$ is well defined for all $x \in \mathbb{R}^n$.

It follows from potential theory (see for instance \cite{Gun34}) that for each $R > 0$ the
integral
$$
   \varPhi (\chi_{B_R} f) (x)
 = \frac{1}{\sigma_n (n-2)} \int_{|y| \leq R} \frac{f (y)}{|x-y|^{n-2}}\, dy
$$
converges uniformly in the ball $\overline{B_R}$ and it belongs to $C^{2,\lambda} (\overline{B_R})$
provided that
   $f \in C^{0,\lambda,\delta+2} (\mathbb{R}^n)$.
Clearly the integral
$$
    \varPhi ((1 - \chi_{B_R}) f) (x)
 = \frac{1}{\sigma_n (n-2)} \int_{|y| \geq R} \frac{f (y)}{|x-y|^{n-2}}\, dy
$$
is a $C^\infty$ functions of $x \in B_R$.
We thus deduce that the potential $\varPhi f$ belongs to $C^{2,\lambda} (B_R)$ in any ball $B_R$,
and so
   $\varPhi f \in C^{2,\lambda}_{\mathrm{loc}} (\mathbb{R}^n)$
for each $f \in C^{0,\lambda,\delta+2} (\mathbb{R}^n)$ with $\delta > 0$.

Moreover, it follows from (\ref{eq.uni}) that the integral $\varPhi f$ converges uniformly on
each compact set $K \subset \mathbb{R}^n$.
For any
   $v \in C^{\infty}_{\mathrm{comp}} (\mathbb{R}^n)$,
using Fubini theorem we get
$$
   \int_{\mathbb{R}^n} (\varPhi f) (x) \varDelta v (x) dx
 = \int_{\mathbb{R}^n} f (y) (\varPhi \varDelta v) (y) dy
 = \int_{\mathbb{R}^n} f (y) v (y) dy,
$$
i.e. $\varDelta (\varPhi f) = f$ in the sense of distributions in $\mathbb{R}^n$ for each
   $f \in C^{0,\lambda,\delta+2} (\mathbb{R}^n)$
with $\delta > 0$, for $e$ is a fundamental solution of convolution type to the Laplace
operator in $\mathbb{R}^n$.

As $f \in C^{0,\lambda,\delta+2} (\mathbb{R}^n)$, we get
$$
   \| (w (x))^{\delta+2-\varepsilon} f \|_{L^q (\mathbb{R}^n)}
 \leq
   \| f \|_{C^{0,0,\delta+2} (\mathbb{R}^n)}
   \| (w (x))^{-\varepsilon} (x) \|_{L^q (\mathbb{R}^n)}
$$
for all $0 < \varepsilon < \delta+2$, i.e., the function $f$  belongs to
   $W^{0,q,\delta+2-\varepsilon} (\mathbb{R}^n)$
for all $q > n/\varepsilon$.
If $0 < \varepsilon < \delta$ then
$$
 - n/q < 0 < \delta-\varepsilon < n-2-n/q
$$
for all $q > n/(n-2-\delta+\varepsilon)$.
According to \cite{McOw79} 
we conclude that
   $\varPhi f \in W^{2,q,\delta-\varepsilon} (\mathbb{R}^n)$.
In particular, the potential $\varPhi f$ vanishes at the infinity point.

Given any $f \in C^{0,\lambda,\delta+2} (\mathbb{R}^n)$, we get
\begin{equation}
\label{eq.F.1}
   \Big| \int_{\mathbb{R}^n} \frac{f (y)}{|x-y|^{n-2}} dy \Big|
 \leq
   c\,
   \| f \|_{C_{0,0,\delta+2} (\mathbb{R}^n)}
   \int_{\mathbb{R}^n} \frac{(w (y))^{-(\delta+2)}}{|x-y|^{n-2}} dy
\end{equation}
with $c$ a constant independent of $x$ and $f$.
As $(w (y))^{-\delta-2}$ belongs to $C^{s,\lambda,\delta+2} (\mathbb{R}^n)$ for any
$s \in \mathbb{Z}_{\geq 0}$,
the potential $\varPhi w^{-\delta-2}$
belongs to $C^{\infty}_{\mathrm{loc}} (\mathbb{R}^n)$, vanishes at infinity and satisfies
$$
   \varDelta (\varPhi w^{-\delta-2}) = w^{-\delta-2}
$$
in $\mathbb{R}^n$.
On the other hand, the integral
$
   \varPhi (\chi_{B_1} w^{-(\delta+2)})
$
is harmonic away from the closed ball $\overline{B_1}$ and it belongs obviously to
   $C^{s,0,n-2} (\mathbb{R}^n \setminus B_1)$
for all $s \in \mathbb{Z}_{\geq 0}$.
It follows that
\begin{equation}
\label{eq.Delta.1}
   \varDelta\, \varPhi ((1 - \chi_{B_1}) w^{-\delta-2})
 = w^{-\delta-2}
\end{equation}
in $\mathbb{R}^n \setminus \overline{B_1}$.

On combining (\ref{eq.Delta.1}) and (\ref{eq.Delta.2}) we see that the difference
$$
   h = \varPhi ((1 - \chi_{B_1}) w^{-\delta-2}) - F
$$
is a harmonic function in $\mathbb{R}^n \setminus \overline{B_1}$ vanishing at the point
of infinity.
It can be recovered via its smooth boundary values on $\partial B_1$.
Indeed, let $\{ h_k^{(j)} (x) \} $ be the system of homogeneous harmonic polynomials forming
an $L^2\,$-orthonormal basis on the unit sphere in $\mathbb{R}^n$, where
   $k$ stands for the polynomial degree and
   $j = 1, \ldots, J (k)$,
with
$$
   J (k) = \frac{(n+2k-2) (n+k-3)!}{k! (n-2)!}
$$
being the number of polynomials of degree $k$ in the basis,
   see \cite[Ch. XI]{Sob74}).
Then the results on the exterior Dirichlet problem show that
$$
   h
 = \sum_{k=0}^\infty
   \sum_{j=1}^{J (k)}
   \frac{c_k^{(j)}}{|x|^{n+2k-2}}\, h_k^{(j)} (x)
 = \frac{1}{|x|^{n-2}}
   \sum_{k=0}^\infty
   \sum_{j=1}^{J (k)}
   \frac{c_k^{(j)}}{|x|^{k}}\, h_k^{(j)} \Big( \frac{x}{|x|} \Big),
$$
where
$$
   c_k^{(j)}
 = \left( \varPhi ((1 - \chi_{B_1}) w^{-\delta-2}) - F,\, h_k^{(j)}
   \right)_{L^2 (\partial B_1)}
$$
and the series converges uniformly on compact subsets of $\mathbb{R}^n \setminus B_1$.
In particular, as
$$
\begin{array}{rcl}
   J (k)
 & \leq
 & c\, k^{n-2},
\\
   \displaystyle
   \max_{|x|=1}{ |h_k^{(j)} (x)|}
 & \leq
 & c_n\, k^{n/2-1},
\end{array}
$$
see \cite[Ch.~XI, \S~2)]{Sob74}) and formula (XI.3.23) {\em ibid}, the Cauchy-Hadamard formula yields
readily
$$
   \limsup_{k \to + \infty} \max_{1 \leq j \leq J (k)} |c_k^{(j)}| \leq 1
$$
and hence
\begin{eqnarray*}
   (1+|x|^2)^{(n-2)/2} |h (x)|
 & \leq &
   \Big( \frac{(1+|x|^2)^{1/2}}{|x|} \Big)^{n-2}
   \sum_{k=0}^\infty
   \sum_{j=1}^{J (k)}
   \frac{|c_k^{(j)}|}{|x|^{k}}\, \Big| h_k^{(j)} \Big( \frac{x}{|x|} \Big) \Big|
\\
 & \leq &
   c \sum_{k=0}^\infty \frac{k^{n/2+n-3}}{|x|^k}
\end{eqnarray*}
with a constant $c$ independent on $k$.
It follows that
   $h \in C^{0,0,n-2} (\mathbb{R}^n \setminus B_1)$.

On the other hand, as already mentioned, the potential $\varPhi (\chi_{B_1} w^{-\delta-2})$ belongs
to
   $C^{s,0,n-2} (\mathbb{R}^n \setminus B_1)$
for all $s = 0, 1, \ldots$, and so
$$
   \varPhi (w^{-\delta-2})
 = \varPhi (\chi_{B_1} w^{-\delta-2})
 + \varPhi ((1 - \chi_{B_1}) w^{-\delta-2})
$$
is of class
   $C^s (\mathbb{R}^n) \cap C^{s,0,\delta} (\mathbb{R}^n \setminus B_1)$
for all $s \in \mathbb{Z}_{\geq 0}$, if $0 < \delta < n-2$.
Thus, $\varPhi (w^{-\delta-2}) \in C^{s,0,\delta} (\mathbb{R}^n)$ for all $s$ and $0 < \delta < n-2$.
We now apply estimate (\ref{eq.F.1}) to conclude that
$$
   \| \varPhi f \|_{C^{0,0,\delta} (\mathbb{R}^n)}
 \leq
   \| f \|_{C^{0,0,\delta+2} (\mathbb{R}^n)}\,
   \| \varPhi (w^{-\delta-2}) \|_{C^{0,0,\delta} (\mathbb{R}^n)}
$$
if $0 < \delta < n-2$.
Now it follows from Lemma \ref{l.apriori.Delta} that $\varPhi$ is the bounded inverse for the operator
$$
   \varDelta :\, C^{2,\lambda,\delta} (\mathbb{R}^n) \to
              C^{0,\lambda,\delta+2} (\mathbb{R}^n)
$$
for all $0 < \delta < n-2$, as desired.
\end{proof}

Lemma \ref{l.phi.Hoelder} gives a proof of the statement 1) for $s = 0$ and $n \geq 3$.
(Note that the statement 1) is actually vacuous for $n = 2$.)

Further, we may use the following decomposition of the standard fundamental solution to the
Laplace equation
\begin{equation}
\label{eq.decomp.phi}
   \phi (x-y)
 = \phi (x-0)
 - \sum_{k=1}^\infty
   \sum_{j=1}^{J (k)}
   \frac{h_k^{(j)} (x) h_k^{(j)} (y)}{(n+2k-2) |x|^{n+2k-2}}
\end{equation}
for $n \geq 2$, where the series converges uniformly along with all derivatives on compact
sets of the cone $\{ |x| > |y| \}$ in $\mathbb{R}^{2n}$
   (see \cite{StWe71}, \cite{Shla92}, \cite{McOw79} and elsewhere).
Set
$$
   \phi_m (x,y)
 = \phi (x-y)
 - \phi (\langle x \rangle-0)
 + \sum_{k=1}^m
   \sum_{j=1}^{J (k)}
   \frac{h_k^{(j)} (\langle x \rangle) h_k^{(j)} (y)}{(n+2k-2) \langle x \rangle^{n+2k-2}},
$$
where $x \mapsto \langle x \rangle$ is a so-called norm smoothing function, i.e.,
   $\langle x \rangle = |x|$ for $|x| \geq 2$ and
   $\langle x \rangle \geq 1$ for all $x \in \mathbb{R}^n$.

\begin{lemma}
\label{l.Gm}
Suppose that
   $n-2+m < \delta < n-1+m$
for some $m \in \mathbb{Z}_{\geq 0}$.
Then the integral
$$
   \varPhi_m f (x) = \int_{\mathbb{R}^n} \phi_m (x,y) f (y) dy
$$
induces a bounded linear operator
$
   \varPhi_{m} : C^{0,\lambda,\delta+2} (\mathbb{R}^n) \to
           C^{2,\lambda,\delta} (\mathbb{R}^n)
$
which coincides with the potential $\varPhi$ on $R^{0,\lambda,\delta+2} (\mathbb{R}^n)$.
\end{lemma}

\begin{proof}
Indeed, since $n+m < \delta+2 < n+m+1$, the integral
$$
   \int_{\mathbb{R}^n} f (y) h_k^{(j)} (y) dy
$$
converges because
\begin{eqnarray*}
   \Big| \int_{\mathbb{R}^n} f (y) h_k^{(j)} (y) dy \Big|
 & \leq &
   \| f \|_{C^{0,0,\delta+2} (\mathbb{R}^n)}
   \int_{\mathbb{R}^n} h_k^{(j)} (y) (w (y))^{-(\delta+2)} dy
\\
 & \leq &
   \| f \|_{C^{0,0,\delta+2} (\mathbb{R}^n)}\,
   \| h^{(j)}_k \|_{C (\partial B_1)}
   \int_{\mathbb{R}^n} (w (y))^{k-\delta-2} dy,
\end{eqnarray*}
the last integral being finite for all $0 \leq k \leq m$ because $-n-1 < k-\delta-2 < -n$.
It follows that the integral operator $K_m$ induced by the kernel
$$
   k_{m} (x,y)
 = \phi (\langle x \rangle-0)
 - \sum_{k=1}^m
   \sum_{j=1}^{J (k)}
   \frac{h_k^{(j)} (\langle x \rangle) h_k^{(j)} (y)}{(n+2k-2) \langle x \rangle^{n+2k-2}}
$$
maps $C^{0,\lambda,\delta+2} (\mathbb{R}^n)$ to functions harmonic outside of the ball $B_2$
and vanishing at the point of infinity.
Hence the integral $(\varPhi_m f) (x)$ converges for all $x \in \mathbb{R}^n$,
   if $f \in C^{0,\lambda,\delta+2} (\mathbb{R}^n)$.
Moreover, since
$$
   \varPhi_m f = \varPhi f - K_m f,
$$
we conclude that
   $\varPhi_m f \in C^{2,\lambda}_{\mathrm{loc}} (\mathbb{R}^n)$
for each function $f \in C^{0,\lambda,\delta+2} (\mathbb{R}^n)$ such that
   $n+m-2 < \delta < n+m-1$.

As already mentioned, if $0 < \varepsilon < \delta+2$, then $f$  belongs to
   $W^{0,q,\delta+2-\varepsilon} (\mathbb{R}^n)$
for all $q > n/\varepsilon$.
If
   $0 < \varepsilon < \delta+2-n-m$
then
$$
   n-2+m-n/q < n-2+m < \delta-\varepsilon < n-1+m-n/q
$$
holds for all $q > n/(n-1+m-\delta+\varepsilon)$.
Thus, by \cite{McOw79}, 
we get
   $\varPhi_m f \in W^{2,q,\delta-\varepsilon} (\mathbb{R}^n)$.
In particular, the potential $\varPhi_m f$ vanishes at the point of infinity.

For $n \geq 3$, it follows from \cite[Lemma 5]{McOw79} that
$$
\begin{array}{rclll}
    |\phi_m (x,y)|
 & \leq
 & \displaystyle
   c\, \frac{(1+|y|^{m+n-2})}{|x-y|^{n-2}(1 + |x|^{m+n-2})},
 & \mbox{if}
 & 2 |y| \geq |x|,
\\
   |\phi_m (x,y)|
 & \leq
 & \displaystyle
   c\, \frac{(1 + |y|^{m+1})}{|x-y|^{n-2}(1+ |x|^{m+1})},
 & \mbox{if}
 & 2 |y| \leq |x|,
\end{array}
$$
where $c$ is a constant independent of $x$ and $y$ which can be different in diverse
applications.
Then
$$
   (w (x)\!)^{\delta}
   \Big| \!\! \int_{|2 y| \geq |x|} \!\!\! \phi_m (x,y) f (y) dy \Big|
 \leq
   c \| f \|_{C^{0,\lambda,\delta+2} (\mathbb{R}^n)}
       (w (x)\!)^{\delta\!-\!m\!-\!n\!+\!2}
       \varPhi (w^{m+n-\delta-4}) (x).
$$
If $n+m-2 < \delta < n+m-1$ then $0 < \delta-m-n+2 < 1 \leq n-2$ and hence by the
assertion 1) which has already been proved we see that $\varPhi (w^{m+n-\delta-4})$ is of class
$C^{0,0,\delta+2-m-n} (\mathbb{R}^n)$.
In particular,
\begin{eqnarray*}
\lefteqn{
   \Big\| \int_{|2 y| \geq |x|} \phi_m (x,y) f (y) dy \Big\|_{C^{0,0,\delta} (\mathbb{R}^n)}
}
\\
 & \leq &
   c\,
   \| f \|_{C^{0,\lambda,\delta+2} (\mathbb{R}^n)}
   \| \varPhi (w^{m+n-\delta-4}) \|_{C^{0,0,\delta+2-m-n} (\mathbb{R}^n)}.
\end{eqnarray*}

Similarly,
$$
   (w (x)\!)^{\delta}
   \Big| \!\! \int_{|2 y| \leq |x|} \!\!\! \phi_m (x,y) f (y) dy \Big|
 \leq
   c \| f \|_{C^{0,\lambda,\delta+2} (\mathbb{R}^n)}
       (w (x)\!)^{\delta\!-\!m\!-\!1}
       \varPhi (w^{m-\delta-1}) (x).
$$
If
   $n \geq 3$ and
   $n+m-2 < \delta < n+m-1$
then
   $0 \leq n-3 < \delta-m-1 < n-2$
and hence by the assertion 1) we conclude that
   $\varPhi (w^{m-\delta-1}) \in C^{0,0,\delta-m-1} (\mathbb{R}^n) $.
It follows that
\begin{eqnarray*}
\lefteqn{
   \Big\| \int_{|2 y| \geq |x|} \phi_m (x,y) f (y) dy \Big\|_{C^{0,0,\delta} (\mathbb{R}^n)}
}
\\
 & \leq &
   c\,
   \| f \|_{C^{0,\lambda,\delta+2} (\mathbb{R}^n)}
   \| \varPhi (w^{m-\delta-1}) \|_{C^{0,0,\delta-m-1} (\mathbb{R}^n)},
\end{eqnarray*}
and so
$$
   \| \varPhi_m f \|_ {C^{0,0,\delta} (\mathbb{R}^n)}
 \leq
   c\, \| f \|_{C^{0,\lambda,\delta+2} (\mathbb{R}^n)}.
$$

Finally, by construction $\varPhi_m f$ coincides with the potential $\varPhi f$ on
   $R^{0,\lambda,\delta+2} (\mathbb{R}^n)$
whence
$$
   \varDelta\, \varPhi_m f = f
$$
for all $f \in R^{0,\lambda,\delta+2,m} (\mathbb{R}^n)$.
Now Lemma \ref{l.apriori.Delta} implies that $\varPhi_m f$ maps
   $R^{0,\lambda,\delta+2} (\mathbb{R}^n)$ continuously into
   $C^{2,\lambda,\delta} (\mathbb{R}^n)$
for the corresponding $\delta$.

For $s = 2$ and $n = 2$ the proof is similar and follows the same scheme as in
   \cite[Lemma 6]{McOw79}
for the weighted Sobolev spaces.
\end{proof}

On summarising we have proved the assertion 2) for $s = 2$ and $n \geq 2$.

\begin{lemma}
\label{l.homo.map}
Suppose $P$ is a homogeneous partial differential operator $P$ of order $0 \leq k \leq s-2$
with constant coefficients.
Then,

1)
$P$ maps
   $C^{s-2,\lambda,\delta+2} (\mathbb{R}^n)$ continuously into
   $R^{s-2-k,\lambda,\delta+2+k} (\mathbb{R}^n)$,
provided that
   $0 < \delta < n-2$ and
   $n-2+m < \delta+k < n-2+m+1$;

2)
$P$ maps
   $R^{s-2,\lambda,\delta+2} (\mathbb{R}^n)$ continuously into
   $R^{s-2-k,\lambda,\delta+2+k} (\mathbb{R}^n)$,
provided that
   $n-2+m < \delta < n-2+m+1$.
\end{lemma}

\begin{proof}
If $f \in C^{s-2,\lambda,\delta+2} (\mathbb{R}^n)$, where
   $0 < \delta < n\!-\!2$ and
   $n\!-\!2\!+\!m < \delta\!+\!k < n\!-\!2\!+\!m\!+\!1$,
then $m < k$ and, given any $h \in H_m$, we use the Green formula for $P$ to get
\begin{eqnarray*}
   \int_{\mathbb R^n} (Pf) (x) h (x) dx
 & = &
   \lim_{R \to + \infty}
   \int_{B_R} (Pf)(x) h (x) dx
\\
 & = &
   \int_{\mathbb R^n} f (x) (P^\ast h)(x) dx
 + \lim_{R \to + \infty} \int_{\partial B_R} G_P (h,f)
\\
 & = &
   0,
\end{eqnarray*}
because  $P^\ast h = 0$ (for $m < k$!) and the modulus of
$
   \displaystyle
   \int_{\partial B_R} G_P (h,f)
$
is dominated by
$$
   \sum_{j=0}^{k-1} \frac{R^{n-1} R^{m-j}}{(1+|R|^2)^{(\delta+2+k-1-j)/2}}.
$$
(Here, by $G_P (\cdot, \cdot)$ is meant a Green operator for $P$.)

Similarly, if $f \in R^{s-2,\lambda,\delta+2} (\mathbb{R}^n)$, where
   $n-2+m < \delta < n-2+m+1$,
then using the Green formula for $P$ yields
\begin{eqnarray*}
   \int_{\mathbb R^n} (Pf) (x) h (x) dx
 & = &
   \lim_{R \to + \infty} \int_{B_R} (Pf) (x) h (x) dx
\\
 & = &
   \int_{\mathbb R^n} f (x) (P^\ast h) (x) dx
 + \lim_{R \to + \infty} \int_{\partial B_R} G_P (h,f)
\\
 & = &
   0
\end{eqnarray*}
for all $h \in H_{\leq m+k}$, because
   $P^\ast h \in H_{\leq m}$
(for the operators $P^\ast$ and $\varDelta$ commute) and
$$
   \Big| \int_{\partial B_R} G_P (h,f) \Big|
 \leq
   c
   \sum_{j=0}^{k-1} \frac{R^{n-1} R^{m+k-j}}{(1+|R|^2)^{(\delta+2+k-1-j)/2}}.
$$
\end{proof}

Finally, we may further argue by induction because on integrating by parts for
   $f \in C^{s-2,\lambda,\delta+2} (\mathbb{R}^n)$
we get
$$
   \partial^{\alpha} \int_{\mathbb{R}^n} \frac{f (y)}{|x-y|^{n-2}} dy
 = \int_{\mathbb{R}^n} \frac{\partial^{\alpha} f (y)}{|x-y|^{n-2}} dy
$$
whenever $\delta > 0$ and $|\alpha| \leq s-2$.
Indeed, if
   $0 < \delta < n-2$ and
   $f \in C^{s-2,\lambda,\delta+2} (\mathbb{R}^n)$
or
   $n-2+m < \delta < n-2+m+1$ and
   $f \in R^{s-2,\lambda,\delta+2} (\mathbb{R}^n)$,
then also
   $\partial^\alpha f$ is of class $R^{s-2-|\alpha|,\lambda,\delta+2+|\alpha|} (\mathbb{R}^n)$
and
$$
   \partial^\alpha (\varPhi f)
 = \varPhi (\partial^\alpha f)
 = \varPhi_m (\partial^\alpha f),
$$
which is due to the properties of $\varPhi$ and $\varPhi_m$ derived above.
(Obviously, we have $\varDelta \partial^\alpha (\varPhi f) = \partial^\alpha f$.)
In particular,
   $\partial^\alpha (\varPhi f) \in C^{0,0,\delta+|\alpha|} (\mathbb{R}^n)$
for all $|\alpha| \leq s-2$ according to Lemma \ref{l.Gm}.
Now using Lemma \ref{l.apriori.Delta} we see that
   $\partial^\alpha (\varPhi f) \in  C^{2,\lambda,\delta+|\alpha|} (\mathbb{R}^n)$
for all $|\alpha| \leq s-2$, and so
   $\varPhi f \in C^{s,\lambda,\delta} (\mathbb{R}^n)$
satisfies $\varDelta (\varPhi f) = f$.

We have thus proved the assertion 2) of the theorem for $n \geq 3$ and $s \geq 2$.
For $s \geq 2$ and $n = 2$ the proof is similar.
\end{proof}

What about the case $\delta < 0$? 
In the case of Sobolev space a duality argument might be used.
Actually we do not consider $\delta < 0$ below.
For handling the Navier-Stokes equations we need merely $\delta > 0$ if we want to provide
a finite energy estimate.

Now we start to study the Laplace operator in the scale
   $C^{k, \mathbf{s} (s,\lambda,\delta)} (\overline{\mathcal{C}_T})$.
As the Laplace operator is not fully consistent with the dilation principle in
   $\overline{\mathcal{C}_T} $
we should expect some loss of regularity of solutions to $\varDelta u = f$ in this scale of
function spaces.

Similarly to the scale $C^{s,\lambda,\delta}$, we will use the potential
$$
   (\varPhi \otimes I) f\, (x,t)
 = \int_{\mathbb{R}^n} \phi (x-y) \, f (y,t) dy
$$
for function $f$ defined on $\overline{\mathcal{C}_T} $.
The variable $t$ enters into the integral as a parameter and the pair $(x,t)$ is assumed
to be in the finite layer $\overline{\mathcal{C}_T} $ over $\mathbb{R}^n$.

Actually, we can easily extend  Theorem \ref{t.weight.Hoelder.Laplace} to the Laplace
operator acting boundedly as
\begin{equation}
\label{eq.Laplace.easy}
   \varDelta:\,
   \bigcap_{j=0}^s C^{j,0} ([0,T],C^{2 (s-j)+k+2,\lambda,\delta} (\mathbb{R}^n))
\to
   \bigcap_{j=0}^s C^{j,0} ([0,T],C^{2 (s-j)+k,\lambda,\delta+2} (\mathbb{R}^n)).
\end{equation}

Let $C^s ([0,T],H_{\leq m})$ be the space of all $C^s$ functions of $t \in [0,T]$ with
values in the harmonic polynomials of degree $\leq m$ in $x$.
Any element $h (x,t)$ can be alternatively thought of as a polynomial of $H_{\leq m}$ whose
coefficients are $C^s$ functions on $[0,T]$.

\begin{lemma}
\label{l.weight.Hoelder.Laplace.easy}
Let
   $n \geq 2$,
   $s$, $k$ be nonnegative integers,
   $0 < \lambda < 1$ and
   $\delta > 0$
The operator (\ref{eq.Laplace.easy}) has closed range unless
   $\delta+2-n \in \mathbb{Z}_{\geq 0}$.
Moreover,

1)
it is an isomorphism, if $0 < \delta < n-2$;

2)
it is an injection, if $n-2+m < \delta < n-1+m$ for some $m \in \mathbb{Z}_{\geq 0}$,
and its range consists of all
$$
   f \in \bigcap_{j=0}^s C^{j,0} ([0,T],C^{2 (s-j)+k,\lambda,\delta+2} (\mathbb{R}^n))
$$
satisfying
$
   \displaystyle
   \int _{\mathbb{R}^n} f (x,t) h (x) dx = 0
$
for all $h \in H_{\leq m}$.
%
\end{lemma}

\begin{proof}
First we note that, by the Liouville theorem, a harmonic function on $\mathbb{R}^n$ whose
growth at the infinity point does not exceed that of $|x|^m$ is a polynomial of degree
$m \in \mathbb{Z}_+$.
Hence, for $\delta > 0$, the kernel of operator (\ref{eq.Laplace.easy}) consists of those
functions
$$
   u \in \bigcap_{j=0}^s C^{j,0} ([0,T],C^{2 (s-j)+k+2,\lambda,\delta} (\mathbb{R}^n))
$$
which depend on the variable $t$ only and vanish as $|x| \to +\infty $.
This means that operator (\ref{eq.Laplace.easy}) is injective for all $s \geq 0$ and $\delta > 0$.

%
Finally, using Theorem \ref{t.weight.Hoelder.Laplace} yields
\begin{eqnarray*}
\lefteqn{
   \sup_{t \in [0,T]}
   \| \partial^j_t (\varPhi \otimes I) f (\cdot, t)
   \|_{C^{2(s-j)+k+2,\lambda,\delta} (\mathbb{R}^n)}
}
\\
 & = &
   \sup_{t \in [0,T]}
   \| (\varPhi \otimes I) \partial^j_t f (\cdot, t) \|_{C^{2(s-j)+k+2,\lambda,\delta} (\mathbb{R}^n)}
\\
 & \leq &
   c\,
   \sup_{t \in [0,T]}
   \| \partial_t^j f (\cdot, t) \|_{C^{2(s-j)+k,\lambda,\delta+2} (\mathbb{R}^n)}
\end{eqnarray*}
for all $0 \leq j \leq s$ and appropriate $\delta$ and $f$.
\end{proof}

By Lemma \ref{l.diff.oper}, the Laplace operator induces a bounded linear operator
$$
   \varDelta :\,
   C^{k+2,\mathbf{s} (s,\lambda,\delta)} (\overline{\mathcal{C}_T})
 \to
   C^{k,\mathbf{s} (s,\lambda,\delta+2)} (\overline{\mathcal{C}_T}).
$$
However, we are also aimed at describing the action of the potential $\varPhi \otimes I$ on the
``parabolic'' H\"older spaces.
To this end, we introduce
   $C^{k+1,\mathbf{s} (s,\lambda,\delta)} (\overline{\mathcal{C}_T}) \cap \mathcal{D}_\varDelta$
to be the space of all functions $u$ from
   $C^{k+1,\mathbf{s} (s,\lambda,\delta)} (\overline{\mathcal{C}_T})$
with the property that
   $\varDelta u \in C^{k,\mathbf{s} (s,\lambda,\delta+2)} (\overline{\mathcal{C}_T})$.
We endow this space with the so-called graph norm
$$
   \| u \|_{C^{k+1,\mathbf{s} (s,\lambda,\delta)} (\overline{\mathcal{C}_T}) \cap \mathcal{D}_\varDelta}
 = \| u \|_{C^{k+1,\mathbf{s} (s,\lambda,\delta)} (\overline{\mathcal{C}_T})}
 + \| \varDelta u \|_{C^{k,\mathbf{s} (s,\lambda,\delta+2)} (\overline{\mathcal{C}_T})}.
$$

Let $C^{s,\lambda} ([0,T], H_{\leq m})$ stand for the space of all $C^{s,\lambda}$ functions
of $t \in [0,T]$ with values in the harmonic polynomials of degree $\leq m$ with respect to
the variable $x \in \mathbb{R}^n$.

\begin{corollary}
\label{c.weight.Hoelder.Laplace.t}
Suppose that
   $n \geq 2$,
   $k$ and $s$ are nonnegative integers,
   $0 < \lambda < 1$
and
   $\delta > 0$, 
   $\delta+2-n \not\in \mathbb{Z}_{\geq 0}$.
Then
   $C^{k+1,\mathbf{s} (s,\lambda,\delta)} (\overline{\mathcal{C}_T}) \cap \mathcal{D}_\varDelta$
is a Banach space and the Laplace operator $\varDelta$ induces a continuous linear operator
\begin{equation}
\label{eq.Laplace.non-coercive}
   \varDelta :\,
   C^{k+1,\mathbf{s} (s,\lambda,\delta)} (\overline{\mathcal{C}_T}) \cap \mathcal{D}_\varDelta
 \to
   C^{k,\mathbf{s} (s,\lambda,\delta+2)} (\overline{\mathcal{C}_T}).
\end{equation}
with closed range.
Moreover,

1)
it is an isomorphism, if $0 < \delta < n-2$;

2)
if $n-2+m < \delta < n -1 + m$ for some $m \in \mathbb{Z}_{\geq 0}$, then it is an injection
and its range $R^{k,\mathbf{s} (s,\lambda,\delta+2)} (\overline{\mathcal{C}_T})$ consists of those
   $f \in C^{k,\mathbf{s} (s,\lambda,\delta+2)} (\overline{\mathcal{C}_T})$
which satisfy
$$
   \int_{\mathbb{R}^n} f (x,t) h (x) dx = 0
$$
for all $h \in H_{\leq m}$.
%
\end{corollary}

\begin{proof}
We first note that the elements of
   $C^{k+1,\mathbf{s} (s,\lambda,\delta)} (\overline{\mathcal{C}_T}) \cap \mathcal{D}_\varDelta$
are $C^{k+2+2s,\lambda}_{\mathrm{loc}}$ functions of $x \in \mathbb{R}^n$,
   which is due to elliptic regularity.
If
   $\{ u_\nu \}$
is a Cauchy sequence in
   $C^{k+1,\mathbf{s} (s,\lambda,\delta)} (\overline{\mathcal{C}_T}) \cap \mathcal{D}_\varDelta$,
then it is a Cauchy sequence in
   $C^{k+1,\mathbf{s} (s,\lambda,\delta)} (\overline{\mathcal{C}_T})$
and
   $\{ \varDelta u_\nu \}$
is a Cauchy sequence in the space
   $C^{k,\mathbf{s} (s,\lambda,\delta+2)} (\overline{\mathcal{C}_T})$.
As the spaces are complete we conclude that the sequence $\{ u_\nu \}$ converges in
   $C^{k+1,\mathbf{s} (s,\lambda,\delta)} (\overline{\mathcal{C}_T})$
to an element $u$ and the sequence $\{ \varDelta u_\nu \} $ converges in
   $C^{k,\mathbf{s} (s,\lambda,\delta+2)} (\overline{\mathcal{C}_T})$
to an element $f$.
Obviously, $\varDelta u = f$ is fulfilled in the sense of distributions.
Hence, $u$ belongs to
   $C^{k+1,\mathbf{s} (s,\lambda,\delta)} (\overline{\mathcal{C}_T}) \cap \mathcal{D}_\varDelta$
and it is the limit of the sequence $\{ u_\nu \} $ in this space.
We have thus proved that the space
   $C^{k+1,\mathbf{s} (s,\lambda,\delta)} (\overline{\mathcal{C}_T}) \cap \mathcal{D}_\varDelta$
is Banach.
Moreover, by the very definition of the space, the Laplace operator $\varDelta$ induces a
continuous linear operator as is shown in (\ref{eq.Laplace.non-coercive}).

Next, the Liouville theorem implies that a harmonic function in $\mathbb{R}^n$ growing as
$|x|^m$ at the point of infinity is  a polynomial of order $m \in \mathbb{Z}_{\geq 0}$.
It follows that for $\delta > 0$ the kernel of $\varDelta$ consists of those functions
   $u \in C^{k+1,\mathbf{s} (s,\lambda,\delta)} (\overline{\mathcal{C}_T}) \cap \mathcal{D}_\varDelta$
which depend on the variable $t$ only and vanish as $|x| \to +\infty$.
This means that the operator $\varDelta$ is injective on
   $C^{k+1,\mathbf{s} (s,\lambda,\delta)} (\overline{\mathcal{C}_T}) \cap \mathcal{D}_\varDelta$,
provided that
   $s \geq 0$ and
   $\delta > 0$.
%
%

Now we need an analogue of Lemmata \ref{l.phi.Hoelder} and \ref{l.Gm}.

\begin{lemma}
\label{l.weight.Hoelder.Laplace.t}
Assume
   $n \geq 2$,
   $k$ and $s$ are nonnegative integers,
   $0 < \lambda < 1$
and
   $\delta > 0$, 
   $\delta+2-n \not\in \mathbb{Z}_{\geq 0}$.
Then the potential $\varPhi \otimes I$ induces a bounded linear operator
$$
\begin{array}{rclll}
   C^{k,\mathbf{s} (s,\lambda,\delta+2)} (\overline{\mathcal{C}_T})
 & \!\! \to \!\!
 & C^{k+1,\mathbf{s} (s,\lambda,\delta)} (\overline{\mathcal{C}_T}) \cap \mathcal{D}_\varDelta,
 & \mbox{if}
 & 0 < \delta < n-2,
\\
   R^{k,\mathbf{s} (s,\lambda,\delta+2)} (\overline{\mathcal{C}_T})
 & \!\! \to \!\!
 & C^{k+1,\mathbf{s} (s,\lambda,\delta)} (\overline{\mathcal{C}_T}) \cap \mathcal{D}_\varDelta,
 & \mbox{if}
 & n-2 < \delta, 
\end{array}
$$
satisfying $\varDelta (\varPhi \otimes I) = I$ on these spaces.
\end{lemma}

\begin{proof}
Indeed, by Theorem  \ref{t.weight.Hoelder.Laplace} we get
\begin{eqnarray*}
\lefteqn{
   \sup_{t',t'' \in [0,T] \atop t' \neq t''}
   \frac{ \| (\varPhi \otimes I) \partial^j_t f\, (\cdot, t') - (\varPhi \otimes I) \partial^j_t f\, (\cdot, t'')
          \|_{C^{2 (s-j)+k+1,0,\delta} (\mathbb{R}^n)}}
        {|t'-t''|^{\lambda/2}}
}
\\
 & = &
   \sup_{t',t'' \in [0,T] \atop t' \neq t''}
   \frac{ \| \varPhi (\partial^j_t f (\cdot, t') - \partial^j_t f (\cdot, t''))
          \|_{C^{2 (s-j)+k+1,0,\delta} (\mathbb{R}^n)}}
        {|t'-t''|^{\lambda/2}}
\\
 & \leq &
   \sup_{t',t'' \in [0,T] \atop t' \neq t''}
   \frac{ \| \partial^j_t f (\cdot, t') - \partial^j_t f (\cdot, t'')
          \|_{C^{2 (s-j)+k-1,\lambda,\delta+2} (\mathbb{R}^n)}}
        {|t'-t''|^{\lambda/2}}
\\
 & \leq &
   c\,
   \sup_{t',t'' \in [0,T] \atop t' \neq t''}
   \frac{ \| \partial^j_t f (\cdot, t') - \partial^j_t f (\cdot, t'')
          \|_{C^{2 (s-j)+k,0,\delta+2} (\mathbb{R}^n)}}
        {|t'-t''|^{\lambda/2}}
\end{eqnarray*}
for all $f$ in
   $C^{k,\mathbf{s} (s,\lambda,\delta+2)} (\overline{\mathcal{C}_T})$,
if $0 < \delta < n-2$, or in the range
   $R^{k,\mathbf{s} (s,\lambda,\delta\!+\!2)} (\overline{\mathcal{C}_T})$,
if $n+m-2 < \delta < n+m-1$,
   where $c$ is a constant granted by Lemma \ref{l.smooth.lipschitz.t}.

It is clear from Lemmata \ref{l.phi.Hoelder} and \ref{l.Gm} that
   $\varDelta (\varPhi \otimes I) f = f$
for all functions $f$ in
   $C^{k,\mathbf{s} (s,\lambda,\delta+2)} (\overline{\mathcal{C}_T})$,
if $0 < \delta < n-2$, or in
   $R^{k,\mathbf{s} (s,\lambda,\delta+2)} (\overline{\mathcal{C}_T})$,
if $n+m-2 < \delta < n+m-1$.

On combining what has been proved with the results of Lemma \ref{l.weight.Hoelder.Laplace.easy}
we obtain readily
$$
   \| (\varPhi \otimes I) f)
   \|_{C^{k+1,\mathbf{s} (s,\lambda,\delta)} (\overline{\mathcal{C}_T}) \cap \mathcal{D}_\varDelta}
 \leq
   c\, \| f \|_{C^{k,\mathbf{s} (s,\lambda,\delta+2} (\overline{\mathcal{C}_T})}
$$
for all function $f$ as above, with $c$ a constant independent of $f$.
\end{proof}

On applying this lemma we complete readily the proof of Corollary \ref{c.weight.Hoelder.Laplace.t},
which has been our goal.
\end{proof}

\begin{lemma}
\label{l.homo.map.t}
Suppose that $P$ is a homogeneous partial differential operator of order $0 \leq k' \leq k$
with constant coefficients which acts in the space variable $x$.
Then,

1)
$P$ maps
   $C^{k,\mathbf{s} (s,\lambda,\delta\!+\!2)} (\mathbb{R}^n \! \times \! [0,T])$ continuously into
   $R^{k\!-\!k',\mathbf{s} (s,\lambda,\delta\!+\!2\!+\!k')} (\mathbb{R}^n \! \times \! [0,T])$,
if
   $0 < \delta < n-2$ and
   $n-2+m < \delta+k' < n-1+m$ for some $m \in \mathbb{Z}_{\geq 0}$;

2)
$P$ maps
   $R^{k,\mathbf{s} (s,\lambda,\delta\!+\!2)} (\mathbb{R}^n \! \times \! [0,T])$ continuously into
   $R^{k\!-\!k',\mathbf{s} (s,\lambda,\delta\!+\!2\!+\!k')} (\mathbb{R}^n \! \times \! [0,T])$,
if
   $n-2+m < \delta < n-1+m$ for some $m \in \mathbb{Z}_{\geq 0}$.
\end{lemma}

\begin{proof}
Indeed, if
   $f \in C^{k,\mathbf{s} (s,\lambda,\delta+2)} (\overline{\mathcal{C}_T})$
and
$$
\begin{array}{rcccl}
   0
 & <
 & \delta
 & <
 & n-2,
\\
   n-2+m
 & <
 & \delta+k'
 & <
 & n-1+m
\end{array}
$$
for some $m \in \mathbb{Z}_{\geq 0}$, then
   $m < k'$
and, for any $h \in H_{\leq m}$, using the Green formula for $P$ yields
\begin{eqnarray*}
\lefteqn{
   \int_{\mathbb R^n} (Pf) (x,t) h (x) dx
}
\\
 & = &
   \lim_{R \to + \infty} \int_{B_R} (Pf)(x,t) h (x) dx
\\
 & = &
   \int_{\mathbb R^n} f (x,t) (P^\ast h) (x) dx
 + \lim_{R \to + \infty} \int_{\partial B_R} G_P (h,f)
\\
 & = &
   0,
\end{eqnarray*}
the last equality being a consequence of the facts that
   $P^\ast h = 0$, for $m < k'$,
and
$$
   \Big| \int_{\partial B_R} G_P (h,f) \Big|
 \leq
   c \sum_{j=0}^{k'-1} \frac{R^{n-1} R^{m-j}}{(1+|R|^2)^{(\delta+2+k'-1-j)/2}}
$$
with some constant $c$ independent of $R$.

Similarly, if
   $f \in R^{k,\mathbf{s} (s,\lambda,\delta+2)} (\overline{\mathcal{C}_T})$
and
   $n-2+m < \delta < n-1+m$,
then, for any $h \in H_{\leq m+k'}$, using the Green formula for $P$ we obtain
\begin{eqnarray*}
\lefteqn{
   \int_{\mathbb R^n} (Pf) (x,t) h (x) dx
}
\\
 & = &
   \lim_{R \to + \infty} \int_{B_R} (Pf) (x,t) h (x) dx
\\
 & = &
   \int_{\mathbb R^n} f (x,t) (P^\ast h) (x) dx
 + \lim_{R \to + \infty}  \int_{\partial B_R} G_P (h,f)
\\
 & = &
   0
\end{eqnarray*}
because $P^\ast h \in H_{\leq m}$ (since $\varDelta P = P \varDelta$) and
$$
   \Big| \int_{\partial B_R} G_P (h,f) \Big|
 \leq
   c \sum_{j=0}^{k'-1} \frac{R^{n-1} R^{m+k'-j}}{(1+|R|^2)^{(\delta+2+k'-1-j)/2}},
$$
the constant $c$ being independent of $R$.
\end{proof}

As a corollary we are in a position to describe the behaviour of the de Rham cohomolgy in
the scale of weighted H\"older spaces.

For this purpose, for a differential operator $A$ acting on sections of the vector bundle
$\varLambda^q$ over $\mathbb{R}^n$, we denote by
   $C^{s,\lambda,\delta} (\mathbb{R}^n, \varLambda^q) \cap \mathcal{S}_A$
the space of all differential forms
   $u \in C^{s,\lambda,\delta} (\mathbb{R}^n, \varLambda^q)$
satisfying $Au = 0$ in the sense of the distributions in $\mathbb{R}^n$.
Similarly, we write
   $C^{k,\mathbf{s} (s,\lambda,\delta)} (\overline{\mathcal{C}_T}, \varLambda^q) \cap \mathcal{S}_A$
for the space of all $q\,$-forms on $\mathbb{R}^n$ with coefficients from
   $C^{k,\mathbf{s} (s,\lambda,\delta)} (\overline{\mathcal{C}_T})$
satisfying $Au\, (\cdot,t) = 0$ in the sense of distributions for all fixed $t \in [0,T]$.
These spaces are obviously closed subspaces of
   $C^{s,\lambda,\delta} (\mathbb{R}^n, \varLambda^q)$ and
   $C^{k,\mathbf{s} (s,\lambda,\delta)} (\overline{\mathcal{C}_T}, \varLambda^q)$,
respectively, and so they are Banach spaces under induced norms.

\begin{corollary}
\label{c.deRham.Hoelder}
Let
   $q \geq 0$,
   $s \in \mathbb{Z}_{\geq 0}$,
   $0 < \lambda < 1$,
and
   $\delta > 0$ satisfy $\delta+2-n \not\in \mathbb{Z}_{\geq 0}$.
If
   $f \in R^{s,\lambda,\delta+2} (\mathbb{R}^n, \varLambda^{q+1}) \cap \mathcal{S}_d$,
then  there is a unique
   $u \in C^{s+1,\lambda,\delta+1} (\mathbb{R}^n, \varLambda^q) \cap \mathcal{S}_{d^\ast}$,
such that
\begin{equation}
\label{eq.deRham.sol}
   du = f
\end{equation}
in $\mathbb{R}^n$.
Moreover,
$$
   \| u \|_{C^{s+1,\lambda,\delta+1} (\mathbb{R}^n, \varLambda^q)}
 \leq
   c\,
   \| f \|_{C^{s,\lambda,\delta+2} (\mathbb{R}^n, \varLambda^{q+1})}
$$
with $c$ a constant independent of $f$, and if $0 < \delta < n$ then the solution $u$ belongs
to
   $R^{s+1,\lambda,\delta+1} (\mathbb{R}^n, \varLambda^q)$.
\end{corollary}

\begin{proof}
Indeed, as $f$ belongs to $R^{s,\lambda,\delta+2} (\mathbb{R}^n, \varLambda^{q+1})$,
Theorem \ref{t.weight.Hoelder.Laplace} and (\ref{eq.deRham}) imply that
$$
\begin{array}{rcl}
   \varPhi f
 & \in
 & C^{s+2,\lambda,\delta} (\mathbb{R}^n, \varLambda^{q+1}),
\\
   d^\ast \varPhi f
 & \in
 & C^{s+1,\lambda,\delta+1} (\mathbb{R}^n, \varLambda^q) \cap \mathcal{S}_{d^\ast},
\\
   d d^\ast \varPhi f
 & \in
 & C^{s,\lambda,\delta+2} (\mathbb{R}^n, \varLambda^{q+1}).
\end{array}
$$

On the other hand,
$$
   \varDelta (d d^\ast\, \varPhi f - f)
 = d d^\ast (f - f)
 = 0
$$
in the sense of distributions on $\mathbb{R}^n$.
Since $\delta > 0$, the operator $\varDelta$ is injective on
   $C^{s,\lambda,\delta+2} (\mathbb{R}^n, \varLambda^{q+1})$,
and hence $d d^\ast \varPhi f = f$.

If $f = 0$ then the solution to (\ref{eq.deRham.sol}) is harmonic in $\mathbb{R}^n$
because of (\ref{eq.deRham}).
It follows that the solution vanishes in all of $\mathbb{R}^n$ by the injectivity of
the operator $\varDelta$ on the space
   $C^{s+1,\lambda,\delta+1} (\mathbb{R}^n, \varLambda^q$,
for $\delta > 0$.

We may now use the Banach closed graph theorem to deduce that the bounded linear operator
$$
   d :\, C^{s+1,\lambda,\delta+1} (\mathbb{R}^n, \varLambda^q) \cap \mathcal{S}_{d^\ast}
     \to C^{s,\lambda,\delta+2} (\mathbb{R}^n, \varLambda^{q+1}) \cap \mathcal{S}_d
$$
is continuously invertible.

If $0 < \delta < n-1$, then $1 < \delta+1 < n$ and so
$$
   u = d^\ast \varPhi f
     \in C^{s+1,\lambda,\delta+1} (\mathbb{R}^n, \varLambda^q)
$$
belongs actually to the range of the Laplace operator by Theorem \ref{t.weight.Hoelder.Laplace}.
Finally, if $n-1 < \delta < n$, then, since
   $\varPhi f \in C^{s+2,\lambda,\delta} (\mathbb{R}^n, \varLambda^{q+1})$
belongs to the range of the operator
$$
   \varDelta :\, C^{s+4,\lambda,\delta-2} (\mathbb{R}^n, \varLambda^{q+1}) \to
              C^{s+2,\lambda,\delta} (\mathbb{R}^n, \varLambda^{q+1}),
$$
on using Lemma \ref{l.homo.map} we conclude that
   $u = d^\ast \varPhi f \in R^{s+1,\lambda,\delta+1} (\mathbb{R}^n, \varLambda^q)$.
\end{proof}

Denote by
   $C^{k,\mathbf{s} (s,\lambda,\delta+1)} (\overline{\mathcal{C}_T},\varLambda^q) \cap \mathcal{D}_d$
the space of all differential forms
   $u \in C^{k,\mathbf{s} (s,\lambda,\delta+1)} (\overline{\mathcal{C}_T},\varLambda^q)$,
such that
   $du \in C^{k,\mathbf{s} (s,\lambda,\delta+2)} (\overline{\mathcal{C}_T},\varLambda^{q+1})$,
endowed with the graph norm
\begin{eqnarray*}
\lefteqn{
   \| u \|_{C^{k,\mathbf{s} (s,\lambda,\delta+1)} (\overline{\mathcal{C}_T},\varLambda^q) \cap \mathcal{D}_d}
}
\\
 & = &
   \| u \|_{C^{k,\mathbf{s} (s,\lambda,\delta+1)} (\overline{\mathcal{C}_T},\varLambda^q)}
 + \| du \|_{C^{k,\mathbf{s} (s,\lambda,\delta+2)} (\overline{\mathcal{C}_T},\varLambda^{q+1})}.
\end{eqnarray*}
As in the proof of Corollary \ref{c.weight.Hoelder.Laplace.t} one shows that
   $C^{k,\mathbf{s} (s,\lambda,\delta+1)} (\overline{\mathcal{C}_T},\varLambda^q) \cap \mathcal{D}_d$
is a Banach space and the differential $d$ induces a bounded linear operator
$$
   d_q :\,
   C^{k,\mathbf{s} (s,\lambda,\delta+1)} (\overline{\mathcal{C}_T} ,\varLambda^q) \cap \mathcal{D}_d \to
   C^{k,\mathbf{s} (s,\lambda,\delta+2)} (\overline{\mathcal{C}_T} ,\varLambda^{q+1}).
$$

\begin{corollary}
\label{c.deRham.Hoelder.t}
Let
   $q \geq 0$,
   $s \in \mathbb{Z}_{\geq 0}$,
   $k = 1, 2, \ldots$,
   $0 < \lambda < 1$,
and let
   $ \delta > 0$ satisfy $\delta+2-n \not\in \mathbb{Z}_{\geq 0}$.
Then, for each
   $f \in R^{k,\mathbf{s} (s,\lambda,\delta+2)} (\overline{\mathcal{C}_T}, \varLambda^{q+1})
     \cap \mathcal{S}_d$
there is a unique differential form
   $u \in C^{k,\mathbf{s} (s,\lambda,\delta+1)} (\overline{\mathcal{C}_T}, \varLambda^q) \cap \mathcal{D}_d$
with the property that
\begin{equation}
\label{eq.deRham.sol.t}
\begin{array}{rcl}
   du (\cdot,t)
 & =
 & f (\cdot,t),
\\
   d^\ast u (\cdot,t)
 & =
 & 0
\end{array}
\end{equation}
in $\overline{\mathcal{C}_T} $.
Moreover,
$$
   \| u \|_{C^{k,\mathbf{s} (s,\lambda,\delta+1)} (\overline{\mathcal{C}_T}, \varLambda^q) \cap \mathcal{D}_d}
 \leq
   c\,
   \| f \|_{C^{k,\mathbf{s} (s,\lambda,\delta+2)} (\overline{\mathcal{C}_T}, \varLambda^{q+1})}
$$
with $c$ a constant independent of $f$, and if $0 < \delta < n$ then the solution $u$ belongs to
   $R^{k,\mathbf{s} (s,\lambda,\delta+1)} (\overline{\mathcal{C}_T}, \varLambda^q)$.
\end{corollary}

\begin{proof}
It runs in much the same way as the proof of Corollary \ref{c.deRham.Hoelder}.
As $f$ belongs to
   $R^{k,\mathbf{s} (s,\lambda,\delta+2)} (\overline{\mathcal{C}_T} , \varLambda^{q+1})$,
Lemma \ref{l.weight.Hoelder.Laplace.t} and (\ref{eq.deRham}) imply that
$$
\begin{array}{rcl}
   (\varPhi \otimes I) f
 & \in
 & C^{k,\mathbf{s} (s,\lambda,\delta)} (\overline{\mathcal{C}_T} , \varLambda^{q+1}) \cap \mathcal{D}_{\varDelta},
\\
   d^\ast (\varPhi \otimes I) f
 & \in
 & C^{k,\mathbf{s} (s,\lambda,\delta+1)} (\overline{\mathcal{C}_T} , \varLambda^q) \cap \mathcal{S}_{d^\ast},
\\
   d d^\ast (\varPhi \otimes I) f
 & \in
 & C^{k-1,\mathbf{s} (s,\lambda,\delta+2)} (\overline{\mathcal{C}_T} , \varLambda^{q+1}).
\end{array}
$$

On the other hand, we obtain
$$
   \varDelta (d d^\ast (\varPhi \otimes I) f (\cdot,t) - f(\cdot,t))
 = d d^\ast (f - f) (\cdot,t)
 = 0
$$
in the sense of distributions on $\mathbb{R}^n$, for each $t \in [0,T]$.
Since $\delta > 0$, the operator $\varDelta$ is injective on
   $C^{k-1,\mathbf{s} (s,\lambda,\delta+2)} (\overline{\mathcal{C}_T} , \varLambda^{q+1})$,
and so $d d^\ast (\varPhi \otimes I) f (\cdot,t) = f (\cdot, t)$ for all $t \in [0,T]$.
In particular, $d^\ast (\varPhi \otimes I) f$ belongs to
   $C^{k,\mathbf{s} (s,\lambda,\delta+1)} (\overline{\mathcal{C}_T} , \varLambda^q) \cap \mathcal{D}_d$,
as is easy to check.

If $f (\cdot,t) = 0$, then the solution to (\ref{eq.deRham.sol.t}) is harmonic in $\mathbb{R}^n$
because of (\ref{eq.deRham}), for each  $t \in [0,T]$.
Therefore, it is identically zero by the injectivity of the operator $\varDelta$ on the space
    $C^{k,\mathbf{s} (s,\lambda,\delta)} (\overline{\mathcal{C}_T} , \varLambda^q)$,
for $\delta > 0$.

Thus, the bounded linear operator
$$
   d :\, C^{k,\mathbf{s} (s,\lambda,\delta+1)} (\overline{\mathcal{C}_T} , \varLambda^q) \cap \mathcal{D}_d
     \to C^{k,\mathbf{s} (s,\lambda,\delta+2)} (\overline{\mathcal{C}_T} , \varLambda^{q+1}) \cap \mathcal{S}_d
$$
restricted to the forms satisfying $d^\ast u = 0$ is continuously invertible by the closed graph
theorem.

If $0 < \delta < n-1$, then $1 < \delta+1 < n$ and
$$
   u := d^\ast (\varPhi \otimes I) f
 \in C^{k,\mathbf{s} (s,\lambda,\delta+1)} (\overline{\mathcal{C}_T} , \varLambda^q) \cap \mathcal{S}_{d^\ast}
$$
belongs to the range of operator (\ref{eq.Laplace.non-coercive}) by Corollary \ref{c.weight.Hoelder.Laplace.t}.
If $n-1 < \delta < n$, then, as
   $(\varPhi \otimes I) f \in C^{k+1,\mathbf{s} (s,\lambda,\delta)} (\overline{\mathcal{C}_T} , \varLambda^{q+1})$
is in the range of (\ref{eq.Laplace.non-coercive}) by Corollary \ref{c.weight.Hoelder.Laplace.t},
we use Lemma \ref{l.homo.map.t} to see that
$
   u := d^\ast (\varPhi \otimes I) f
 \in R^{k,\mathbf{s} (s,\lambda,\delta+1)} (\overline{\mathcal{C}_T} , \varLambda^q),
$
as desired.
\end{proof}

\begin{corollary}
\label{c.deRham.Hoelder.tt}
Suppose that
   $q \geq 0$,
   $s \in \mathbb{Z}_{\geq 0}$,
   $k$ is a positive integer,
   $0 < \lambda < 1$,
and
   $\delta > 0$ satisfies $\delta+2-n \not\in \mathbb{Z}_{\geq 0}$.
Then, for any
   $u \in R^{k,\mathbf{s} (s,\lambda,\delta)} (\overline{\mathcal{C}_T} , \varLambda^q)$
satisfying $d^\ast u = 0$, it follows that
$$
u = d^\ast (\varPhi \otimes I) du.
$$
\end{corollary}

\begin{proof}
Indeed, Lemma \ref{l.homo.map.t} yields
   $du \in R^{k-1,\mathbf{s} (s,\lambda,\delta+1)} (\overline{\mathcal{C}_T} , \varLambda^{q+1})$.
Use Lemma \ref{l.weight.Hoelder.Laplace.t} and (\ref{eq.deRham}) to see that
$$
   u
 = (\varPhi \otimes I) (d^\ast d + d d^\ast) u
 = (\varPhi \otimes I) d^\ast du
 \in C^{k,\mathbf{s} (s,\lambda,\delta)} (\overline{\mathcal{C}_T} , \varLambda^q)
$$
and
$
   d^\ast (\varPhi \otimes I) du
 \in R^{k-1,\mathbf{s} (s,\lambda,\delta)} (\overline{\mathcal{C}_T} , \varLambda^q).
$

On the other hand, by the same Lemma \ref{l.weight.Hoelder.Laplace.t}, we have
$$
   \varDelta \left( d^\ast (\varPhi \otimes I) du - (\varPhi \otimes I) d^\ast du \right)
 = d^\ast du - d^\ast du
 = 0.
$$
Finally, since the Laplacian $\varDelta$ is injective on
   $C^{k-1,\mathbf{s} (s,\lambda,\delta)} (\overline{\mathcal{C}_T} , \varLambda^{q+1})$,
we obtain
$$
   d^\ast (\varPhi \otimes I) du = (\varPhi \otimes I) d^\ast d u = u,
$$
as desired.
\end{proof}

\part{Open mapping theorem}
\label{p.omt}

\section{The heat operator in the weighted H\"older spaces}
\label{s.thoitwHs}

As usual, we denote by $\gamma_{t_0} u$ the restriction of a continuous function $u$ in the layer
$\overline{\mathcal{C}_T} $ to the hyperplane $\{ t = t_0 \}$ in $\mathbb{R}^{n+1}$, where
   $t_0 \in [0,T]$.
The following lemma is obvious.

\begin{lemma}
\label{l.bound.trace.holder}
Let $s, k \in \mathbb{Z}_{\geq 0}$ and $\lambda \in [0,1]$.
The restriction $\gamma_0$ induces a bounded linear operator
$$
   \gamma_0 :\, C^{k,\mathbf{s} (s,\lambda,\delta)} (\overline{\mathcal{C}_T})
            \to C^{2s+k,\lambda,\delta} (\mathbb{R}^n).
$$
\end{lemma}

Let $\psi_\mu$ be the standard fundamental solution of the convolution type to the heat operator
$H_\mu$ in $\mathbb{R}^{n+1}$, $n\geq 1$,
$$
   \psi_\mu (x,t)
 = \frac{\theta (t)}{\left(4  \pi \mu t\right)^{n/2}}\
   e^{-\frac{\scriptstyle |x|^2}{\scriptstyle 4 \mu t}},
$$
where $\theta (t)$ is the Heaviside function.
Denote by
\begin{eqnarray*}
   (\varPsi_\mu  f) (x,t)
 & = &
   \int_0^{t} \!\!\! \int _{\mathbb{R}^n} \psi_\mu (x-y, t-s)\, f (y,s)\, dy ds,
\\
   (\varPsi_{\mu,0}  u_0) (x,t)
 & = &
   \int_{\mathbb{R}^n} \psi_\mu (x-y, t)\, u_0 (y)\, dy
\end{eqnarray*}
the so-called volume parabolic potential and Poisson parabolic potentials, respectively, defined
for $(x,t) \in \overline{\mathcal{C}_T} $.

Consider the Cauchy problem for the heat operator in the weighted H\"older spaces.
Given functions
   $f$ in $\overline{\mathcal{C}_T} $ and
   $u _0$ on $\mathbb{R}^n$,
find a function $u$ in $\overline{\mathcal{C}_T} $, such that
\begin{equation}
\label{pr.Cauchy.heat}
\left\{ \begin{array}{rclll}
          H_\mu\, u (x,t)
        & =
        & f (x,t)
        & \mbox{for}
        & (x,t) \in \mathbb{R}^n \times (0,T),
\\
          \gamma_0 u\, (x,0)
        & =
        & u_0 (x)
        & \mbox{for}
        & x \in \mathbb{R}^n.
        \end{array}
\right.
\end{equation}

\begin{lemma}
\label{l.unique.heat.initial.hoelder}
For each real $\delta$, problem \ref{pr.Cauchy.heat} has at most one solution in  the space
$C^{\mathbf{s} (1,\lambda,\delta)} (\overline{\mathcal{C}_T})$.
\end{lemma}

\begin{proof}
The lemma follows, for instance, from Theorem 16 in \cite[Ch.~1, \S~9]{Fri64}.
There have been also much more advanced results.
\end{proof}

The solution of the Cauchy problem in weighted H\"{o}lder spaces is recovered from the data
by means of the Green formula.

\begin{lemma}
\label{l.bound.heat.initial.hoelder}
Assume that
   $\delta > 0$,
Then,  for each function $u \in C^{\mathbf{s} (1,\lambda,\delta)} (\overline{\mathcal{C}_T})$, it
follows that
$$
   u
 = \varPsi_\mu\, H_\mu u
 + \varPsi_{\mu,0}\, \gamma_0 u.
$$
\end{lemma}

\begin{proof}
See \em{ibid}.
\end{proof}

The following lemma is a main technical tool in estimating parabolic potentials in weighted
H\"{o}lder spaces.

\begin{lemma}
\label{l.heat.key0}
For each
   $\delta > 0$ and
   $\gamma > 0$
there is a positive constant $c$ depending on $\delta$, $\gamma$ and $T$, such that
$$
   \int_{\mathbb{R}^n}
   \Big( 1 + \frac{|x-y|^2}{4 \mu t} \Big)^{\gamma}
   \psi_{\mu} (x-y,t)\,
   \frac{dy}{(1+|y|^2)^{\delta/2}}\,
 \leq
   c\, (1+|x|^2)^{-\delta/2}
$$
for all $(x,t) \in \overline{\mathcal{C}_T} $.
\end{lemma}

\begin{proof}
First we note that, given any real number $r$, the function
$$
   f (s) = (1 + s)^{r} e^{-s/2}
$$
is bounded on $[0, \infty)$.
It follows that
\begin{eqnarray*}
\lefteqn{
   (1 + 2s)^{\gamma} e^{-s}
}
\\
 & \leq &
   2^\gamma
   \sup_{s \geq 0} \left( (1+s)^\gamma e^{-s/2} \right)
   \sup_{s \geq 0} \left( (1+s)^r e^{-s/2} \right)
   (1+s)^{-r}
\\
 & \leq &
   c\, (1+s)^{-r}
\end{eqnarray*}
for all $s \geq 0$, the constant $c$ depending only on $r$ and $\gamma$.
Taking
$
   \displaystyle
   s = \frac{|x-y|^2}{8 \mu t}
$
yields immediately
$$
   \Big( 1+ \frac{|x-y|^2}{4 \mu t} \Big)^\gamma
   e^{- \frac{\scriptstyle |x-y|^2}{\scriptstyle 8 \mu t}}
 \leq
   c\, \Big( 1 + \frac{|x-y|^2}{8 \mu t} \Big)^{-r}
$$
for all $t > 0$.

If $|x| \geq 2 |y|$ then
$$
   |x-y| \geq ||x|-|y||
         \geq ||x| - |x|/2|
         = |x|/2.
$$
Hence, on choosing $r = \delta/2$ we obtain
\begin{eqnarray*}
   \Big( 1 + \frac{|x-y|^2}{4 \mu t} \Big)^\gamma
   e^{- \frac{\scriptstyle |x-y|^2}{\scriptstyle 8 \mu t}} (1+|x|^2)^{\delta/2}
 & \leq &
   c\,
   \frac{(1+|x|^2)^{\delta/2}}
        {\left( 1 + \frac{\scriptstyle |x-y|^2}{\scriptstyle 8 \mu t} \right)^{\delta/2}}
\\
 & \leq &
   c\,
   \frac{(1+|x|^2)^{\delta/2}}
        {\left( 1 + \frac{\scriptstyle |x|^2}{\scriptstyle 32 \mu t} \right)^{\delta/2}}
\\
 & \leq &
   c\, (\max \{ 1, 32 \mu T \})^{\delta/2},
\end{eqnarray*}
and so
$$
   \int\limits_{\mathbb{R}^n} \!\!
   \Big( \! 1 + \frac{|x\!-\!y|^2}{4 \mu t} \! \Big)^{\gamma}
   \psi_{\mu} (x\!-\!y,t)\,
   \frac{(1\!+\!|x|^2)^{\delta/2}}{(1\!+\!|y|^2)^{\delta/2}}\, dy
 \leq
   c\, (\max \{ 1, 32 \mu T \})^{\delta/2} \!\!
   \int\limits_{\mathbb{R}^n} \!\!
   \frac{e^{- \frac{\scriptstyle |x\!-\!y|^2}{\scriptstyle 8 \mu t}}}{(4 \pi \mu t)^{n/2}}\, dy
$$
for all $(x,t) \in \overline{\mathcal{C}_T} $, where
$$
   \int_{\mathbb{R}^n}
   \frac{e^{- \frac{\scriptstyle |x-y|^2}{\scriptstyle 8 \mu t}}}{(4 \pi \mu t)^{n/2}}\, dy
 =
   \Big( \frac{2}{\pi} \Big)^{n/2}
   \int_{\mathbb{R}^n}
   e^{- \frac{\scriptstyle |x-y|^2}{\scriptstyle 8 \mu t}}\,
   d \frac{y}{\sqrt{8 \mu t}}
 = 2^{n/2}
$$
reduces to the so-called Gau{\ss} (-Euler-Poisson) integral.

Finally, if $|x|\leq 2 |y|$, then
\begin{eqnarray*}
   \int\limits_{\mathbb{R}^n} \!\!
   \Big( \! 1 + \frac{|x\!-\!y|^2}{4 \mu t} \! \Big)^{\gamma}
   \psi_{\mu} (x\!-\!y,t)\,
   \frac{(1\!+\!|x|^2)^{\delta/2}}{(1\!+\!|y|^2)^{\delta/2}}\, dy
 & \leq &
   2^{\delta/2}\, c
   \int\limits_{\mathbb{R}^n}
   \frac{e^{- \frac{\scriptstyle |x-y|^2}{\scriptstyle 8 \mu t}}}{(4 \pi \mu t)^{n/2}}\,
   dy
\\
 & = &
   2^{(\delta+n)/2}\, c,
\end{eqnarray*}
as is evaluated above.
\end{proof}

Denote by
   $C^{k,\mathbf{s} (s,\lambda,\delta)} (\overline{\mathcal{C}_T}) \cap \mathcal{D}_{H_\mu}$
the domain of the heat operator acting in
   $C^{k,\mathbf{s} (s,\lambda,\delta)} (\overline{\mathcal{C}_T})$.
This space is topologised and complete (i.e., Banach) under the graph norm
$$
   \| u \|_{C^{k,\mathbf{s} (s,\lambda,\delta)} (\overline{\mathcal{C}_T}) \cap \mathcal{D}_{H_\mu}}
 = \| u \|_{C^{k,\mathbf{s} (s,\lambda,\delta)} (\overline{\mathcal{C}_T})}
 + \| H_\mu u \|_{C^{k,\mathbf{s} (s,\lambda,\delta)} (\overline{\mathcal{C}_T})}.
$$

\begin{lemma}
\label{l.heat.key1}
Let
   $s, k \in \mathbb{Z}_{\geq 0}$,
   $0 < \lambda< 1$ and
   $\delta > 0$.
The parabolic potentials
   $\varPsi_\mu$ and
   $\varPsi_{\mu,0}$
induce bounded linear operators
$$
\begin{array}{rrcl}
   \varPsi_\mu :
 & C^{k,\mathbf{s} (s,\lambda,\delta)} (\overline{\mathcal{C}_T})
 & \to
 & C^{k,\mathbf{s} (s,\lambda,\delta)} (\overline{\mathcal{C}_T}) \cap \mathcal{D}_{H_\mu},
\\
   \varPsi_{\mu,0} :
 & C^{2s+k,\lambda,\delta} (\mathbb{R}^n)
 & \to
 & C^{k,\mathbf{s} (s,\lambda,\delta)} (\overline{\mathcal{C}_T}) \cap \mathcal{D}_{H_\mu}.
\end{array}
$$
\end{lemma}

\begin{proof}
We first prove the boundedness of the operator given by the Poisson parabolic potential.
It is well known that $\varPsi_{\mu,0}$ maps
   $C^{2s+k,\lambda,\delta} (\mathbb{R}^n)$ continuously into
   $C^{k,\mathbf{s} (s,\lambda,0)} _{\mathrm{loc}} (\mathbb{R}^n \times [0,\infty))$.

Suppose $|\alpha| \leq k+2s$.
Then
\begin{eqnarray*}
   \partial^\alpha_x (\varPsi_{\mu,0} u_0) (x,t)
 & = &
   \int_{\mathbb{R}^n} \psi_\mu (x-y,t) \partial^\alpha_{y} u_0 (y)\, dy
\\
 & = &
   \varPsi_{\mu,0} (\partial^\alpha u_0) (x,t)
\end{eqnarray*}
for all $t > 0$, which is due to the properties of convolution.
By Lemma \ref{l.heat.key0}, it follows that
\begin{eqnarray*}
\lefteqn{
   (w (x))^{\delta+|\alpha|} |\partial_x^\alpha (\varPsi_{\mu,0} u_0) (x,t)|
}
\\
 & \leq &
   \| \partial^\alpha  u_0 \|_{C^{0,\lambda,\delta+|\alpha|} (\mathbb{R}^n)}
   \int_{\mathbb{R}^n}
   \psi_\mu (x-y,t)
   \frac{(1+|x|^2)^{(\delta+|\alpha|)/2}}{(1+|y|^2)^{(\delta+|\alpha|)/2}}\,
   dy
\\
 & \leq &
   c\, \| \partial^\alpha  u_0 \|_{C^{0,\lambda,\delta+|\alpha|} (\mathbb{R}^n)}
\end{eqnarray*}
with $c$ a constant depending on $\alpha$ and $T$ but not on $u_0$.
Furthermore, we evaluate easily
\begin{eqnarray*}
\lefteqn{
   (w (x,y))^{\delta+|\alpha|+\lambda}\,
   \frac{|\partial_x^\alpha (\varPsi_{\mu,0} u_0) (x,t) - \partial_y^\alpha (\varPsi_{\mu,0} u_0) (y,t)|}
        {|x-y|^\lambda}
}
\\
 & = &
   (w (x,y))^{\delta+|\alpha|+\lambda}\,
   \Big| \int_{\mathbb{R}^n}
         \frac{\psi_\mu (x-z,t) - \psi_\mu (y-z,t)}{|x-y|^\lambda}\, \partial^\alpha u_0 (z)\, dz
   \Big|
\\
 & \leq &
   (w (x,y))^{\delta+|\alpha|+\lambda}\,
   \int_{\mathbb{R}^n}
   \psi_\mu (z,t)\,
   \frac{|\partial^\alpha u_0 (z+x) - \partial^\alpha u_0 (z+y) |}{|(z+x)-(z+y)|^\lambda}\, dz
\\
 & \leq &
   \| \partial^\alpha u_0 \|_{C^{0,\lambda,\delta+|\alpha|} (\mathbb{R}^n)}
   (w (x,y))^{\delta\!+\!|\alpha|\!+\!\lambda}
   \int_{\mathbb{R}^n}
   \psi_\mu (z,t)\, \frac{dz}{(w (x\!+\!z,y\!+\!z))^{\delta\!+\!|\alpha|\!+\!\lambda}}
\end{eqnarray*}
for all $x, y \in \mathbb{R}^n$ and $t > 0$.
On the other hand, applying Lemma \ref{l.heat.key0} we conclude that
\begin{eqnarray*}
   \int_{\mathbb{R}^n}
   \psi_\mu (z,t)\, (w (x+z))^{-(\delta+|\alpha|+\lambda)}\, dz
 & = &
   \int_{\mathbb{R}^n}
   \psi_\mu (x-z)\, (w (z))^{- (\delta+|\alpha|+\lambda)}\, dz
\\
 & \leq &
   c\, \frac{1}{(w (x))^{\delta+|\alpha|+\lambda}},
\end{eqnarray*}
where $c$ is a constant independent of $x$ and $t$.
Since
   $w (x+z,y+z)$ exceeds both $w (x+z)$ and $w (y+z)$ and
   $\delta + |\alpha| + \lambda \geq 0$,
this implies
\begin{eqnarray*}
   \int_{\mathbb{R}^n}
   \psi_\mu (z,t)\, (w (x+z,y+z))^{- (\delta+|\alpha|+\lambda)}\,
   dz
 & \leq &
   c\,
   \frac{1}{(\max \{ w (x), w (y) \})^{\delta+|\alpha|+\lambda}}
\\
 & \leq &
   c\, (w (x,y))^{-(\delta+|\alpha|+\lambda)},
\end{eqnarray*}
the last inequality being due to the fact that
   $w (x,y) \leq \sqrt{2}\, \max \{ w (x), w (y) \}$.
The constant $c$ depends neither on $x$, $y$ nor on $t \in [0,T]$ and it may be different
in diverse applications.

Our next concern will be the H\"{o}lder continuity of the potential $\varPsi_{\nu,0} u_0$ in $t$.
We get
\begin{eqnarray*}
\lefteqn{
   (w (x))^{\delta+|\alpha|}\,
   \frac{|\partial_x^\alpha (\varPsi_{\mu,0} u_0) (x,t')
       -  \partial_x^\alpha (\varPsi_{\mu,0} u_0) (x,t'')}{|t'-t''|^{\lambda/2}}|
}
\\
 & = &
   (w (x))^{\delta+|\alpha|}
   \Big| \int_{\mathbb{R}^n}
         \frac{\psi_\mu (x-z,t') - \psi_\mu (x-z,t'')}{|t'-t''|^{\lambda/2}}\,
         \partial^\alpha u_0 (z)\, dz
   \Big|
\\
 & = &
   (w (x))^{\delta+|\alpha|}
   \Big| \int_{\mathbb{R}^n}
         \frac{e^{- \frac{\scriptstyle |z|^2}{\scriptstyle 4 \mu}}}{(4 \pi \mu)^{n/2}}
         \frac{\partial^\alpha u_0 (x + z \sqrt{t'}) - \partial^\alpha u_0 (x + z \sqrt{t''})}
              {|t'-t''|^{\lambda/2}}
   \Big|,
\end{eqnarray*}
which is dominated by the
\begin{equation}
\label{eq.motn}
   (w (x))^{\delta+|\alpha|}
   \frac{|\sqrt{t'} - \sqrt{t''}|^\lambda}{|t'-t''|^{\lambda/2}}
   \int_{\mathbb{R}^n}
   \frac{e^{- \frac{\scriptstyle |z|^2}{\scriptstyle 4 \mu}}}{(4 \pi \mu)^{n/2}}
   \frac{|z|^\lambda\, dz}{(w (x + z \sqrt{t'}, x + z \sqrt{t''}))^{\delta+|\alpha|+\lambda}}
\end{equation}
multiple of the norm
   $\| \partial^\alpha u_0 \|_{C^{0,\lambda,\delta+|\alpha|} (\mathbb{R}^n)}$.
It remains to estimate (\ref{eq.motn}).
To this end we apply Lemma \ref{l.heat.key0} to get
\begin{eqnarray*}
   \int_{\mathbb{R}^n}
   \frac{e^{- \frac{\scriptstyle |z|^2}{\scriptstyle 4 \mu}}}{(4 \pi \mu)^{n/2}}
   \frac{|z|^\lambda\, dz}{(w (x + z \sqrt{t}))^{\delta+|\alpha|+\lambda}}
 & = &
   \int_{\mathbb{R}^n}
   \psi_\mu (x-y,t)\,
   \frac{\left( \frac{|x-y|}{\sqrt{t}} \right)^\lambda\, dz}
        {(w (y)))^{\delta+|\alpha|+\lambda}}
\\
 & \leq &
   c\, \frac{1}{(w (x))^{\delta+|\alpha|+\lambda}}
\end{eqnarray*}
with $c$ a constant independent of $(x,t) \in \overline{\mathcal{C}_T} $.
Therefore, arguing as above we obtain
\begin{eqnarray*}
   \int_{\mathbb{R}^n}
   \frac{e^{- \frac{\scriptstyle |z|^2}{\scriptstyle 4 \mu}}}{(4 \pi \mu)^{n/2}}
   \frac{|z|^\lambda\, dz}{(w (x + z \sqrt{t'}, x + z \sqrt{t''}))^{\delta+|\alpha|+\lambda}}
 & \leq &
   c\,
   \frac{1}{(\max \{ w (x), w (y) \})^{\delta+|\alpha|+\lambda}}
\\
 & \leq &
   c\, (w (x,y))^{-(\delta+|\alpha|+\lambda)}.
\end{eqnarray*}

We have thus proved that there is a constant $c > 0$ depending on $k$, $s$, $\lambda$ and
$\delta$, such that
\begin{equation}
\label{eq.sopPpisv}
   \| \varPsi_{\mu,0} u_0
   \|_{C^{k+2s,\lambda,0,\lambda/2,\delta} (\overline{\mathcal{C}_T})}
 \leq
   c\,
   \| u_0 \|_{C^{k+2s,\lambda,\delta} (\mathbb{R}^n)}
\end{equation}
for all $u_0 \in C^{k+2s,\lambda,\delta} (\mathbb{R}^n)$.
On applying the embedding given by Theorem \ref{t.emb.hoelder.t} we deduce that
\begin{eqnarray*}
   \| \partial^j_t (\varPsi_{\mu,0} u_0)
   \|_{C^{k,\mathbf{s} (s-j,\lambda,\delta)} (\overline{\mathcal{C}_T})}
 & = &
   \mu\, \| \partial^{j-1}_t \varDelta (\varPsi_{\mu,0} u_0)
         \|_{C^{k,\mathbf{s} (s-j,\lambda,\delta)} (\overline{\mathcal{C}_T})}
\\
 & = &
   \mu^j\, \| \varDelta^j  (\varPsi_{\mu,0} u_0)
           \|_{C^{k,\mathbf{s} (s-j,\lambda,\delta)} (\overline{\mathcal{C}_T})}
\\
 & \leq &
   c\,
   \| (\varPsi_{\mu,0} u_0) \|_{C^{k+2s,\lambda,0,\lambda/2,\delta} (\overline{\mathcal{C}_T})}
\\
 & \leq &
   c\,
   \| u_0 \|_{C^{k+2s,\lambda,\delta} (\mathbb{R}^n)}
\end{eqnarray*}
for all $j = 0, 1, \ldots, s$.
Hence it follows that the operator $\varPsi_{\mu,0}$ is bounded in the desired spaces.

We now turn to the volume parabolic potential $\varPsi_\mu$.
We will tacitly use the well-known fact that $\varPsi_{\mu}$ maps
   $C^{k,\mathbf{s} (s,\lambda,\delta)} (\overline{\mathcal{C}_T})$ continuously into the local space
   $C^{k,\mathbf{s} (s+1,\lambda,0)}_{\mathrm{loc}} (\overline{\mathcal{C}_T})$.

Pick any multi-index $\alpha \in \mathbb{Z}_{\geq 0}^n$ with $|\alpha| \leq k+2s$.
Using the properties of convolution in $\mathbb{R}^n$ one obtains
\begin{eqnarray*}
   \partial^\alpha _x (\varPsi_\mu f) (x,t)
 & = &
   \int_{0}^t
   \partial^\alpha_x
   \int_{\mathbb{R}^n} \psi_\mu (x-y,t-s) f (y,s) dy\, ds
\\
 & = &
   \int_0^t \!\! \int_{\mathbb{R}^n}
   \psi_\mu (x-y,t-s)\, \partial^\alpha_y f (y,s) dy\, ds
\\
 & = &
   \varPsi_\mu (\partial^\alpha f) (x,t)
\end{eqnarray*}
for all $f \in C^{k,\mathbf{s} (s,\lambda,\delta)} (\overline{\mathcal{C}_T})$.
By Lemma \ref{l.heat.key0},
\begin{eqnarray*}
\lefteqn{
   (w (x))^{\delta+|\alpha|}\, |\partial^\alpha_x (\varPsi_\mu f) (x,t)|
}
\\
 & \leq &
   \| \partial^\alpha f
   \|_{C^{k+2s-|\alpha|,0,\lambda,\lambda/2,\delta+|\alpha|} (\overline{\mathcal{C}_T})}
   \int_0^{t} \!\! \int_{\mathbb{R}^n} \!\!
   \psi_\mu (x-y,t-s)\,
   \Big( \frac{w (x)}{w (y)} \Big)^{\delta+|\alpha|}
   dy ds
\\
 & \leq &
   c\,
   \| f \|_{C^{k+2s,0,\lambda,\lambda/2,\delta} (\overline{\mathcal{C}_T})}
\end{eqnarray*}
for all $(x,t)$ in the layer $\overline{\mathcal{C}_T} $,
   the constant $c$ depending on finite $T$ but not on $f$.

In order to estimate the derivatives of $\varPsi_\mu f$ in $t$ we now use the fact that
$\psi_\mu$ is a (right) fundamental solution of convolution type to the heat operator.
If
   $f$ is an arbitrary function of $C^{k,\mathbf{s} (s,\lambda,\delta)} (\overline{\mathcal{C}_T})$
and
   $j = 1, \ldots, s$,
then Theorem \ref{t.emb.hoelder.t} yields
\begin{eqnarray*}
\lefteqn{
   \| \partial^j _t (\varPsi_\mu f)
   \|_{C^{k,\mathbf{s} (s-j,\lambda,\delta)} (\cdot)}
}
\\
 & = &
   \| \partial^{j-1}_t f + \mu \varDelta \partial^{j-1}_t (\varPsi_\mu f)
   \|_{C^{k,\mathbf{s} (s-j,\lambda,\delta)} (\cdot)}
\\
 & \leq &
   \| \partial^{j-1}_t f \|_{C^{k,\mathbf{s} (s-j,\lambda,\delta)} (\cdot)}
 +
   \| \mu \partial^{j-2}_t \varDelta f + \mu^2 \varDelta^2 \partial^{j-2}_t (\varPsi_\mu f)
   \|_{C^{k,\mathbf{s} (s-j,\lambda,\delta+2)} (\cdot)}
\\
 & \leq &
   \sum_{i=0}^{j-1}
   \| \mu^i \partial_t^{j-1-i} \varDelta^i f
   \|_{C^{k,\mathbf{s} (s-j,\lambda,\delta+2i)} (\cdot)}
 + \| \mu^j \varDelta^j (\varPsi_\mu f) \|_{C^{k,\mathbf{s} (s-j,\lambda,\delta+2j)} (\cdot)}
\\
 & \leq &
   c
   \left( \| f \|_{C^{k,\mathbf{s} (s-1,\lambda,\delta)} (\cdot)}
        + \| \varPsi_\mu f \|_{C^{k+2s,0,\lambda,\lambda/2,\delta} (\cdot)}
   \right),
\end{eqnarray*}
where all function spaces are over the layer $\overline{\mathcal{C}_T} $ which we omit by
abuse of notation.
The constant $c$ is independent of $f$, and so $\varPsi_{\mu}$ acts boundedly in the desired
spaces.
\end{proof}

Our next goal is to show a more subtle result on the volume parabolic potential $\varPsi_\mu$
which we need in the sequel.

\begin{theorem}
\label{t.heat.key2}
Let
   $s$ be a positive integer,
   $k \in \mathbb{Z}_{\geq 0}$,
   $0 < \lambda < 1$
and
   $\delta > 0$.
The potential $\varPsi_\mu $ induces a bounded linear operator
$$
   C^{k,\mathbf{s} (s-1,\lambda,\delta+2)} (\overline{\mathcal{C}_T}) \to
   C^{k,\mathbf{s} (s,\lambda,\delta)} (\overline{\mathcal{C}_T}).
$$
\end{theorem}

\begin{proof}
We begin with a weak a priori estimate of Schauder type for the heat operator in weighted spaces.

\begin{lemma}
\label{l.apriori.heat}
Suppose that
   $0 < \lambda < 1$ and
   $\delta > 0$.
If
   $f \in C^{\mathbf{s} (0,\lambda,\delta+2)} (\overline{\mathcal{C}_T})$
and a function
   $u \in C^{1,0,\lambda,\lambda/2,\delta)} (\overline{\mathcal{C}_T}) \cap
          C^{\mathbf{s} (1,\lambda,0)}_{\mathrm{loc}} (\overline{\mathcal{C}_T})$
satisfies
   $H_\mu u = f$ in $\overline{\mathcal{C}_T} $ and
   $\gamma_0 u = 0$
then $u$ belongs actually to the space $C^{\mathbf{s} (1,\lambda,\delta)} (\overline{\mathcal{C}_T})$ and
$$
   \| u \|_{C^{\mathbf{s} (1,\lambda,\delta)} (\overline{\mathcal{C}_T})}
 \leq
   c
   \left( \| u \|_{C^{1,0,\lambda,\lambda/2,\delta)} (\overline{\mathcal{C}_T})}
        + \| f \|_{C^{\mathbf{s} (0,\lambda,\delta+2)} (\overline{\mathcal{C}_T})}
   \right),
$$
where $c$ depends on $\lambda$ and $\delta$ but not on $f$ and $u$.
\end{lemma}

\begin{proof}
First, by Lemma \ref{l.emb.loc} and parabolic regularity, we conclude that any function $u$
satisfying the hypotheses of the lemma belongs to
   $C^{\mathbf{s} (1,\lambda,0)}_{\mathrm{loc}} (\overline{\mathcal{C}_T})$.
Then,
   by standard a priori estimates for solutions of the Cauchy problem for the heat operator
   (see for instance \cite[Lemma 9.2.1]{Kry96}),
there is a constant $c > 0$ such that
$$
   \| u \|_{C^{\mathbf{s} (1,\lambda)} (\overline{\mathcal{C}_T})}
 \leq
   c
   \left( \| u \|_{C^{\mathbf{s} (0,\lambda)} (\overline{\mathcal{C}_T})}
        + \| H_\mu u \|_{C^{\mathbf{s} (0,\lambda)} (\overline{\mathcal{C}_T})}
   \right)
$$
for all
$
   u
 \in
   C^{\mathbf{s} (0,\lambda)} (\overline{\mathcal{C}_T}) \cap
   C^{\mathbf{s} (1,\lambda)}_{\mathrm{loc}} (\overline{\mathcal{C}_T})
$
with
   $H_\mu w \in C^{\mathbf{s} (0,\lambda)} (\overline{\mathcal{C}_T})$
satisfying $\gamma_0 u = 0$.

Here, $C^{\mathbf{s} (s,\lambda)} (\overline{\mathcal{C}_T})$ are normed spaces of H\"{o}lder continuous
functions in all of $\overline{\mathcal{C}_T} $ which include parabolic anisotropy.
This scale of function spaces is different from our scale of weighted spaces, for it does not
specify the behaviour of functions at the infinity points of $\overline{\mathcal{C}_T} $.
However, the corresponding local versions of both scales coincide.

Let $\mathcal{X}_1$ and $\mathcal{X}_2 $ be two bounded domains in $\mathbb{R}^n$, such that
   $\mathcal{X}_1 \Subset \mathcal{X}_2$.
Fix a function $\chi \in C^\infty_{\mathrm{comp}} (\mathcal{X}_2)$ which is equal to $1$ in a
neighbourhood of the closure of $\mathcal{X}_1$.
An easy calculation shows that
$$
   H_\mu (\chi u)
 = \chi (H_\mu u) - \mu \chi (\varDelta u) - 2 \nabla \chi \cdot \nabla u
$$
for each distribution $u$ in the cylinder $\mathcal{C}_T (\mathcal{X}_2)$.
As usual, using the multiplication operator $u \mapsto \chi u$ one concludes that there is a
constant $c$ depending merely on the distance
$$
   \mathrm{dist} (\partial \mathcal{X}_1, \partial \mathcal{X}_2)
 = \inf_{x \in \partial \mathcal{X}_1 \atop y \in \partial \mathcal{X}_2}
   |x-y|,
$$
such that
\begin{equation}
\label{eq.apriori.heat.0}
   \| u \|_{C^{\mathbf{s} (1,\lambda)} (\overline{\mathcal{C}_T (\mathcal{X}_1)})}
 \leq
   c
   \left( \| u \|_{C^{1,0,\lambda,\lambda/2} (\overline{\mathcal{C}_T (\mathcal{X}_2)})}
        + \| H_\mu u \|_{C^{\mathbf{s} (0,\lambda)} (\overline{\overline{\mathcal{C}_T} (\mathcal{X}_2)})}
   \right)
\end{equation}
for all functions
   $u \in C^{1,0,\lambda,\lambda/2} (\overline{\overline{\mathcal{C}_T} (\mathcal{X}_2)}) \cap
          C^{\mathbf{s} (1,\lambda)}_{\mathrm{loc}} (\overline{\mathcal{C}_T (\mathcal{X}_2)})$
which satisfy
   $H_\mu u \in C^{\mathbf{s} (0,\lambda)} (\overline{\mathcal{C}_T (\mathcal{X}_2)})$
and
   $\gamma_0 u = 0$.

In particular, (\ref{eq.apriori.heat.0}) applies to any function $u$ satisfying the hypotheses
of Lemma \ref{l.apriori.heat}, which yields
\begin{eqnarray*}
\lefteqn{
   \| \partial_t u \|_{C^{\mathbf{s} (0,\lambda)} (\overline{\mathcal{C}_T (B_2)})}
 + \sum_{|\alpha| \leq 2} \|
   \partial_x^\alpha u \|_{C^{\mathbf{s} (0,\lambda)} (\overline{\mathcal{C}_T (B_2)})}
}
\\
 & \leq &
   c
   \left( \| u \|_{C^{1,0,\lambda,\lambda/2} (\overline{\mathcal{C}_T (B_4)})}
        + \| H_\mu u \|_{C^{\mathbf{s} (0,\lambda)} (\overline{\mathcal{C}_T (B_4)})}
   \right)
\end{eqnarray*}
with $c$ a constant independent on $u$.
If $|x| \leq 4$, then
$
   1 \leq (1+|x|^2)^{1/2} \leq \sqrt{17}.
$
Hence, letting the constant $c$ depend also on $\delta$ and using (\ref{eq.half}) we can
rewrite the last inequality as
\begin{eqnarray}
\label{eq.apriori.heat.10}
\lefteqn{
   \| \partial_t u
   \|_{C^{\mathbf{s} (0,\lambda,\delta+2)} (\overline{\mathcal{C}_T (B_1)})}
 + \sum_{|\alpha| \leq 2}
   \| \partial_x^\alpha u
   \|_{C^{\mathbf{s} (0,\lambda,\delta+|\alpha|)} (\overline{\mathcal{C}_T (B_1)})}
}
\nonumber
\\
 & \leq &
   c (\delta)
   \left(  \| u \|_{C^{1,0,\lambda,\lambda/2,\delta} (\overline{\mathcal{C}_T (B_2)})}
         + \| H_\mu u \|_{C^{\mathbf{s} (0,\lambda,\delta+2)} (\overline{\mathcal{C}_T (B_2)})}
   \right),
\nonumber
\\
\lefteqn{
   \| \partial_t u
   \|_{C^{\mathbf{s} (0,\lambda,\delta+2)} (\overline{\mathcal{C}_T (B_2)})}
 + \sum_{|\alpha| \leq 2}
   \| \partial_x^\alpha u
   \|_{C^{\mathbf{s} (0,\lambda,\delta+|\alpha|)} (\overline{\mathcal{C}_T (B_2)})}
}
\nonumber
\\
 & \leq &
   c (\delta)
   \left(  \| u \|_{C^{1,0,\lambda,\lambda/2,\delta} (\overline{\mathcal{C}_T (B_4)})}
         + \| H_\mu u \|_{C^{\mathbf{s} (0,\lambda,\delta+2)} (\overline{\mathcal{C}_T (B_4)})}
   \right)
\nonumber
\\
\end{eqnarray}
for all $u$ as in the statement of the lemma.

We now consider a region in $\mathbb{R}^n$ of the form $r < |x| < 8 r$,
   where $r \geq 1$ is a fixed constant,
and set
$$
   u_r (x,t) = u (r x, r^2 t)
$$
for $1 \leq |x| \leq 8$ and
    $t \in [0,T/r^2]$.
Then
$$
\begin{array}{rcl}
   \partial_i u_r (x,t)
 & =
 & r (\partial_i u) (r x, r^2 t),
\\
   \partial_j \partial_i u_r (x,t)
 & =
 & r^2 (\partial_j \partial_i u) (r x, r^2 t),
\\
   \partial _t u_r (x,t)
 & =
 & r^2 (\partial_t u) (r x, r^2 t)
\end{array}
$$
whence
$$
   H_\mu u_r (x,t)
 = r^2 \left( \partial_t v (r x,r^2 t) - (\mu \varDelta u) (r x, r^2 t) \right)
 = r^2 f (r x, r^2 t).
$$
On applying estimate (\ref{eq.apriori.heat.0}) to $u_r$ we obtain
\begin{eqnarray*}
\lefteqn{
   \| \partial_t u_r
   \|_{C^{\mathbf{s} (0,\lambda)} ((\overline{B_4} \setminus B_2) \times [0, T/r^2])}
 + \sum_{|\alpha| \leq 2}
   \| \partial_x^\alpha  u_r
   \|_{C^{\mathbf{s} (0,\lambda)} ((\overline{B_4} \setminus B_2) \times [0, T/r^2])}
}
\\
 \!\! & \!\! \leq \!\! & \!\!
   c
   \left( \!
   \| u_r \|_{C^{1,0,\lambda,\lambda/2} ((\overline{B_8} \setminus B_1) \times [0, T/r^2])}
 + \| H_\mu u_r \|_{C^{\mathbf{s} (0,\lambda)} ((\overline{B_8} \setminus B_1 \times [0,T/r^2])}
   \! \right),
\end{eqnarray*}
where the constant $c$ could be chosen to be independent of $r$ because
   $r \geq 1$ and
   $0 < T/r^2 \leq T$.
This is equivalent to
\begin{eqnarray*}
\lefteqn{
   r^{\delta+2}
   \| \partial_t u
   \|_{C^{\mathbf{s} (0,\lambda)} ((\overline{B_{4r}} \setminus B_{2r}) \times [0,T])}
 + \sum_{|\alpha| \leq 2}
   r^{\delta+|\alpha|}
   \| \partial^\alpha_x u
   \|_{C^{\mathbf{s} (0,\lambda)} ((\overline{B_{4r}} \setminus  B_{2r}) \times [0,T])}
}
\\
 \!\! & \!\! \leq \!\! & \!\!
   c\,
   \Big(
   \sum_{|\beta| \leq 1}
   r^{\delta+|\beta|}
   \| \partial^\beta u
   \|_{C^{\mathbf{s} (0,\lambda)} ((\overline{B_{8r}} \setminus B_r) \times [0,T])}
 + r^{\delta+2}
   \| L_\mu u \|_{C^{\mathbf{s} (0,\lambda)} ((\overline{B_{8r}} \setminus B_r) \times [0,T])}
   \Big),
\end{eqnarray*}
the constant $c>0$ being independent of $r \geq 1$ and $u$.

For $r \leq x \leq 8r$, we have $r < (1+|x|^2)^{1/2} \leq 9r$.
Hence, the latter inequality just amounts to saying that
\begin{eqnarray*}
\lefteqn{
   \| \partial_t u
   \|_{C^{\mathbf{s} (0,\lambda,\delta+2)} ((\overline{B_{4r}} \setminus B_{2r}) \times [0,T])}
 + \sum_{|\alpha| \leq 2}
   \| \partial_x^\alpha u
   \|_{C^{\mathbf{s} (0,\lambda,\delta+|\alpha|)} ((\overline{B_{4r}} \setminus B_{2r}) \times [0,T])}
}
\\
 \!\! & \!\! \leq \!\! & \!\!
   c (\delta)
   \left(
   \| u \|_{C^{1,0,\lambda,\lambda/2,\delta} ((\overline{B_{8r}} \setminus B_r) \times [0,T])}
+ \| H_\mu u \|_{C^{\mathbf{s} (0,\lambda,\delta+2)} ((\overline{B_{8r}} \setminus B_r) \times [0,T])}
   \right)
\end{eqnarray*}
holds with some constant $c (\delta) > 0$ independent of $r$ and $u$.
Choose $r =2^m$, for $m = 0, 1, \ldots$, to get
\begin{eqnarray}
\label{eq.apriori.heat.m}
\lefteqn{
   \| \partial_t u
   \|_{C^{\mathbf{s} (0,\lambda,\delta+2)} ((\overline{B_{2^{m+2}}} \setminus B_{2^{m+1}}) \times [0,T])}
 + \sum_{|\alpha| \leq 2}
   \| \partial_x^\alpha u
   \|_{C^{\mathbf{s} (0,\lambda,\delta+|\alpha|)} ((\overline{B_{2^{m+2}}} \setminus B_{2^{m+1}}) \times [0,T])}
}
\nonumber
\\
 \!\! & \!\! \leq \!\! & \!\!
   c (\delta)
   \left(
   \| u \|_{C^{1,0,\lambda,\lambda/2,\delta} ((\overline{B_{2^{m+3}}} \setminus B_{2^m}) \times [0,T])}
+ \| H_\mu u \|_{C^{\mathbf{s} (0,\lambda,\delta+2)} ((\overline{B_{2^{m+3}}} \setminus B_{2^m}) \times [0,T])}
   \right)
\nonumber
\\
\end{eqnarray}
with $c (\delta)$ a constant independent of $u$.

Finally, on combining Theorem \ref{t.emb.hoelder.t} with
   (\ref{eq.apriori.heat.10}) and
   (\ref{eq.apriori.heat.m}) for $m = 0, 1, \ldots$
we obtain the statement of the lemma.
\end{proof}

In the next lemma we present an intermediate assertion towards the complete proof
of Theorem \ref{t.heat.key2}.

\begin{lemma}
\label{l.heat.key2}
Let
   $s$ be a positive integer,
   $k \in \mathbb{Z}_{\geq 0}$,
   $0 < \lambda < 1$
and
   $\delta > 0$.
The parabolic potential $\varPsi_\mu$ induces a bounded linear operator
$$
   C^{k,\mathbf{s} (s-1,\lambda,\delta+2)} (\overline{\mathcal{C}_T}) \to
   C^{k+1,\mathbf{s} (s-1,\lambda,\delta)} (\overline{\mathcal{C}_T}).
$$
\end{lemma}

\begin{proof}
It is well known that $\varPsi_\mu$ maps
   $C^{k,\mathbf{s} (s-1,\lambda,\delta+2)} (\overline{\mathcal{C}_T})$ continuously into the Fr\'echet space
   $C^{k,\mathbf{s} (s,\lambda,0)}_{\mathrm{loc}} (\overline{\mathcal{C}_T})$
provided that $0 < \lambda < 1$.
Arguing as in Lemma \ref{l.heat.key1} we see that
\begin{eqnarray*}
\lefteqn{
   (w (x))^\delta |(\varPsi_\mu f) (x,t)|
}
\\
 & \leq &
   \| f \|_{C^{\mathbf{s} (0,\lambda,\delta+2)} (\overline{\mathcal{C}_T})}
   \int_0^{t} \!\! \int_{\mathbb{R}^n}
   \psi_\mu (x-y,t-s)
   \frac{(1+|x|^2)^{\delta/2}}{(1+|y|^2)^{(\delta+2)/2}}
   dy ds
\\
 & \leq &
   c\,
   \| f \|_{C^{\mathbf{s} (0,\lambda,\delta+2)} (\overline{\mathcal{C}_T})}
   \frac{c}{1+|x|^2}
   \sup_{t \in [0,T]} \int_0^{t} ds,
\end{eqnarray*}
where $c > 0$ is a constant depending on $T$ but not on $x$ and $f$ because of
Lemma \ref{l.heat.key0}.

Furthermore, using Lemma \ref{l.heat.key0} and (\ref{eq.half}) we obtain
\begin{eqnarray*}
\lefteqn{
   (w (x,y))^{\delta+\lambda}\,
   \frac{|(\varPsi_\mu f) (x,t) - (\varPsi_\mu f) (y,t)|}{|x-y|^\lambda}
}
\\
 & \leq &
   (w (x,y))^{\delta+\lambda}
   \int_0^{t}
   \left(
   \int_{\mathbb{R}^n} \psi_\mu (z,t-s) \frac{|f (x+z) - f (y+z)|}{|x-y|^\lambda} dz
   \right)
   ds
\\
 & \leq &
   \| f \|_{C^{\mathbf{s} (0,\lambda,\delta+2)} (\overline{\mathcal{C}_T})}
   \int_0^{t} \!\! \int_{\mathbb{R}^n}
   \psi_\mu (z,t-s)
   \frac{(w (x,y))^{\delta+\lambda}}
        {(w (x+z,y+z))^{\delta+\lambda+2}}
   dz ds
\\
 & \leq &
 \| f \|_{C^{\mathbf{s} (0,\lambda,\delta+2)} (\overline{\mathcal{C}_T})}\,
 \frac{c}{(w (x,y))^2}
 \int_0^T ds
\end{eqnarray*}
for $x \neq y$,
   the constant $c$ being independent of $f$ and different in diverse applications.

In addition, pick $t', t'' \in [0,T]$.
Without restriction of generality one can assume that $t' < t''$.
To evaluate the difference quotient
$$
   (w (x))^\delta\,
   \frac{|(\varPsi_\mu f) (x,t') - (\varPsi_\mu f) (x,t'')|}{|t'-t''|^{\lambda/2}}
$$
we make the changes of variables
   $t'-s = s'$ and
   $t''-s = s''$
in the integrals in question.
A direct calculation shows that the quotient is majorised by the sum of two terms
$$
\begin{array}{rcl}
   I_1
 & :=
 & \displaystyle
   (w (x))^\delta
   \int_0^{t'} \!\! \int_{\mathbb{R}^n}
   \psi_\mu (x-y,s)
   \frac{|f (y,t'-s) - f (y,t''-s)|}{|t'-t''|^{\lambda/2}}\,
   dy ds,
\\
   I_2
 & :=
 & \displaystyle
   \frac{(w (x))^\delta}{|t'-t''|^{\lambda/2}}
   \int_{t'}^{t''} \!\! \int_{\mathbb{R}^n}
   \psi_\mu (x-y,s)\,
   |f (y,t''-s)|\, dy ds.
\end{array}
$$
We get
\begin{eqnarray*}
   I_1
 & \leq &
   \| f \|_{C^{\mathbf{s} (0,\lambda,\delta+2)} (\overline{\mathcal{C}_T})}
   \int_0^{t'} \!\! \int_{\mathbb{R}^n}
   \psi_\mu (x-y,s)\, \frac{(w (x))^\delta}{(w (y))^{\delta+2}}
   dy ds
\\
 & \leq &
   \frac{c}{(w (x))^2}\,
   \| f \|_{C^{\mathbf{s} (0,\lambda,\delta+2)} (\overline{\mathcal{C}_T})}
\end{eqnarray*}
with $c$ a constant depending neither on $x \in \mathbb{R}^n$ and $t', t'' \in [0,T]$
nor on $f$,
   the last inequality being due to Lemma \ref{l.heat.key0}.
The term $I_2$ is estimated by using the mean value theorem for integrals, for
$$
   I_2
 \leq
   \| f \|_{C^{\mathbf{s} (0,0,\delta+2)} (\overline{\mathcal{C}_T})}
   \frac{|t'-t''|}{|t'-t''|^{\lambda/2}}
   \int_{\mathbb{R}^n}
   \psi_\mu (x-y,t_{\vartheta})\,
   \frac{(w (x))^\delta}{(w (y))^{\delta+2}}
   dy
$$
with an intermediate point $t_ {\vartheta} \in (t',t'')$.
On applying Lemma \ref{l.heat.key0} once again we get
$$
   I_2
 \leq
   c\,
   \frac{|t'-t''|^{1-\lambda/2}}{(w (x))^2}\,
   \| f \|_{C^{\mathbf{s} (0,0,\delta+2)} (\overline{\mathcal{C}_T})}
$$
and so
$$
   (w (x))^\delta\,
   \frac{|(\varPsi_\mu f) (x,t') - (\varPsi_\mu f) (x,t'')|}{|t'-t''|^{\lambda/2}}
 \leq
   c\, \| f \|_{C^{\mathbf{s} (0,\lambda,\delta+2)} (\overline{\mathcal{C}_T})}
$$
with $c$ a constant independent of $f$.
Summarising we see that $\varPsi_\mu$ induces a bounded linear operator
$$
   C^{\mathbf{s} (0,\lambda,\delta+2)} (\overline{\mathcal{C}_T}) \to
   C^{\mathbf{s} (0,\lambda,\delta)} (\overline{\mathcal{C}_T}).
$$

Similarly, for any multi-index $\alpha \in \mathbb{Z}_{\geq 0}$ of norm $|\alpha| = 1$,
it follows by Lemma \ref{l.heat.key0} that
\begin{eqnarray*}
\lefteqn{
   (w (x))^{\delta+1} |\partial^\alpha (\varPsi_{\mu} f) (x,t)|
}
\\
 & \leq &
   \| f \|_{C^{\mathbf{s} (0,\lambda,\delta+2)} (\overline{\mathcal{C}_T})}
   \int_0^t \!\! \int_{\mathbb{R}^n}
   \frac{2 |x-y|}{4 \mu (t-s)}
   \psi_\mu (x-y,t-s)
   \frac{(w (x))^{\delta+1}}{(w (y))^{\delta+2}}
   dy ds
\\
 & \leq &
   \| f \|_{C^{\mathbf{s} (0,\lambda,\delta+2)} (\overline{\mathcal{C}_T})}\,
   \frac{c}{w (x)}
   \sup_{t \in [0,T]} \int_0^{t} \frac{ds}{\sqrt{t-s}}
\\
 & = &
   2 \sqrt{T}\, \frac{c}{w (x)}\,
   \| f \|_{C^{\mathbf{s} (0,\lambda,\delta+2)} (\overline{\mathcal{C}_T})},
\end{eqnarray*}
where $c$ is a constant independent of $f$.
Furthermore, using Lemma \ref{l.heat.key0} and (\ref{eq.half}) we obtain
\begin{eqnarray*}
\lefteqn{
   (w (x,y))^{\delta+1+\lambda}\,
   \frac{|\partial^\alpha_x (\varPsi_\mu f) (x,t) - \partial^\alpha_y (\varPsi_\mu f) (y,t)|}
        {|x-y|^\lambda}
}
\\
 \! & \! \leq \! & \!
   (w (x,y))^{\delta+1+\lambda}
   \int_0^t
   \!  \Big( \!
   \int_{\mathbb{R}^n} \!
   \frac{2 |z|}{4 \mu (t-s)}
   \psi_\mu (z,t-s)\,
   \frac{f (x\!-\!z,s)- f (y\!-\!z,s)|}{|x-y|^\lambda}
   dz
   \Big)
   ds
\\
 \! & \! \leq \! & \!
   \| f \|_{C^{\mathbf{s} (0,\lambda,\delta+2)} (\overline{\mathcal{C}_T})}
   \!
   \int_0^{t} \!\! \int_{\mathbb{R}^n}
   \!
   \frac{2 |z|}{4 \mu (t\!-\!s)}
   \psi_\mu (z,t\!-\!s)
   \frac{(w (x,y))^{\delta+1+\lambda}}
        {(w (x\!-\!z,y\!-\!z))^{\delta+2+\lambda}}\,
   dz ds
\\
 \! & \! \leq \! & \!
   \sup_{t \in [0,T]} \int_0^t \frac{ds}{\sqrt{t-s}}\,
   \frac{c}{w (x,y)}\,
   \| f \|_{C^{\mathbf{s} (0,\lambda,\delta+2)} (\overline{\mathcal{C}_T})}
\end{eqnarray*}
for all $x \neq y$ and $t > 0$,
   with $c$ a constant independent of $f$.

To estimate the difference quotient
$$
   (w (x))^{\delta+1}\,
   \frac{|\partial^\alpha_x (\varPsi_\mu f) (x,t') - \partial^\alpha_x (\varPsi_\mu f) (x,t'')|}
        {|t'-t''|^{\lambda/2}}
$$
for $x \in \mathbb{R}^n$ and $t', t'' \in [0,T]$, we argue as above.
To be definite assume that $t' < t''$.
An immediate calculation shows that the quotient is majorised by the sum of two terms
$$
\begin{array}{rcl}
   I_1'
 & :=
 & \displaystyle
   (w (x))^{\delta+1}
   \int_0^{t'} \!\! \int_{\mathbb{R}^n}
   \frac{2 |x-y|}{4 \mu s}\,
   \psi_\mu (x-y,s)
   \frac{|f (y,t'-s) - f (y,t''-s)|}{|t'-t''|^{\lambda/2}}\,
   dy ds,
\\
   I_2'
 & :=
 & \displaystyle
   \frac{(w (x))^{\delta+1}}{|t'-t''|^{\lambda/2}}
   \int_{t'}^{t''} \!\! \int_{\mathbb{R}^n}
   \frac{2 |x-y|}{4 \mu s}\,
   \psi_\mu (x-y,s)\,
   |f (y,t''-s)|\, dy ds.
\end{array}
$$
One verifies readily that
\begin{eqnarray*}
   I_1'
 & \leq &
   \| f \|_{C^{\mathbf{s} (0,\lambda,\delta+2)} (\overline{\mathcal{C}_T})}
   \int_0^{t'} \!\! \int_{\mathbb{R}^n}
   \frac{2 |x-y|}{4 \mu s}\,
   \psi_\mu (x-y,s)\,
   \frac{(w (x))^{\delta+1}}{(w (y))^{\delta+2}}
   dy ds
\\
 & \leq &
   \frac{c}{w (x)}\,
   \| f \|_{C^{\mathbf{s} (0,\lambda,\delta+2)} (\overline{\mathcal{C}_T})}\,
   \int_0^{T} \frac{ds}{\sqrt{s}}
\end{eqnarray*}
with $c$ a constant independent of both $x$ and $t'$, $t''$ and $f$,
   the last inequality being a consequence of Lemma \ref{l.heat.key0}.
On the other hand, the term $I_2'$ is estimated by using the mean value theorem for integrals
in a generalised form.
More precisely, we obtain
$$
   I_2'
 \leq
   \| f \|_{C^{\mathbf{s} (0,0,\delta+2)} (\overline{\mathcal{C}_T})}
   \frac{\displaystyle \int_{t'}^{t''} \frac{ds}{\sqrt{s}}}
        {|t'-t''|^{\lambda/2}}
   \int_{\mathbb{R}^n} \!\! \int_{t'}^{t''}
   \frac{2 |x-y|}{4 \mu\sqrt{t_{\vartheta}}}\,
   \psi_\mu (x-y,t_{\vartheta})\,
   \frac{(w (x))^{\delta+1}}{(w (y))^{\delta+2}}
   dy ds
$$
with an intermediate point $t_ {\vartheta} \in (t',t'')$.
Since
$$
   \int_{t'}^{t''} \frac{ds}{\sqrt{s}}
 \leq
   2\, \sqrt{t''-t'}
$$
for all $t'$, $t''$ satisfying $0 \leq t' \leq t''$, an application of Lemma \ref{l.heat.key0}
yields that
$$
   I_2'
 \leq
   c\,
   \frac{|t'-t''|^{(1-\lambda)/2}}{w (x)}\,
   \| f \|_{C^{\mathbf{s} (0,0,\delta+2)} (\overline{\mathcal{C}_T})}
$$
and so
$$
   (w (x))^{\delta+1}\,
   \frac{|\partial^\alpha _x (\varPsi_\mu f) (x,t') - \partial^\alpha _x (\varPsi_\mu f) (x,t'')|}
        {|t'-t''|^{\lambda/2}}
 \leq
   c\, \| f \|_{C^{\mathbf{s} (0,\lambda,\delta+2)} (\overline{\mathcal{C}_T})}
$$
with $c$ a constant independent of $f$.

Hence, the volume parabolic potential $\varPsi_\mu$ maps
   $C^{\mathbf{s} (0,\lambda,\delta+2)} (\overline{\mathcal{C}_T})$ continuously into
   $C^{1,\mathbf{s} (0,\lambda,\delta)} (\overline{\mathcal{C}_T})$.
I.e., we have proved the statement of the lemma for $s=1$ and $k=0$.

Suppose $s > 1$ and $k$ is a positive integer.
If $f \in C^{k,\mathbf{s} (s-1,\lambda,\delta+2)} (\overline{\mathcal{C}_T})$ then,
   as in Lemma \ref{l.heat.key1},
$$
   \partial^{\alpha+\beta}_x \varPsi_\mu f
 = \varPsi_\mu (\partial^{\alpha+\beta}_x f)
$$
for all multi-indices $\alpha, \beta \in \mathbb{Z}_{\geq 0}^n$ such that
   $|\alpha| \leq 2 (s-1)$ and
   $|\beta| \leq k$.
Hence it follows that
\begin{equation}
\label{eq.heat.key2.alpha1}
   \| \varPsi_\mu f \|_{C^{2(s-1)+k+1,0,\lambda,\lambda/2,\delta} (\overline{\mathcal{C}_T})}
 \leq
   c\,
   \| f \|_{C^{2(s-1)+k,0,\lambda,\lambda/2,\delta+2} (\overline{\mathcal{C}_T})}
\end{equation}
where $c$ is a positive constant depending on $s$, $k$ and $T$ but not on $f$.

For $s=2$, we have
   $f \in C^{2+k,1,\lambda,\lambda/2,\delta+2} (\overline{\mathcal{C}_T})$.
By the embedding of Theorem \ref{t.emb.hoelder.t},
\begin{eqnarray*}
\lefteqn{
   \| \partial_t (\varPsi_\mu f)
   \|_{C^{k+1,0,\lambda,\lambda/2,\delta} (\overline{\mathcal{C}_T})}
}
\\
 & = &
   \| f + \mu \varDelta (\varPsi_{\mu} f)
   \|_{C^{k+1,0,\lambda,\lambda/2,\delta} (\overline{\mathcal{C}_T})}
\\
 & \leq &
   \| f \|_{C^{k+1,0,\lambda,\lambda/2,\delta} (\overline{\mathcal{C}_T})}
 + c \mu\,
   \| \varDelta (\varPsi_{\mu} f)
   \|_{C^{k+1,0,\lambda,\lambda/2,\delta+2} (\overline{\mathcal{C}_T})}
\\
 & \leq &
   \| f \|_{C^{k+1,0,\lambda,\lambda/2,\delta} (\overline{\mathcal{C}_T})}
 + c \mu\,
   \| \varPsi_{\mu} f
   \|_{C^{k+3,0,\lambda,\lambda/2,\delta} (\overline{\mathcal{C}_T})}
\\
 & \leq &
   c\,
   \| f \|_{C^{k+2,0,\lambda,\lambda/2,\delta+2} (\overline{\mathcal{C}_T})}
\end{eqnarray*}
and so
$$
   \| \varPsi_\mu f
   \|_{C^{k+3,1,\lambda,\lambda/2,\delta} (\overline{\mathcal{C}_T})}
 \leq
   c\,
   \| f \|_{C^{k+2,1,\lambda,\lambda/2,\delta+2} (\overline{\mathcal{C}_T})},
$$
the constant $c$ is independent of $f$ and it may be different in diverse applications.
More generally,
   for arbitrary $s > 1$ and $1 \leq  j \leq s-1$,
the embedding of Theorem \ref{t.emb.hoelder.t} implies
\begin{eqnarray}
\label{eq.heat.key2.alpha2}
\lefteqn{
   \| \partial^j_t (\varPsi_\mu f)
   \|_{C^{k+1,\mathbf{s} (s-1-j,\lambda,\delta)} (\cdot)}
}
\nonumber
\\
 \! & \! = \! & \!
   \| \partial^{j\!-\!1}_t f + \mu \varDelta \partial^{j\!-\!1}_t (\varPsi_\mu f)
   \|_{C^{k+1,\mathbf{s} (s-1-j,\lambda,\delta)} (\cdot)}
\nonumber
\\
 \! & \! \leq \! & \!
   \| \partial^{j\!-\!1}_t f
   \|_{C^{k+1,\mathbf{s} (s-1-j,\lambda,\delta)} (\cdot)}
 + c
   \| \mu \partial^{j\!-\!2}_t \varDelta f + \mu^2 \varDelta^2 \partial^{j\!-\!2}_t (\varPsi_\mu f)
   \|_{C^{k+1,\mathbf{s} (s-1-j,\lambda,\delta+2)} (\cdot)}
\nonumber
\\
 \! & \! \leq \! & \!
   c\!
   \sum_{i=0}^{j-1}
   \| \partial_t^{j\!-\!1\!-\!i} \mu^i \varDelta^i f \|_{C^{k+1,\mathbf{s} (s-1-j,\lambda,\delta+2i)} (\cdot)}
 + c
   \| \mu^j \varDelta^j (\varPsi_\mu f) \|_{C^{k+1,\mathbf{s} (s-1-j,\lambda,\delta+2j)} (\cdot)}
\nonumber
\\
 \! & \! \leq \! & \!
   c
   \left(
   \| f \|_{C^{k+1,\mathbf{s} (s-2,\lambda,\delta+2)} (\overline{\mathcal{C}_T})}
 + \| \varPsi_\mu f \|_{C^{2s+k+1,0,\lambda,\lambda/2,\delta} (\cdot)}
   \right)
\nonumber
\\
\end{eqnarray}
which is dominated by a constant multiple of
   $\| f \|_{C^{k,\mathbf{s} (s-1,\lambda,\delta+2)} (\overline{\mathcal{C}_T})}$
with a constant independent of $f$.
In (\ref{eq.heat.key2.alpha2}) all function spaces are over the domain $\overline{\mathcal{C}_T} $
which we omit for short.

We thus arrive at an inequality
$$
   \| \varPsi_\mu f \|_{C^{k+1,\mathbf{s} (s-1,\lambda,\delta)} (\overline{\mathcal{C}_T})}
 \leq
   c\,
   \| f \|_{C^{k,\mathbf{s} (s-1,\lambda,\delta+2)} (\overline{\mathcal{C}_T})},
$$
showing the boundedness of $\varPsi_\mu$ in the desired spaces.
\end{proof}

We are now in a position to finish the proof of Theorem \ref{t.heat.key2}.
If $f$ is an arbitrary function in
   $C^{k,\mathbf{s} (s-1),\lambda,\delta+2)} (\overline{\mathcal{C}_T})$
and
   $|\alpha| \leq 2(s-1)$,
   $|\beta| \leq k$,
then Lemma \ref{l.heat.key2} shows that
$
   \partial^{\alpha+\beta}_x (\varPsi_\mu f)
 = \varPsi_\mu (\partial^{\alpha+\beta}_x f)
$
belongs to
$
   C^{1,0,\lambda,\lambda/2,\delta+|\alpha|+|\beta|} (\overline{\mathcal{C}_T})
$
and
$$
\begin{array}{rclcl}
   H_\mu\, \partial^{\alpha+\beta}_x (\varPsi_\mu f)
 & =
 & \partial^{\alpha+\beta}_x f
 & \in
 & C^{\mathbf{s} (0,\lambda,\delta+|\alpha|+|\beta|+2)} (\overline{\mathcal{C}_T}),
\\
   \gamma_0\, \partial^{\alpha+\beta}_x (\varPsi_\mu f)
 & =
 & 0.
 &
 &
\end{array}
$$
From the estimate of Lemma \ref{l.apriori.heat} and inequality (\ref{eq.heat.key2.alpha1}) it
follows that
\begin{equation}
\label{eq.heat.key2.beta1}
   \| \partial^{\alpha+\beta}_x (\varPsi_\mu f)
   \|_{C^{\mathbf{s} (1,\lambda,\delta+|\alpha|+|\beta|} (\overline{\mathcal{C}_T})}
\leq
  c\,
  \| \partial^{\alpha+\beta} f
  \|_{C^{\mathbf{s} (0,\lambda,\delta+|\alpha|+|\beta|+2)} (\overline{\mathcal{C}_T})}
\end{equation}
with a positive constant $c$ independent of $f$.
Moreover, arguing as in (\ref{eq.heat.key2.alpha2}) we get
\begin{eqnarray}
\label{eq.heat.key2.beta2}
\lefteqn{
   \| \partial^s_t (\varPsi_{\mu} f) \|_{C^{k,0,\lambda,\lambda/2,\delta} (\cdot)}
}
\nonumber
\\
 & = &
   \| \partial^{s-1}_t f + \mu \varDelta \partial^{s-1}_t (\varPsi_{\mu} f)
   \|_{C^{k,0,\lambda,\lambda/2,\delta} (\cdot)}
\nonumber
\\
 & \leq &
   \| \partial^{s-1}_t f
   \|_{C^{k,0,\lambda,\lambda/2,\delta} (\cdot)}
 + c\,
   \| \mu \partial^{s-2}_t \varDelta f + \mu^2 \varDelta^2 \partial^{s-2}_t (\varPsi_{\mu} f)
   \|_{C^{k,0,\lambda,\lambda/2,\delta+2} (\cdot)}
\nonumber
\\
 & \leq &
   c
   \sum_{i=0}^{s-1}
   \| \partial_t^{s-1-i} \mu^i \varDelta^i f
   \|_{C^{k,0,\lambda,\lambda/2,\delta+2i} (\cdot)}
 + c
   \| \mu^s \varDelta^s (\varPsi_{\mu} f)
   \|_{C^{k,0,\lambda,\lambda/2,\delta+2s} (\cdot)}
\nonumber
\\
 & \leq &
   c
   \left( \| f \|_{C^{k,\mathbf{s} (s-1,\lambda,\delta+2)} (\cdot)}
        + \| \varPsi_{\mu} f \|_{C^{2s+k,0,\lambda,\lambda/2,\delta} (\cdot)}
   \right)
\nonumber
\\
\end{eqnarray}
which is dominated by
   $c\, \| f \|_{C^{k,\mathbf{s} (s-1,\lambda,\delta+2)} (\overline{\mathcal{C}_T})}$
with a constant $c$ independent of $f$.

On combining
   (\ref{eq.heat.key2.alpha1}),
   (\ref{eq.heat.key2.alpha2}),
   (\ref{eq.heat.key2.beta1}),
   (\ref{eq.heat.key2.beta2})
we see that the potential $\varPsi_\mu$ induces a bounded linear operator
$$
   C^{k,\mathbf{s} (s-1,\lambda,\delta+2)} (\overline{\mathcal{C}_T}) \to
   C^{k,\mathbf{s} (s,\lambda,\delta)} (\overline{\mathcal{C}_T}),
$$
provided that $0 < \lambda < 1$.
\end{proof}

\section{The linearised Navier-Stokes equations }
\label{s.NS.lin}

Now we begin to study the operators related to a linearisation of the Navier-Stokes equations.
For this purpose, denote by
$$
   \mathcal{A}_{V_0}
 = \Big( \begin{array}{cc}
           H_{\mu} + V_0
         & d^0
\\
           \gamma_0
         & 0
         \end{array}
         \Big)
$$
a  linearisation of the Navier-Stokes equations with first order term
$$
   V_0 u
 = \ast \left( \ast g^{(0)} \wedge u \right)
 + \ast \left( \ast d^1 u \wedge v^{(1)} \right)
 + d^0 \ast \left( \ast v^{(2)} \wedge u \right)
$$
in $\overline{\mathcal{C}_T} $, where
   $v^{(1)}$ and $v^{(2)}$ are fixed one-forms and
   $g^{(0)}$ is a fixed two-form,
cf. Lemma \ref{l.D.deRham}.

\begin{theorem}
\label{t.NS.deriv.unique}
Let
   $n \geq 2$,
   $s$ be a positive integer,
   $k \in \mathbb{Z}_{\geq 0}$
and
   $\delta > n/2$.
If
   the coefficients of $v^{(1)}$ are of class $C^{\mathbf{s} (0,0,0)}$,
   the coefficients of $v^{(2)}$ are of class $C^{1,\mathbf{s} (0,0,-1)}$ and
   the coefficients of $g^{(0)}$ are of class $C^{\mathbf{s} (0,0,0)}$
in $\overline{\mathcal{C}_T} $, then any pair $U = (u,p)^T$ of
\begin{equation}
\label{eq.unicla}
\begin{array}{rcl}
   u
 & \in
 & C^{\mathbf{s} (1,0,\delta)} (\overline{\mathcal{C}_T} , \varLambda^1) \cap \mathcal{S}_{d^\ast},
\\
   p
 & \in
 & C^{1,\mathbf{s} (0,0,\delta-1)} (\overline{\mathcal{C}_T})
\end{array}
\end{equation}
satisfying $\mathcal{A}_{V_0} U = 0$ in the layer is identically zero.
\end{theorem}

\begin{proof}
For $U =(u,p)^T$ the equality
   $\mathcal{A}_{V_0} U = 0$
just amounts to saying that
\begin{equation}
\label{eq.NS.deriv.0}
\begin{array}{rclrl}
   H_\mu u + d^0 p
 & =
 & - V_0 u
 & \mbox{in}
 & \overline{\mathcal{C}_T} ,
\\
   \gamma_0 u
 & =
 & 0
 & \mbox{in}
 & \mathbb{R}^n.
\end{array}
\end{equation}

Since $\delta > n/2$, one may follow
   \cite{Lera34a} or
   Theorem 3.2 for $n = 2$ and Theorem 3.4 for $n = 3$ of \cite{Tema79},
proving the uniqueness result with the use of integration by parts.
Indeed,
$$
   \partial_t \| u (\cdot, t) \|^2_{L^2 (B_R)}
 = 2\, (\partial_t u, u)_{L^2 (B_R)},
$$
for we restrict ourselves to real-valued forms.
From the properties of $u$ and $p$ listed in (\ref{eq.unicla}) we conclude that
   $dp \in C^{\mathbf{s} (0,0,\delta)} (\overline{\mathcal{C}_T} , \varLambda^1)$
and the coefficients of the one-forms
   $u$,
   $\partial_i u$,
   $\partial_t u$
and
   $H_\mu u$,
   $dp$
are square integrable over $\mathbb{R}^n$ for each fixed $t \in [0,T]$,
   which is due to Lemma \ref{l.emb.L2t}.
Hence, using (\ref{eq.deRham}) and the Stokes formula we get
\begin{eqnarray*}
\lefteqn{
   (H_\mu u + dp, u)_{L^2 (\mathbb{R}^n,\varLambda^1)}
}
\\
 & = &
   \lim_{R \to +\infty}
   \Big(
   \frac{1}{2}\, \partial_t \| u (\cdot, t) \|^2_{L^2 (B_R,\varLambda^1)}
 + \mu \| du (\cdot, t) \|^2_{L^2 (B_R,\varLambda^2)}
 + \mu \| d^\ast u (\cdot, t) \|^2_{L^2 (B_R)}
\\
 &   &
   \ \ \ \ \ \ \ \ \ \ \ \ + \
   (p, d^\ast u)_{L^2 (B_R)}
 - \int_{\partial B_R} u^\ast \Big( \frac{\partial u}{\partial \nu} - p \nu \Big) ds
   \Big)
\\
 & = &
   \frac{1}{2}\, \partial_t \| u (\cdot, t )\|^2_{L^2 (\mathbb{R}^n, \varLambda^1)}
 + \mu \| du (\cdot, t) \|^2_{L^2 (\mathbb{R}^n,\varLambda^2)},
\end{eqnarray*}
for
   $u' u$ is of class $C^{\mathbf{s} (0,0,2 \delta+1)}$,
   $p u$ is of class $C^{\mathbf{s} (0,0,2 \delta-1)}$
in the layer $\overline{\mathcal{C}_T} $ by Lemma \ref{l.product} and
$$
   R^{n-1-(2 \delta-1)} = R^{n-2 \delta}  \to 0
$$
if $R \to +\infty$.

Furthermore, the restrictions which we put on the forms
   $g^{(0)}$ and
   $v^{(1)}$,
   $v^{(2)}$
guarantee that the integrals
$$
\begin{array}{ccc}
 \!   \! (\ast (\ast g^{(0)} \wedge u), u)_{L^2 (\mathbb{R}^n,\varLambda^1)},
 \! & \! (\ast (\ast du \wedge v^{(1)}), u)_{L^2 (\mathbb{R}^n,\varLambda^1)},
 \! & \! (d \ast (\ast v^{(2)} \wedge u), u)_{L^2 (\mathbb{R}^n,\varLambda^1)}
\end{array}
$$
converge for each $t \in [0,1]$, for the integrands belong to
   $C^{\mathbf{s} (0,0,2 \delta)} (\overline{\mathcal{C}_T})$
(see Lemmata \ref{l.emb.L2t} and \ref{l.product}).
Hence, (\ref{eq.NS.deriv.0}) implies
$$
   \frac{1}{2}\,
   \partial_t \| u (\cdot, t) \|^2_{L^2 (\mathbb{R}^n,\varLambda^1)}
 + \mu\, \| du (\cdot, t) \|^2_{L^2 (\mathbb{R}^n,\varLambda^2)}
 = - (V_0 u (\cdot,t), u (\cdot,t))_{L^2 (\mathbb{R}^n,\varLambda^1)}.
$$
Since $g^{(0)}$ is of class $C^{\mathbf{s} (0,0,0)}$ in the layer $\overline{\mathcal{C}_T} $,
there is a constant $c$ independent of
   $u \in C^{\mathbf{s} (0,0,\delta)} (\overline{\mathcal{C}_T} , \varLambda^1)$,
such that
$$
   (\ast (\ast g^{(0)} \wedge u, u) _{L^2 (\mathbb{R}^n, \varLambda^1)}
 \leq
    c\, \| u \|^2_{L^2 (\mathbb{R}^n, \varLambda^1)}
$$
for all $t \in [0,T]$.
Besides,
\begin{eqnarray*}
\lefteqn{
   (\ast (\ast du \wedge v^{(1)}), u)_{L^2 (\mathbb{R}^n, \varLambda^1)}
}
\\
 & \leq &
   c\,
   \| v^{(1)} \|_{C^{\mathbf{s} (0,0,0)} (\overline{\mathcal{C}_T} , \varLambda^1)}
   \| du \| _{L^2 (\mathbb{R}^n, \varLambda^2)}
   \| u \| _{L^2 (\mathbb{R}^n, \varLambda^1)}
\\
 & \leq &
   b \mu \| du \|^2 _{L^2 (\mathbb{R}^n, \varLambda^2)}
 + \frac{c^2}{4 b \mu}\,
   \| v^{(1)} \|^2_{C^{\mathbf{s} (0,0,0)} (\overline{\mathcal{C}_T} , \varLambda^1)}\,
   \| u \|^2_{L^2 (\mathbb{R}^n, \varLambda^1)}
\end{eqnarray*}
with any constant $b > 0$, for $2 b_1 b_2 \leq b_1^2 + b_2^2$.
Finally we integrate by parts to observe that
$$
   (d (\ast v^{(2)} \wedge u), u)_{L^2 (\mathbb{R}^n,\varLambda^1)}
 = (\ast (\ast v^{(2)} \wedge u), d^\ast u)_{L^2 (\mathbb{R}^n)}
 = 0.
$$
On combining these estimates we see that there is a constant $c > 0$ independent of $u$
and $t \in [0,T]$, such that
\begin{equation}
\label{eq.uniq.5}
   -\, (V_0 u (\cdot,t), u (\cdot,t))_{L^2 (\mathbb{R}^n, \varLambda^1)}
 \leq
   c\, \| u (\cdot,t) \|_{L^2 (\mathbb{R}^n, \varLambda^1)}^2
 + \mu\, \| du (\cdot,t) \|^2_{L^2 (\mathbb{R}^n, \varLambda^2)}
\end{equation}
for all
$
   u
 \in
   C^{\mathbf{s} (1,0,\delta)} (\overline{\mathcal{C}_T} , \varLambda^1) \cap \mathcal{S}_{d^\ast}.
$
It follows that
$$
   \frac{1}{2}\,
   \partial_t\,
   \| u (\cdot, t) \|_{L^2 (\mathbb{R}^n, \varLambda^1)}^2
 \leq
   c\, \| u (\cdot,t) \|_{L^2 (\mathbb{R}^n,\varLambda^1)}^2
$$
for all $t \in [0,T]$.

Now we note that the from the inequality
   $x' (t) \leq a (t) x (t)$
for all $t$ in some interval of the real axis it follows that
$$
   \frac{d}{dt} \left( e^{- A (t)} x (t) \right) \leq 0,
$$
where $A$ is a primitive function for $a$.
Therefore, since $A (t) = 2c\, t$  is a primitive for the function  $a (t) = 2c$, we conclude
that
$$
   \frac{d}{dt} \Big( e^{- 2c\, t} \| u (\cdot, t) \|^2_{L^2 (\mathbb{R}^n, \varLambda^1)} \Big)
 \leq
   0
$$
for all $t \in [0,T]$.

Pick any $t \in (0,T]$.
Then
\begin{eqnarray*}
   \int_0^t
   \frac{d}{ds}
   \Big( e^{-2c\, s} \| u (\cdot, s) \|^2_{L^2 (\mathbb{R}^n, \varLambda^1)} \Big) ds
 & = &
   e^{-2c\, t} \| u (\cdot,t) \|^2_{L^2 (\mathbb{R}^n,\varLambda^1)}
 - \| u (\cdot,0) \|^2_{L^2 (\mathbb{R}^n, \varLambda^1)}
\\
 & = &
   e^{- 2c\, t} \| u (\cdot,t) \|^2_{L^2 (\mathbb{R}^n, \varLambda^1)}
\\
 & \leq &
   0
\end{eqnarray*}
because $u (x,0) = 0$ for all $x \in \mathbb{R}^n$.
Thus,
$$
   \| u (\cdot,t) \|^2_{L^2 (\mathbb{R}^n, \varLambda^1)}
 \leq
   0
$$
for all $t \in [0,T]$, i.e., $u \equiv 0$.
Hence it follows that $dp = 0$, i.e., $p$ does not depend on $x$.
However, the function $|p (t)|$ is dominated by
$
   (1+|x|^2)^{-(\delta-1)/2}
$
as $|x| \to +\infty$.
Since
   $\delta > n/2 \geq 1$
we deduce readily that $p (t)$ is identically equal to zero,
   as desired.
\end{proof}

As is already mentioned, the scale of weighted spaces
   $C^{k, \mathbf{s} (s,\lambda,\delta)}$
in the layer $\overline{\mathcal{C}_T} $ does not fully agree with the dilation principle for
parabolic equations.
Differentiation in the time variable $t$ does not lead to increasing the weight exponent
$\delta$ which results in committing a violation of compact embeddings.
In order to get rid of this shortage we go to slightly modify the above scale by introducing
an additional H\"{o}lder exponent $\lambda'$ which should exceed $\lambda$ and thus affect to
a gain of ``smoothness'' in $t$.
This manipulation of function spaces seems to be justified by the refined structure of the
Navier-Stokes equations.
For
   $s, k \in \mathbb{Z}_{\geq 0}$ and
   $0 < \lambda < \lambda' < 1$,
we introduce
$$
   \mathcal{F}^{k,\mathbf{s} (s,\lambda,\lambda',\delta)} (\overline{\mathcal{C}_T})
 :=
   C^{k+1,\mathbf{s} (s,\lambda,\delta)} (\overline{\mathcal{C}_T}) \cap
   C^{k,\mathbf{s} (s,\lambda',\delta)} (\overline{\mathcal{C}_T}).
$$
When given the norm
$$
   \| u \|_{\mathcal{F}^{k,\mathbf{s} (s,\lambda,\lambda',\delta)} (\overline{\mathcal{C}_T})}
 :=
   \| u \|_{C^{k+1,\mathbf{s} (s,\lambda,\delta)} (\overline{\mathcal{C}_T})}
 + \| u \|_{C^{k,\mathbf{s} (s,\lambda',\delta)} (\overline{\mathcal{C}_T})}.
$$
this is obviously a Banach space.
To certain extent these spaces are similar to those with two-norm convergence which are
of key importance for ill-posed problems.

If
   $s \in \mathbb{Z}_{\geq 0}$ and
   $0 \! < \! \lambda \! < \! \lambda' \! < \! 1$,
then all statements above on the
   Laplace operator,
   de Rham complex and
   heat operator
are true for the scale
   $\mathcal{F}^{k,\mathbf{s} (s,\lambda,\lambda',\delta)} (\overline{\mathcal{C}_T})$
instead of the scale
   $C^{k,\mathbf{s} (s,\lambda,\delta)} (\overline{\mathcal{C}_T})$.
The range
   $R^{k,\mathbf{s} (s,\lambda,\lambda',\delta)} (\overline{\mathcal{C}_T})$
of the operator
$$
   \varDelta :\,
   \mathcal{F}^{k,\mathbf{s} (s,\lambda,\lambda',\delta)} (\overline{\mathcal{C}_T}) \cap \mathcal{D}_{\varDelta}
 \to
   \mathcal{F}^{k,\mathbf{s} (s,\lambda,\lambda',\delta+2)} (\overline{\mathcal{C}_T})
$$
just amounts to the whole space
   $\mathcal{F}^{k,\mathbf{s} (s,\lambda,\lambda',\delta+2)} (\overline{\mathcal{C}_T})$,
if $0 < \delta < n-2$, and it reduces to the intersection
$
   R^{k+1,\mathbf{s} (s,\lambda,\delta+2)} (\overline{\mathcal{C}_T}) \cap
   R^{k,\mathbf{s} (s,\lambda',\delta+2)} (\overline{\mathcal{C}_T}),
$
if $\delta$ belongs to an interval
   $n-2+m < \delta < n-1+m$.

\begin{lemma}
\label{l.mathfrak.compact}
Let
   $s$ be a positive integer,
   $k \in \mathbb{Z}_{\geq 0}$,
   $0 < \lambda < \lambda' < 1$
and
   $\delta > \delta'$.
Then the embedding
$$
   \mathcal{F}^{k,\mathbf{s} (s,\lambda,\lambda',\delta)} (\overline{\mathcal{C}_T} , \varLambda^q)
 \hookrightarrow
   \mathcal{F}^{k+1,\mathbf{s} (s-1,\lambda,\lambda',\delta')} (\overline{\mathcal{C}_T} , \varLambda^q)
 $$
 is compact.
\end{lemma}

\begin{proof}
By abuse of notation we omit the domain and target bundle in designations.
By Theorem \ref{t.emb.hoelder.t},

1)
the space
   $C^{k+1,\mathbf{s} (s,\lambda,\delta)}$
is  embedded compactly into
   $C^{k+1,\mathbf{s} (s-1,\lambda',\delta')}$,
since
   $s+\lambda > s-1+\lambda'$ and
   $\delta > \delta'$;

2)
the space
   $C^{k,\mathbf{s} (s,\lambda',\delta)}$
is  embedded compactly into
   $C^{k,\mathbf{s} (s,\lambda,\delta')}$,
for
   $0 < \lambda < \lambda'$ and
   $\delta > \delta'$;

3)
the space
   $C^{k,\mathbf{s} (s,\lambda,\delta')}$
is embedded continuously into
   $C^{k+2,\mathbf{s} (s-1,\lambda,\delta')}$.

Hence it follows that if $S$ is a bounded set in
$$
   \mathcal{F}^{k,\mathbf{s} (s,\lambda,\lambda',\delta)}
 =
   C^{k+1,\mathbf{s} (s,\lambda,\delta)} \cap
   C^{k,\mathbf{s} (s,\lambda',\delta)},
$$
then any sequence from $S$ has a subsequence  converging in the space
$$
   \mathcal{F}^{k+1,\mathbf{s} (s-1,\lambda,\lambda',\delta')}
 =
   C^{k+2,\mathbf{s} (s-1,\lambda,\delta')} \cap
   C^{k+1,\mathbf{s} (s-1,\lambda',\delta')},
$$
as desired.
\end{proof}

\begin{corollary}
\label{c.NSL.hoelder0}
Let
   $n \geq 2$,
   $s$ be a positive integer,
   $k \in \mathbb{Z}_{\geq 0}$
and
   $\gamma \geq 0$.

1)
If
   $0 < \lambda \leq 1$
and
   the coefficients of $v^{(1)}$ are of class $C^{k,\mathbf{s} (s-1,\lambda,\gamma-1)}$,
   the coefficients of $v^{(2)}$ are of class $C^{k+1,\mathbf{s} (s-1,\lambda,\gamma-1)}$ and
   the coefficients of $g^{(0)}$ are of class $C^{k,\mathbf{s} (s-1,\lambda,\gamma)}$
in $\overline{\mathcal{C}_T} $,
   then the operator $\mathcal{A}_{V_0}$ induces a bounded linear operator
$$
   \begin{array}{c}
   C^{k,\mathbf{s} (s,\lambda,\delta)} (\overline{\mathcal{C}_T} ,\varLambda^1) \cap
   \mathcal{S}_{d^\ast}
\\
   \oplus
\\
   C^{k+1,\mathbf{s} (s-1,\lambda,\delta+\gamma-1)} (\overline{\mathcal{C}_T})
   \end{array}
 \to
   \begin{array}{c}
   C^{k,\mathbf{s} (s-1,\lambda,\delta+\gamma)} (\overline{\mathcal{C}_T} ,\varLambda^1)
\\
   \oplus
\\
   C^{2s+k,\lambda,\delta} (\mathbb{R}^n, \varLambda^1) \cap \mathcal{S}_{d^\ast}.
   \end{array}
$$

2)
If
   $0 < \lambda < \lambda' \leq 1$
and
   the coefficients of $v^{(1)}$ are of class $\mathcal{F}^{k,\mathbf{s} (s-1,\lambda,\lambda',\gamma-1)}$,
   the coefficients of $v^{(2)}$ are of class $\mathcal{F}^{k+1,\mathbf{s} (s-1,\lambda,\lambda',\gamma-1)}$ and
   the coefficients of $g^{(0)}$ are of class $\mathcal{F}^{k,\mathbf{s} (s-1,\lambda,\lambda',\gamma)}$
in $\overline{\mathcal{C}_T} $,
   then the operator $\mathcal{A}_{V_0}$ induces a bounded linear operator
$$
   \begin{array}{c}
   \mathcal{F}^{k,\mathbf{s} (s,\lambda,\lambda',\delta)} (\overline{\mathcal{C}_T} ,\varLambda^1) \cap
   \mathcal{S}_{d^\ast}
\\
   \oplus
\\
   \mathcal{F}^{k+1,\mathbf{s} (s-1,\lambda,\lambda',\delta+\gamma-1)} (\overline{\mathcal{C}_T})
   \end{array}
 \to
   \begin{array}{c}
   \mathcal{F}^{k,\mathbf{s} (s-1,\lambda,\lambda',\delta+\gamma)} (\overline{\mathcal{C}_T} ,\varLambda^1)
\\
   \oplus
\\
   C^{2s+k+1,\lambda,\delta} (\mathbb{R}^n,\varLambda^1) \cap \mathcal{S}_{d^\ast}.
   \end{array}
$$

3)
If moreover
   $\delta > n/2 $ and
   the coefficients of $v^{(1)}$ are of class $C^{\mathbf{s} (0,0,0)}$ in $\overline{\mathcal{C}_T} $,
then the null-space of the operators consists of all pairs $(0,c)^T$, where $c$ is a constant.
\end{corollary}

\begin{proof}
This follows from Lemma \ref{l.map.hoelder.V0} and Theorem \ref{t.NS.deriv.unique}.
\end{proof}

Part 3) of the corollary just amounts to saying that the pressure $p$ is defined up to a real
constant.

Consider now the operators
\begin{eqnarray*}
   W_0 f
 & = &
   d \ast (\ast g^{(0)} \wedge d^\ast (\varPhi \otimes I) f)
 + d \ast (\ast f \wedge v^{(1)}),
\\
   U_0 f
 & =
 & \ast (\ast g^{(0)} \wedge d^\ast (\varPhi \otimes I) f) + \ast (\ast f \wedge v^{(1)})
\end{eqnarray*}
which map 2\,-forms on $\overline{\mathcal{C}_T} $ into 2\,- and 1\,-forms, respectively.
It follows from (\ref{eq.deRham}) and Corollary \ref{c.deRham.Hoelder.tt} that
   $d (U_0 f) = W _0  f$
for
   $f \in C^{k\!-\!1,\mathbf{s} (s,\lambda,\delta\!+\!1)} (\overline{\mathcal{C}_T} ,\varLambda^2)$,
and
\begin{equation}
\label{eq.VW}
   d (V_0 u)
 = d \ast (\ast g^{(0)}  \wedge d^\ast (\varPhi \otimes I) du) + d \ast (\ast du \wedge v^{(1)})
 = W_0 (du)
\end{equation}
for all $u \in C^{k,\mathbf{s} (s,\lambda,\delta)} (\overline{\mathcal{C}_T} ,\varLambda^1)$.
Equality (\ref{eq.VW}) can be equivalently reformulated by saying that the pair
   $\{ V_0, W_0 \}$
is a homomorphism of the de Rham complex at steps $1$ and $2$.

Using these operators allows us to pass for the study of nonlinear Navier-Stokes equations
to a weakly perturbed Cauchy problem for the heat equation in the scale
   $C^{k-1,\mathbf{s} (s,\lambda,\delta+1)} (\overline{\mathcal{C}_T} , \varLambda^2)$
rather than to a linearisation of the Navier-Stokes equations in the scale
   $C^{k,\mathbf{s} (s,\lambda,\delta)} (\overline{\mathcal{C}_T} , \varLambda^1)$.
Our next concern will be to describe this trick.

\begin{lemma}
\label{l.map.W0}
Let
   $s$ and $k \geq 2$ be positive integers,
   $0 < \lambda < \lambda' < 1$,
and
   $1 \leq \delta < n$ be different from $n-2$, $n-1$.

1)
If
   the coefficients of $g^{(0)}$ are of class $C^{k,\mathbf{s} (s-1,\lambda,\delta+1)}$ and
   the coefficients of $v^{(1)}$ are of class $C^{k,\mathbf{s} (s-1,\lambda,\delta)}$
in the layer $\overline{\mathcal{C}_T} $ then the linear operators
\begin{equation}
\label{eq.map.W0.hat.C}
\begin{array}{rrcl}
   U_0\, :
 & R^{k-1,\mathbf{s} (s,\lambda,\delta+1)} (\overline{\mathcal{C}_T} , \varLambda^2)
 & \to
 & C^{k,\mathbf{s} (s-1,\lambda,\delta+2)} (\overline{\mathcal{C}_T} , \varLambda^1),
\\
   W_0\, :
 & R^{k-1,\mathbf{s} (s,\lambda,\delta+1)} (\overline{\mathcal{C}_T} , \varLambda^2)
 & \to
 & C^{k-1,\mathbf{s} (s-1,\lambda,\delta+3)} (\overline{\mathcal{C}_T} , \varLambda^2)
\end{array}
\end{equation}
are bounded.

2)
If, moreover, $1 \! < \! \delta \! < \! n$ and
   the coefficients of $g^{(0)}$ are of class $\mathcal{F}^{k,\mathbf{s} (s-1,\lambda,\lambda',\delta+1)}$,
   the coefficients of $v^{(1)}$ are of class $\mathcal{F}^{k,\mathbf{s} (s-1,\lambda,\lambda',\delta)}$
in the layer $\overline{\mathcal{C}_T} $ then the operators
\begin{equation}
\label{eq.map.W0.hat}
\begin{array}{rrcl}
   U_0\, :
 & R^{k-1,\mathbf{s} (s,\lambda,\lambda',\delta+1)} (\overline{\mathcal{C}_T} , \varLambda^2)
 & \!\! \to \!\!
 & \mathcal{F}^{k,\mathbf{s} (s-1,\lambda,\lambda',\delta+2)} (\overline{\mathcal{C}_T} , \varLambda^1),
\\
   W_0\, :
 & R^{k-1,\mathbf{s} (s,\lambda,\lambda',\delta+1)} (\overline{\mathcal{C}_T} , \varLambda^2)
 & \!\! \to \!\!
 & \mathcal{F}^{k-1,\mathbf{s} (s-1,\lambda,\lambda',\delta+3)} (\overline{\mathcal{C}_T} , \varLambda^2)
\end{array}
\end{equation}
are compact.
\end{lemma}

\begin{proof}
Let
   $g^{(0)}$ and
   $v^{(1)}$
satisfy the hypotheses listed in 1).
Pick any $\delta'$ such that $1 \leq \delta' \leq \delta$.
Then, according to Lemmata \ref{l.product} and \ref{l.weight.Hoelder.Laplace.t}, we get
\begin{eqnarray}
\label{eq.W.est.0}
\lefteqn{
   \| U_0 f \|_{C^{k,\mathbf{s} (s-1,\lambda,\delta+2)} (\cdot)}
}
\nonumber
\\
 & \leq &
   c
   \left(
   N (g^{(0)})\,
   \| d^\ast (\varPhi \otimes I) f \|_{C^{k,\mathbf{s} (s-1,\lambda,\delta')} (\cdot)}
 + N (v^{(1)})\,
   \| f \|_{C^{k,\mathbf{s} (s-1,\lambda,\delta'+1)} (\cdot)}
   \right)
\nonumber
\\
 & \leq &
   c
   \left(
   N (g^{(0)})\,
   \| f \|_{C^{k,\mathbf{s} (s-1,\lambda,\delta'+1)} (\cdot)}
 + N (v^{(1)})\,
   \| f  \|_{C^{k,\mathbf{s} (s-1,\lambda,\delta'+1)} (\cdot)}
   \right)
\nonumber
\\
\end{eqnarray}
for all
   $f \in C^{k-1,\mathbf{s} (s,\lambda,\delta+1)} (\overline{\mathcal{C}_T} , \varLambda^2)$,
where we omit the domain and target bundles for short.
The constant $c$ depends neither on $g^{(0)}$ and $v^{(1)}$ nor on $f$ and it may be different
in diverse applications, and
$$
\begin{array}{rcl}
   N (g^{(0)})
 & =
 & \|  g^{(0)} \|_{C^{k,\mathbf{s} (s-1,\lambda,\delta-\delta'+2)} (\cdot)},
\\
   N (v^{(1)})
 & =
 & \| v^{(1)} \|_{C^{k,\mathbf{s} (s-1,\lambda,\delta-\delta'+1)} (\cdot)}.
\end{array}
$$
Note that
   $N (g^{(0)})$ and
   $N (v^{(1)})$
are dominated by the norms
   $\|  g^{(0)} \|_{C^{k,\mathbf{s} (s-1,\lambda,\delta+1)} (\cdot)}$ and
   $\| v^{(1)} \|_{C^{k,\mathbf{s} (s-1,\lambda,\delta)} (\cdot)}$,
respectively, for the inequality
   $\delta - \delta' + 2 \leq \delta + 1$
is equivalent to $1 \leq \delta'$.
As the space
   $C^{k-1,\mathbf{s} (s,\lambda,\delta'+1)} (\cdot)$ is embedded continuously into
   $C^{k,\mathbf{s} (s-1,\lambda,\delta'+1)} (\cdot)$,
we see that
$$
   \| U_0 f \|_{C^{k,\mathbf{s} (s-1,\lambda,\delta+2)} (\cdot)}
 \leq
   c\,
   \| f \|_{C^{k-1,\mathbf{s} (s,\lambda,\delta'+1)} (\cdot)}
$$
with $c$ a constant independent of $f$.
In particular, for
   $\delta' = \delta$ and
   $1 \leq \delta < n$,
we derive the boundedness of the operator $U_0$ in (\ref{eq.map.W0.hat.C}).
Using Lemma \ref{l.homo.map.t} we conclude that the operator $W_0$ in (\ref{eq.map.W0.hat.C})
is bounded, too, for
   $d U_0 = W_0$.
This completes the proof of part 1).

In part 2) we assume that $\delta > 1$.
Then there is a $\delta' \geq 1$ such that $\delta > \delta'$.
According to Lemma \ref{l.mathfrak.compact}, if $S$ is bounded set in
$$
   \mathcal{F}^{k-1,\mathbf{s} (s,\lambda,\lambda',\delta+1)} (\cdot)
 = C^{k,\mathbf{s} (s,\lambda,\delta+1)} (\cdot) \cap
   C^{k-1,\mathbf{s} (s,\lambda',\delta+1)} (\cdot),
$$
then there is a sequence $\{ f_\nu \}$ in $S$ which converges in the space
$$
   \mathcal{F}^{k,\mathbf{s} (s-1,\lambda,\lambda',\delta''+1)} (\cdot)
 = C^{k+1,\mathbf{s} (s-1,\lambda,\delta'+1)} (\cdot) \cap
   C^{k,\mathbf{s} (s-1,\lambda',\delta'+1)} (\cdot)
$$
to a limit $f$.
By (\ref{eq.W.est.0}),
\begin{eqnarray}
\label{eq.W.est.00}
\lefteqn{
   \| U_0 (f_\nu - f) \|_{C^{k+1,\mathbf{s} (s-1,\lambda,\delta+2)} (\cdot)}
}
\nonumber
\\
 & \leq &
   c
   \Big(
   \|  g^{(0)} \|_{C^{k+1,\mathbf{s} (s-1,\lambda,\delta+1)} (\cdot)}
   \| f_\nu - f \|_{C^{k+1,\mathbf{s} (s-1,\lambda,\delta'+1)} (\cdot)}
\nonumber
\\
 & &
   +\
   \| f_\nu - f \|_{C^{k+1,\mathbf{s} (s-1,\lambda,\delta'+1)} (\cdot)}
   \| v^{(1)} \|_{C^{k+1,\mathbf{s} (s-1,\lambda,\delta)} (\cdot)}
   \Big)
\nonumber
\\
 & \to &
   0
\nonumber
\\
\end{eqnarray}
as $\nu \to \infty$.
On the other hand, using (\ref{eq.W.est.0}) with $\lambda'$ instead of $\lambda$, we obtain
\begin{eqnarray*}
\lefteqn{
   \| U_0 (f_\nu - f) \|_{C^{k,\mathbf{s} (s-1,\lambda',\delta+2)} (\cdot)}
}
\\
 & \leq &
   c
   \Big(
   \|  g^{(0)} \|_{C^{k,\mathbf{s} (s-1,\lambda',\delta+1)} (\cdot)}
   \| f_\nu - f \|_{C^{k,\mathbf{s} (s-1,\lambda',\delta'+1)} (\cdot)}
\\
 & &
   +\
   \| f_\nu - f \|_{C^{k,\mathbf{s} (s-1,\lambda',\delta'+1)} (\cdot)}
   \| v^{(1)} \|_{C^{k,\mathbf{s} (s-1,\lambda',\delta)} (\cdot)}
   \Big)
\\
 & \to &
   0
\end{eqnarray*}
as $\nu \to \infty$.

We have thus proved that the sequence $\{ U_0 f_\nu \}$ converges to $U_0 f$ in the norm of
the space
   $\mathcal{F}^{k,\mathbf{s} (s-1,\lambda,\lambda',\delta+2)} (\cdot)$.
Hence it follows that the map $U_0$ in (\ref{eq.map.W0.hat}) is compact.
Then Lemma \ref{l.diff.oper} implies that the map $W_0$ in (\ref{eq.map.W0.hat}) is compact,
too, because $W_0 = d U_0$.
\end{proof}

\begin{lemma}
\label{l.sol.psi}
Let
   $s \geq 1$ and $k \geq 2$ be integers,
   $1 \leq \delta < n$ be different from $n-2$ and $n-1$,
and
   the coefficients of $g^{(0)}$ be of class $C^{k,\mathbf{s} (s-1,\lambda,\delta+1)}$,
   the coefficients of $v^{(1)}$ be of class $C^{k,\mathbf{s} (s-1,\lambda,\delta)}$
in $\overline{\mathcal{C}_T} $.
If
$
   g_0 \in  R^{k-1,\mathbf{s} (s,\lambda,\delta+1)} (\overline{\mathcal{C}_T} , \varLambda^2) \cap
            \mathcal{S}_d
$
then any two-form $g$ on $\mathbb{R}^n$ with coefficients in
   $R^{k-1,\mathbf{s} (s,\lambda,\delta+1)} (\overline{\mathcal{C}_T})$
satisfying
\begin{equation}
\label{eq.heat.pseudo}
   g + \varPsi_\mu W_0 g = g_0
\end{equation}
belongs to
$
   R^{k-1,\mathbf{s} (s,\lambda,\delta+1)} (\overline{\mathcal{C}_T} , \varLambda^2) \cap
          \mathcal{S}_d.
$
\end{lemma}

\begin{proof}
Let
   $g \in C^{k-1,\mathbf{s} (s,\lambda,\delta+1)} (\overline{\mathcal{C}_T} ,\varLambda^2)$
be a solution to equation (\ref{eq.heat.pseudo}).
If $n = 2$, then $dg = 0$ because $d^2$ vanishes identically.
If $n \geq 3$, then from (\ref{eq.deRham}) and (\ref{eq.commute}) it follows that
   $H_\mu d = d H_\mu$,
   $d \varPsi_\mu = \varPsi_\mu d$,
   $d W_0 g = 0$,
and so $dg = 0$, as desired.
\end{proof}

\begin{lemma}
\label{l.reduce.psi}
Assume that
   $s$ and $k \geq 2$ are positive integers,
   $0 \! < \! \delta \! < \! n$ is different from $n-2$, $n-1$,
and
   the coefficients of $g^{(0)}$ are of class $C^{k,\mathbf{s} (s-1,\lambda,\delta+1)}$,
   the coefficients of $v^{(1)}$ are of class $C^{k,\mathbf{s} (s-1,\lambda,\delta)}$, and
   the coefficients of $v^{(2)}$ are of class $C^{k+1,\mathbf{s} (s-1,\lambda,-1)}$
in the layer $\overline{\mathcal{C}_T} $.
Let moreover $F = (f,u_0)^T$ be an arbitrary pair of
$$
   C^{k,\mathbf{s} (s-1,\lambda,\delta)} (\overline{\mathcal{C}_T} , \varLambda^1)
 \times
   C^{2s+k,\lambda,\delta} (\mathbb{R}^n, \varLambda^1) \cap \mathcal{S}_{d^\ast}.
$$

1)
If
$
   U
 \! = \! (u,p)^T
 \in
   C^{k-1,\mathbf{s} (s,\lambda,\delta)} (\overline{\mathcal{C}_T} ,\varLambda^1) \cap \mathcal{S}_{d^\ast}
 \times
   C^{k-1,\mathbf{s} (s-1,\lambda,\delta-1)} (\overline{\mathcal{C}_T})
$
satisfies
   $du \in R^{k\!-\!1,\mathbf{s} (s,\lambda,\delta\!+\!1)} (\overline{\mathcal{C}_T},\varLambda^2)$
and
\begin{equation}
\label{eq.NS.lin.psi}
   A_{V_0} U = F
\end{equation}
then $g = du$ is a solution to equation (\ref{eq.heat.pseudo}) with
   $g_0 = \varPsi_{\mu,0} du_0 + \varPsi_\mu df$.

2)
Conversely, if
$
   g
 \in
   R^{k-1,\mathbf{s} (s,\lambda,\delta+1)} (\overline{\mathcal{C}_T} ,\varLambda^2) 
$
is a solution to equation (\ref{eq.heat.pseudo}) with
   $g_0 = \varPsi_{\mu,0} du_0 + \varPsi_\mu df$
in $R^{k-1,\mathbf{s} (s,\lambda,\delta+1)} (\overline{\mathcal{C}_T} ,\varLambda^2)$, then the pair
$$
\begin{array}{rcl}
   u
 & =
 & d^\ast (\varPhi \otimes I) g,
\\
   p
 & =
 & d^\ast (\varPhi \otimes I) \left( f - (H_\mu + V_0) u \right)
\end{array}
$$
belongs to
$
   C^{k\!-\!1,\mathbf{s} (s,\lambda,\delta)} (\overline{\mathcal{C}_T} ,\varLambda^1) \cap \mathcal{S}_{d^\ast}
 \times
   C^{k\!-\!1,\mathbf{s} (s\!-\!1,\lambda,\delta\!-\!1)} (\overline{\mathcal{C}_T}),
$
satisfies (\ref{eq.NS.lin.psi}) and, in addition,
   $du \in R^{k\!-\!1,\mathbf{s} (s,\lambda,\delta\!+\!1)} (\overline{\mathcal{C}_T},\varLambda^2)$.
\end{lemma}

There is a gap in the smoothness of $u$ and $p$ in Lemma~\ref{l.reduce.psi}.
It is caused by the lack of smoothing properties of the Newton potential $\varPhi \otimes I$ acting in
the parabolic H\"{o}lder spaces.
Lemma \ref{l.weight.Hoelder.Laplace.t} guarantees that $\varPhi \otimes I$ improves the smoothness in $x$
by one while one would like to have the gain $2$.
However, we were not able to prove this.

\begin{proof}
1)
Let $U = (u,p)^T$ be a solution to (\ref{eq.NS.lin.psi}).
By (\ref{eq.commute}), we get
   $d H_\mu u = H_\mu du$,
and so using (\ref{eq.VW}) yields
$$
   \Big\{
   \begin{array}{rclll}
     H_\mu du +  W_0 du
   & =
   & df
   & \mbox{in}
   & \overline{\mathcal{C}_T} ,
\\
     \gamma_0\, du
   & =
   & d u_0
   & \mbox{on}
   & \mathbb{R}^{n},
\end{array}
$$
the last equality being due to $d\, \gamma_0 u = \gamma_0\, du$.
It remains to apply
   Lemmata \ref{l.homo.map.t}, \ref{l.bound.heat.initial.hoelder}, \ref{l.heat.key1}
   and Theorem \ref{t.heat.key2}
to conclude that the two-form
   $g = du$
is of class $R^{k-1,\mathbf{s} (s,\lambda,\delta+1)} (\overline{\mathcal{C}_T})$ and satisfies
equation (\ref{eq.heat.pseudo}).
(Obviously, $g$ is closed in the layer.)

2)
Set
   $g_0 = \varPsi_{\mu,0} d u_0 + \varPsi_\mu df$.
Since
   $d \varPsi_{\mu,0} = \varPsi_{\mu,0} d$ and
   $d \varPsi_\mu = \varPsi_\mu d$,
it follows by Lemmata \ref{l.homo.map.t} and \ref{l.homo.map} that
$
   g_0
 \in
   R^{k-1,\mathbf{s} (s,\lambda,\delta+1)} (\overline{\mathcal{C}_T} ,\varLambda^2) \cap \mathcal{S}_d,
$
if $0 < \delta < n$.
Hence, any solution $g$ to (\ref{eq.heat.pseudo}) is in
   $R^{k-1,\mathbf{s} (s,\lambda,\delta+1)} (\overline{\mathcal{C}_T} ,\varLambda^2) \cap \mathcal{S}_d$
because of Lemma \ref{l.sol.psi}.

Now, Corollary \ref{c.deRham.Hoelder.t} implies that
   $u = d^\ast (\varPhi \otimes I) g$
is an one-form with coefficients in
   $R^{k-1,\mathbf{s} (s,\lambda,\delta)} (\overline{\mathcal{C}_T})$
satisfying $du = g$ in the layer.
Using Lemma \ref{l.bound.heat.initial.hoelder} and formula (\ref{eq.VW}) we see that
$$
   \Big\{
   \begin{array}{rclll}
      d \left( H_\mu u + V_0 u - f \right)
   & =
   & 0
   & \mbox{in}
   & \overline{\mathcal{C}_T} ,
\\
     d \left( \gamma_0 u - u_0 \right)
   & =
   & 0
   & \mbox{on}
   & \mathbb{R}^n.
   \end{array}
$$
As $0 < \delta < n$ and
$$
   H_\mu u + V_0 u - f
 \in
   C^{k-1,\mathbf{s} (s-1,\lambda,\delta)} (\overline{\mathcal{C}_T} ,\varLambda^1) \cap \mathcal{S}_d,
$$
Corollary \ref{c.deRham.Hoelder.t} shows that the function
$
   p = d^\ast (\varPhi \otimes I) \left( f - H_\mu u - V_0 u \right)
$
belongs to the space
$
   C^{k-1,\mathbf{s} (s-1,\lambda,\delta-1)} (\overline{\mathcal{C}_T})
$
and it satisfies
$$
   H_\mu u + V_0 u + dp = f
$$
in $\overline{\mathcal{C}_T} $.
Finally, since
$$
\begin{array}{rcl}
   d \left( \gamma_0 u - u_0 \right)
 & =
 & 0,
\\
   d^\ast \left( \gamma_0 u - u_0 \right)
 & =
 & 0,
\end{array}
$$
we get $\gamma_0 u = u_0$ in all of $\mathbb{R}^n$ because of Corollary \ref{c.deRham.Hoelder}.
Hence, the pair $U = (u,p)^T$ is a solution to (\ref{eq.NS.lin.psi}).
\end{proof}

\begin{corollary}
\label{c.invertible.psi}
Let
   $s$ and $k \geq 2$ be positive integers,
   $0 < \lambda < \lambda' < 1$,
   $n/2 < \delta < n$ be different from $n-2$, $n-1$,
and
   the coefficients of $g^{(0)}$ be of class $\mathcal{F}^{k,\mathbf{s} (s-1,\lambda,\lambda',\delta+1)}$,
   the coefficients of $v^{(1)}$ be of class $\mathcal{F}^{k,\mathbf{s} (s-1,\lambda,\lambda',\delta)}$
in the layer $\overline{\mathcal{C}_T} $.
Then the operator
\begin{equation}
\label{eq.heat.pseudo.d2}
   I \! + \! \varPsi_\mu W_0 :
   R^{k\!-\!1,\mathbf{s} (s,\lambda,\lambda',\delta\!+\!1)} (\overline{\mathcal{C}_T} , \varLambda^2) \cap
   \mathcal{S}_d
 \to
   R^{k\!-\!1,\mathbf{s} (s,\lambda,\lambda',\delta\!+\!1)} (\overline{\mathcal{C}_T} , \varLambda^2) \cap
   \mathcal{S}_d
\end{equation}
is continuously invertible.
\end{corollary}

\begin{proof}
First we observe by Lemmata \ref{l.homo.map.t}, \ref{l.map.W0} and \ref{l.sol.psi} that the operator
$I + \varPsi_\mu W_0$ is a continuous selfmapping of
$
   R^{k\!-\!1,\mathbf{s} (s,\lambda,\lambda',\delta\!+\!1)} (\overline{\mathcal{C}_T} , \varLambda^2) \cap
   \mathcal{S}_d.
$
Our next goal is to show that this mapping is one-to-one.

\begin{lemma}
\label{l.unique.pseudo}
Suppose that
   $s$ and $k \geq 2$ are positive integers,
   $n/2 < \delta < n$ is different from $n-2$, $n-1$,
and
   the coefficients of $g^{(0)}$ are of class $C^{k,\mathbf{s} (s-1,\lambda,\delta+1)}$,
   the coefficients of $v^{(1)}$ are of class $C^{k,\mathbf{s} (s-1,\lambda,\delta)}$
in the layer $\overline{\mathcal{C}_T} $.
Then any form
$
   g
 \in
   R^{k-1,\mathbf{s} (s,\lambda,\delta+1)} (\overline{\mathcal{C}_T} , \varLambda^2)
$
satisfying $(I + \varPsi_\mu W_0) g = 0$ is identically zero.
\end{lemma}

\begin{proof}
Indeed, Lemma \ref{l.sol.psi} yields readily
   $dg = 0$
in $\overline{\mathcal{C}_T} $.
Then using Corollary \ref{c.deRham.Hoelder.t} and equality (\ref{eq.VW}) we deduce that
the function
   $u = d^\ast (\varPhi \otimes I) g$
satisfies $du = g$ and
$$
   \Big\{
   \begin{array}{rclll}
     d (H_\mu u + V_0 u)
   & =
   & 0
   & \mbox{in}
   & \overline{\mathcal{C}_T} ,
\\
     d^\ast u
   & =
   & 0
   & \mbox{in}
   & \overline{\mathcal{C}_T}
   \end{array}
$$
whence
$$
   \begin{array}{rccclcc}
   d (\gamma_0 u)
 & =
 & \gamma_0 (g)
 & =
 & \gamma_0 (-\, \varPsi_\mu W_0 g)
 & =
 & 0,
\\
   d^\ast (\gamma_0 u)
 & =
 & \gamma_0 (d^\ast u)
 & =
 & 0
 &
 &
   \end{array}
$$
on $\mathbb{R}^n$.
According to (\ref{eq.deRham}), the last two equalities imply that $\gamma_0 u$ is a harmonic
one-form on $\mathbb{R}^n$.
As $\delta > 0$ it follows that $\gamma_0 u$ vanishes at the point at infinity and hence
$\gamma_0 u \equiv 0$ on $\mathbb{R}^n$.

Since $d (H_\mu u + V_0 u) = 0$, the function
$$
   p = - d^\ast (\varPhi \otimes I) \left( H_\mu u + V_0 u \right)
$$
belongs to the space
   $C^{k,\mathbf{s} (s-1,\lambda,\delta-1)} (\overline{\mathcal{C}_T})$
and satisfies
$
   H_\mu u + V_0 u + dp = 0
$
in $\overline{\mathcal{C}_T} $.
Therefore, the pair $U = (u,p)^T$ lies in the direct product
$$
   C^{k-2,\mathbf{s} (s,\lambda,\delta)} (\overline{\mathcal{C}_T} ,\varLambda^1) \cap \mathcal{S}_{d^\ast}
 \times
   C^{k,\mathbf{s} (s-1,\lambda,\delta-1)} (\overline{\mathcal{C}_T})
$$
and satisfies
$
   A_{V_0} U = 0.
$
Finally, by the uniqueness result of Theorem \ref{t.NS.deriv.unique} we get $u \equiv 0$, and so
$g \equiv 0$, too.
\end{proof}

We are now in a position to complete the proof of Corollary \ref{c.invertible.psi}.
According to Lemma \ref{l.map.W0}, the operator (\ref{eq.heat.pseudo.d2}) is Fredholm and
its index equals zero.
Then the statement of the corollary follows from Lemma \ref{l.unique.pseudo} and Fredholm theorems.
\end{proof}

\begin{corollary}
\label{c.NS.deriv.unique}
Assume that
   $s$ and $k \geq 2$ are positive integers,
   $0 < \lambda < \lambda' < 1$,
   $n/2 < \! \delta \! < n$ is different from $n-2$, $n-1$,
and
   the coefficients of $g^{(0)}$ are of class $\mathcal{F}^{k,\mathbf{s} (s-1,\lambda,\lambda',\delta+1)}$,
   the coefficients of $v^{(1)}$ are of class $\mathcal{F}^{k,\mathbf{s} (s-1,\lambda,\lambda',\delta)}$, and
   the coefficients of $v^{(2)}$ are of class $\mathcal{F}^{k+1,\mathbf{s} (s-1,\lambda,\lambda',-1)}$
in the layer $\overline{\mathcal{C}_T} $.
Then for any pair
$$
   F
 \! = \! (f,u_0)^T
 \in
   \mathcal{F}^{k,\mathbf{s} (s,\lambda,\lambda',\delta)} (\overline{\mathcal{C}_T} , \varLambda^1)
 \times
   C^{2s+k+1,\lambda,\delta} (\mathbb{R}^n, \varLambda^1) \cap \mathcal{S}_{d^\ast}
$$
there is a unique solution
$$
   U
 \! = \! (u,p)^T
 \in
   \mathcal{F}^{k\!-\!1,\mathbf{s} (s,\lambda,\lambda',\delta)} (\overline{\mathcal{C}_T} ,\varLambda^1) \cap
   \mathcal{S}_{d^\ast}
 \times
   \mathcal{F}^{k\!-\!1,\mathbf{s} (s\!-\!1,\lambda,\lambda',\delta\!-\!1)} (\overline{\mathcal{C}_T})
$$
to the equation $A_{V_0} U = F$ and the corresponding linear operator
   $U = A_{V_0}^{-1} (F)$
is bounded.
\end{corollary}

\begin{proof}
It follows from
   Corollary \ref{c.invertible.psi},
   Theorem  \ref{t.NS.deriv.unique} and
   Lemma \ref{l.reduce.psi},
for
\begin{eqnarray*}
   dF
 & \in &
   R^{k-1,\mathbf{s} (s,\lambda,\lambda',\delta+1)} (\overline{\mathcal{C}_T}, \varLambda^2)
 \times
   R^{2s+k,\lambda,\delta+1} (\mathbb{R}^n, \varLambda^2),
\\
   \varPsi_{\mu,0} d u_0 + \varPsi_\mu df
 & \in &
   R^{k-1,\mathbf{s} (s,\lambda,\lambda',\delta+1)} (\overline{\mathcal{C}_T}, \varLambda^2),
\end{eqnarray*}
the latter inclusion being due to Lemma \ref{l.heat.key1}.
\end{proof}

\section{The Navier-Stokes equations as an open map}
\label{s.NS.OpenMap}

We plan to prove that the Navier-Stokes equations can be treated as a nonlinear injective
Fredholm operator with open range in proper Banach spaces.
Recall that a nonlinear operator $\mathcal{A} : X \to Y$ in Banach spaces $X$, $Y$ is called
Fredholm if it has a Frech\'et derivative at each point $x_0 \in X$ and this derivative is a
Fredholm linear map from $X$ to $Y$ (see \cite{Sm65}).

For this purpose we set
$$
   \mathcal{A}
 = \Big( \begin{array}{cc}
           H_\mu + \mathbf{D}^1
         & d
\\
           \gamma_0
         & 0
         \end{array}
   \Big)
 = \mathcal{A}_0
 + \Big( \begin{array}{cc}
           \mathbf{D}^1
         & 0
\\
           0
         & 0
         \end{array}
   \Big).
$$
Clearly, the operator is well defined in the scale of weighted H\"older spaces introduced above.
For nonlinear Fredholm operators we may use a Sard theorem for Banach spaces (see \cite{Sm65})
and the Baire category theorem to obtain additional information on the range of $\mathcal{A}$.

Thus, given any data $F = (f,u_0)^T$, we look for a solution $U = (u,p)^T$ of nonlinear
Navier-Stokes equations
\begin{equation}
\label{eq.NS.D}
   \mathcal{A}\, \Big( \begin{array}{c} u \\ p \end{array} \Big)
 = \Big( \begin{array}{c} f \\ u_0 \end{array} \Big)
\end{equation}
in the scale of weighted H\"older spaces on $\overline{\mathcal{C}_T} $.

\begin{theorem}
\label{t.NS.unique}
Suppose
   $s \geq 1$ and $k$ are nonnegative integers,
   $n/2 < \delta$,
and
   the coefficients of $g^{(0)}$ are of class $C^{\mathbf{s} (0,0,0)}$,
   the coefficients of $v^{(1)}$ are of class $C^{\mathbf{s} (0,0,0)}$, and
   the coefficients of $v^{(2)}$ are of class $C^{1,\mathbf{s} (0,0,-1)}$
in the layer $\overline{\mathcal{C}_T} $.
Then, for each pair
$$
   F \! = \! (f,u_0)^T
 \in
   C^{k,\mathbf{s} (s-1,\lambda,\delta)} (\overline{\mathcal{C}_T} ,\varLambda^1)
 \times
   C^{2s+k,\lambda,\delta} (\mathbb{R}^n,\varLambda^1) \cap \mathcal{S}_{d^\ast}
$$
the nonlinear Navier-Stokes type equation
\begin{equation}
\label{eq.NS.D1}
   \mathcal{A}_{0} \Big( \begin{array}{c} u \\ p \end{array} \Big)
 + \Big( \begin{array}{c} V_0 u  \\ 0 \end{array} \Big)
 + a\, \Big( \begin{array}{c} \mathbf{D}^1 u \\ 0 \end{array} \Big)
= \Big( \begin{array}{c} f \\ u_0 \end{array} \Big),
\end{equation}
where $a$ is an arbitrary nonnegative real number, has at most one solution in the space
$$
   C^{k,\mathbf{s} (s,\lambda,\delta)} (\overline{\mathcal{C}_T} ,\varLambda^1) \cap \mathcal{S}_{d^\ast}
 \times
   C^{k+1,\mathbf{s} (s-1,\lambda,\delta-1)} (\overline{\mathcal{C}_T}).
$$
\end{theorem}

\begin{proof}
One may follow
   the original paper \cite{Lera34a} or
   the proofs of Theorem 3.2 for $n = 2$ and Theorem 3.4 for $n = 3$ in \cite{Tema79},
showing the uniqueness result by integration by parts.
Cf. also Theorem \ref{t.NS.deriv.unique}.

Indeed, let
   $(u',p'')$ and
   $(u'',p'')$
be any two solutions to (\ref{eq.NS.D}).
Then for the difference
   $(u,p) = (u' - u'', p' - p'')$
we get
\begin{equation}
\label{eq.NS.D0}
   \Big( \begin{array}{c} H_\mu u + dp \\ 0 \end{array} \Big)
 = a\, \Big( \begin{array}{c}  \mathbf{D}^1 u'' - \mathbf{D}^1 u' \\ 0 \end{array} \Big)
 - \Big( \begin{array}{c} V_0 u \\ 0 \end{array} \Big).
\end{equation}

It is easy to see that
$$
   \partial_t\, \| u (\cdot, t) \|^2_{L^2 (B_R,\varLambda^1)}
 = 2\, (\partial_t u, u)_{L^2 (B_R,\varLambda^1)}
$$
for all $t \in [0,T]$.
As
$$
\begin{array}{rcl}
   u
 & \in
 & C^{k,\mathbf{s} (s,\lambda,\delta)} (\overline{\mathcal{C}_T} ,\varLambda^1) \cap \mathcal{S}_{d^\ast},
\\
   p
 & \in
 & C^{k+1,\mathbf{s} (s-1,\lambda,\delta-1)} (\overline{\mathcal{C}_T})
\end{array}
$$
and $\delta > n/2$, the coefficients of the one-forms
   $u$,
   $\partial_i u$,
   $\partial_t u$,
   $H_\mu u$ and
   $dp$
are, by Lemma \ref{l.emb.L2t}, square integrable over all of $\mathbb{R}^n$ for each fixed $t \in [0,T]$.
Hence it follows that
\begin{eqnarray*}
\lefteqn{
   (H_\mu u + dp, u)_{L^2 (\mathbb{R}^n,\varLambda^1)}
}
\\
 & = &
   \lim_{R \to +\infty}
   \Big(
   \frac{1}{2}\, \partial_t \| u (\cdot, t) \|^2_{L^2 (B_R,\varLambda^1)}
 + \mu \| du (\cdot, t) \|^2_{L^2 (B_R,\varLambda^2)}
 + \mu \| d^\ast u (\cdot, t) \|^2_{L^2 (B_R)}
\\
 &   &
   \ \ \ \ \ \ \ \ \ \ \ \ + \
   (p, d^\ast u)_{L^2 (B_R)}
 - \int_{\partial B_R} u^\ast \Big( \frac{\partial u}{\partial \nu} - p \nu \Big) ds
   \Big)
\\
 & = &
   \frac{1}{2}\, \partial_t \| u (\cdot, t )\|^2_{L^2 (\mathbb{R}^n, \varLambda^1)}
 + \mu \sum_{i=1}^n \| \partial_i u (\cdot, t) \|^2_{L^2 (\mathbb{R}^n,\varLambda^1)}
\end{eqnarray*}
because
   $d^\ast u = 0$
and
   $R^{n-1-(2 \delta-1)} = R^{n-2 \delta}  \to 0$ if $R \to +\infty$.
We have used here the identity
\begin{equation}
\label{eq.Dirichlet}
   \| du (\cdot, t) \|^2_{L^2 (\mathbb{R}^n,\varLambda^2)}
 + \| d^\ast u (\cdot, t) \|^2_{L^2 (\mathbb{R}^n)}
 = \sum_{i=1}^n \| \partial_i u (\cdot, t) \|^2_{L^2 (\mathbb{R}^n,\varLambda^1)}
\end{equation}
which can be checked, for example, by the Plancherel theorem.

Furthermore, the integrals
   $(\mathbf{D}^1 u', u)_{L^2 (\mathbb{R}^n,\varLambda^1)}$ and
   $(\mathbf{D}^1 u'', u)_{L^2 (\mathbb{R}^n,\varLambda^1)}$
converge, for both
   $\mathbf{D}^1 u'$ and
   $\mathbf{D}^1 u''$
are one-forms with coefficients of class
   $C^{\mathbf{s} (0,0,\delta+1)}$
in the layer $\overline{\mathcal{C}_T} $ and
   $u \in C^{\mathbf{s} (0,0,\delta)} (\overline{\mathcal{C}_T} ,\varLambda^1)$
with $\delta > n/2$ (see Lemmata \ref{l.emb.L2t} and \ref{l.product}).
Therefore, (\ref{eq.NS.D0}) implies
\begin{eqnarray*}
\lefteqn{
   \frac{1}{2}\, \partial_t\, \| u (\cdot,t) \|^2_{L^2 (\mathbb{R}^n,\varLambda^1)}
 + \mu \sum_{i=1}^n \| \partial_i u (\cdot,t) \|^2_{L^2 (\mathbb{R}^n,\varLambda^1)}
}
\\
 & = &
 -\, a\,
   \left( (\mathbf{D}^1 u', u)_{L^2 (\mathbb{R}^n, \varLambda^1)}
        - (\mathbf{D}^1 u'', u)_{L^2 (\mathbb{R}^n, \varLambda^1)}
   \right)
 - (V _0 u, u)_{L^2 (\mathbb{R}^n,\varLambda^1)}.
\end{eqnarray*}
A trivial verification shows that
\begin{eqnarray*}
\lefteqn{
   (\mathbf{D}^1 u', u)_{L^2 (\mathbb{R}^n,\varLambda^1)}
 - (\mathbf{D}^1 u'', u)_{L^2 (\mathbb{R}^n,\varLambda^1)}
}
\\
 & = &
   \sum_{i,j=1}^n
   \int_{\mathbb{R}^n}
   \left( u'_i (\partial_i u'_j) u_j - u''_i (\partial_i u''_j) u_j \right) dx
\\
 & = &
   \sum_{i,j=1}^n
   \int_{\mathbb{R}^n}
   \left( (u'_i - u''_i) (\partial_i u'_j) u_j + u''_i (\partial_i (u'_j - u''_j)) u_j \right)
   dx
\\
 & = &
   \sum_{i,j=1}^n
   \int_{\mathbb{R}^n}
   \left( u_i (\partial_i u'_j) u_j + u''_i (\partial_i u_j) u_j \right)
   dx,
\end{eqnarray*}
where
$
   \displaystyle
   u = \sum_{j=1}^n u_j dx^j
$
and similarly for the forms $u'$ and $u''$.
As $d^\ast u'' = 0$, it follows that
\begin{eqnarray*}
\lefteqn{
   \sum_{i,j=1}^n
   \int_{\mathbb{R}^n}
   u''_i (\partial_i u_j) u_j\, dx
 = \sum_{j=1}^n
   \lim_{R \to +\infty}
}
\\
 & \times &
   \Big(
   \int_{B_R}
   \Big( d^\ast u'') (u_j u_j) - \sum_{i=1}^n u''_i (\partial_i u_j) u_j  \Big) dx
 + \sum_{i=1}^n \int_{\partial B_R} u''_i u_j u_j ds
   \Big)
\\
 & = &
 - \sum_{i,j=1}^n \int_{\mathbb{R}^n}
   u''_i  (\partial_i u_j) u_j\, dx,
\end{eqnarray*}
for $R^{n-1-3 \delta-1} \to 0$ if $R \to +\infty$.
Since we arrived at the same sum with opposite sign, we conclude that this sum vanishes.
On integrating by parts once again we obtain
$$
   \sum_{i,j=1}^n \int_{\mathbb{R}^n} u_i (\partial_i u'_j) u_j\, dx
 = \sum_{j=1}^n \int_{\mathbb{R}^n} (d^\ast u) (u'_j u_j)
 - \sum_{i,j=1}^n \int_{\mathbb{R}^n} u_i (\partial_i u_j) u'_j\, dx,
$$
and so
\begin{eqnarray*}
\lefteqn{
   \frac{1}{2}\, \partial_t\, \| u (\cdot,t) \|^2_{L^2 (\mathbb{R}^n,\varLambda^1)}
 + \mu \sum_{i=1}^n \| \partial_i u (\cdot,t) \|^2_{L^2 (\mathbb{R}^n,\varLambda^1)}
}
\\
 & = &
   a
   \sum_{i,j=1}^n
   \int_{\mathbb{R}^n} u_i (\cdot,t) (\partial_i u_j) (\cdot,t) u'_j (\cdot,t)\, dx
 - (V _0 u (\cdot,t), u (\cdot,t))_{L^2 (\mathbb{R}^n,\varLambda^1)}
\end{eqnarray*}
for all $t \in [0,T]$.

As both $u'$ and $u$ have coefficients in
   $C^{k,\mathbf{s} (s,\lambda,\delta)} (\overline{\mathcal{C}_T})$,
the pairwise products $u_i u'_j$ are of class
   $C^{\mathbf{s} (0,0,2 \delta)} (\overline{\mathcal{C}_T})$,
where $2 \delta > 0$ by assumption.
In particular, the functions
   $u_i (\cdot,t) u'_j (\cdot,t)$
are square integrable over $\mathbb{R}^n$ for each fixed $t \in [0,T]$
   (see Lemma \ref{l.emb.L2t}).
By the Cauchy-Schwarz inequality,
\begin{eqnarray*}
   \sum_{i,j=1}^n
   \int_{\mathbb{R}^n} u_i (\partial_i u_j) u'_j\, dx
 & \leq &
   \sum_{i,j=1}^n
   \| \partial_i u_j \|_{L^2 (\mathbb{R}^n)}
   \| u_i u_j \|_{L^2 (\mathbb{R}^n)}
\\
 & \leq &
   \Big( \sum_{i,j=1}^n \| \partial_i u_j \|^2_{L^2 (\mathbb{R}^n)} \Big)^{1/2}
   \Big( \sum_{i,j=1}^n \| u_i u'_j \|^2_{L^2 (\mathbb{R}^n)} \Big)^{1/2}
\end{eqnarray*}
whence
\begin{eqnarray*}
\lefteqn{
   a
   \sum_{i,j=1}^n
   \int_{\mathbb{R}^n} u_i (\cdot,t) (\partial_i u_j) (\cdot,t) u'_j (\cdot,t)\, dx
}
\\
 & \leq &
   b \mu
   \sum_{i,j=1}^n
   \| \partial_i u_j (\cdot, t) \|^2_{L^2 (\mathbb{R}^n)}
 + \frac{a^2}{4 b \mu}
   \sum_{i,j=1}^n \| u_i (\cdot,t) u'_j (\cdot,t) \|^2_{L^2 (\mathbb{R}^n)}
\end{eqnarray*}
with an arbitrary constant $b > 0$ independent of $u'$ and $u''$, for
   $2 b_1 b_2 \leq b_1^2 + b_2^2$.
Then, using estimate (\ref{eq.uniq.5}) for the term
   $(V_0 (u,u)_{L^2 (\mathbb{R}^n,\varLambda^1)}$
obtained in the proof of Theorem \ref{t.NS.deriv.unique}, and identity (\ref{eq.Dirichlet})
we conclude that
$$
   \frac{1}{2}\, \partial_t\, \| u (\cdot,t) \|^2_{L^2 (\mathbb{R}^n,\varLambda^1)}
 \leq
   \frac{a^2}{2 \mu}
   \sum_{i,j=1}^n
   \| u_i (\cdot,t) u'_j) (\cdot,t) \|^2_{L^2 (\mathbb{R}^n)}
 +  c\, \| u (\cdot,t) \|^2_{L^2 (\mathbb{R}^n,\varLambda^1)},
$$
for all $t \in [0,T]$, where the constant $c$ depends neither on $t$ nor on $u'$ and $u''$.
Since
   $u' \in C^{\mathbf{s} (0,0,\delta)} (\overline{\mathcal{C}_T} ,\varLambda^1)$,
there is another constant $C$ depending on $u'$ but not on $x$ and $t$, such that
$$
   |v (x,t)|^2 \leq C\, (1+|x|^2)^{-\delta}
$$
for all $(x,t) \in \overline{\mathcal{C}_T} $.
Hence it follows that
$$
   \frac{1}{2}\, \partial_t\, \| u (\cdot,t) \|^2_{L^2 (\mathbb{R}^n,\varLambda^1)}
\leq
   \Big( \frac{a^2}{2 \mu} C^2 + c \Big)\,
   \| u (\cdot,t) \|^2_{L^2 (\mathbb{R}^n,\varLambda^1)}
$$
for all $t \in [ 0,T]$.

\begin{remark} 
\label{r.Energy}
As is seen in the proof of Theorem \ref{t.NS.unique}, one can easily use the integration by parts 
for solutions to the Navier-Stokes equations in the weighted spaces for $\delta > n/2$. 
This means that the standard energy estimates are still valid for them  
   (see for instance \cite[(3.47)]{Tema79}).
\end{remark}

The rest of the proof runs in the same way as that of Theorem \ref{t.NS.deriv.unique}.
This leads immediately to
   $u \equiv 0$ and
   $p \equiv 0$,
i.e., the solutions $(u',p')$ and $(u'',p'')$ coincide, as desired.
\end{proof}

The nonlinear Navier-Stokes equations reduce in much the same way as the corresponding
linearised equations.
To this end we introduce nonlinear pseudodifferential operators
\begin{eqnarray*}
   \mathbf{D}^2 g
 & = &
   d \ast (\ast g \wedge d^\ast (\varPhi \otimes I) g),
\\
   \mathcal{Q} g
 & =
 & \ast (\ast g \wedge d^\ast (\varPhi \otimes I) g)
\end{eqnarray*}
which map two-forms on $\overline{\mathcal{C}_T} $ into two- and one-forms, respectively.
By the very construction, we have
   $d (\mathcal{Q} g) = \mathbf{D}^2 g$.

It follows from (\ref{eq.deRham}) and Corollary \ref{c.deRham.Hoelder.tt} that
\begin{equation}
\label{eq.VW.n}
   d\, \mathbf{D}^1 u
 = d \ast (\ast du \wedge u)
 = d \ast (\ast du \wedge  d^\ast (\varPhi \otimes I) du)
 = \mathbf{D}^2 (du)
\end{equation}
for all
   $u \in C^{k,\mathbf{s} (s,\lambda,\delta)} (\overline{\mathcal{C}_T} , \varLambda^1)$.
Equality (\ref{eq.VW.n}) means that the pair $\{ \mathbf{D}^1, \mathbf{D}^2 \}$ is a nonlinear
homomorphism of the de Rham complex at steps $1$ and $2$.
Using this homomorphism allows one to reduce the Navier-Stokes equations to a  nonlinear
Cauchy problem for the heat equation in the scale
   $C^{k\!-\!1,\mathbf{s} (s,\lambda,\delta\!+\!1)} (\overline{\mathcal{C}_T} ,\varLambda^2)$, 
	cf. \cite[Ch. 2]{BerMaj02}. 
We proceed with an explicit description.

\begin{lemma}
\label{l.map.G}
Let
   $s \geq 1$ and $k \geq 2$ be integers,
   $0 < \lambda < \lambda' < 1$,
and let
   $1 \leq \delta < n$ be different from $n-2$ and $n-1$.
Then the nonlinear operators
\begin{equation}
\label{eq.map.G.tilde}
\begin{array}{rrcl}
   \mathcal{Q} \, :
 \!\! & \!\! R^{k,\mathbf{s} (s-1,\lambda,\delta+1)} (\overline{\mathcal{C}_T} , \varLambda^2)
 \!\! & \!\! \to
 \!\! & \!\! C^{k,\mathbf{s} (s-1,\lambda,\delta+2)} (\overline{\mathcal{C}_T} , \varLambda^1),
\\
   \mathbf{D}^2 \, :
 \!\! & \!\! R^{k,\mathbf{s} (s-1,\lambda,\delta+1)} (\overline{\mathcal{C}_T} , \varLambda^2)
 \!\! & \!\! \to
 \!\! & \!\! C^{k-1,\mathbf{s} (s-1,\lambda,\delta+3)} (\overline{\mathcal{C}_T} , \varLambda^2)
\end{array}
\end{equation}
are continuous.
If, in addition, $1 < \delta < n$, then the nonlinear operators
\begin{equation}
\label{eq.map.G.tilde.mathfrak}
\begin{array}{rrcl}
   \mathcal{Q} \, :
 \!\! & \!\! R^{k-1,\mathbf{s} (s,\lambda,\lambda',\delta+1)} (\overline{\mathcal{C}_T} , \varLambda^2)
 \!\! & \!\! \to
 \!\! & \!\! \mathcal{F}^{k,\mathbf{s} (s-1,\lambda,\lambda',\delta+2)} (\overline{\mathcal{C}_T} , \varLambda^1),
\\
   \mathbf{D}^2 \, :
 \!\! & \!\! R^{k-1,\mathbf{s} (s,\lambda,\lambda',\delta+1)} (\overline{\mathcal{C}_T} , \varLambda^2)
 \!\! & \!\! \to
 \!\! & \!\! \mathcal{F}^{k-1,\mathbf{s} (s-1,\lambda,\lambda',\delta+3)} (\overline{\mathcal{C}_T} , \varLambda^2)
\end{array}
\end{equation}
are compact and continuous.
\end{lemma}

\begin{proof}
Indeed, given any elements $g^{(0)}$ and $g$ in
   $C^{k,\mathbf{s} (s-1,\lambda,\delta+1)} (\overline{\mathcal{C}_T} , \varLambda^2)$,
we get
\begin{eqnarray}
\label{eq.G.point}
\lefteqn{
   \mathcal{Q} g
 = \ast (\ast g \wedge d^\ast (\varPhi \otimes I) g)
}
\nonumber
\\
 & = &
   \ast (\ast (g-g^{(0)}) \wedge  d^\ast (\varPhi \otimes I) g^{(0)})
 + \ast (\ast g^{(0)} \wedge d^\ast (\varPhi \otimes I) (g-g^{(0)}))
\nonumber
\\
 & + &
   \ast (\ast (g-g^{(0)}) \wedge  d^\ast (\varPhi \otimes I) (g-g^{(0)}))
 + \ast (\ast g^{(0)} \wedge d^\ast (\varPhi \otimes I) g^{(0)}).
\nonumber
\\
\end{eqnarray}

Fix $g^{(0)}$ and any real $\delta'$ satisfying $1 \leq \delta' \leq \delta$.
From (\ref{eq.G.point}) and Lemma \ref{l.product} it follows that
\begin{eqnarray}
\label{eq.G.est.0}
\lefteqn{
   \| \mathcal{Q} g - \mathcal{Q} g^{(0)}
   \|_{C^{k,\mathbf{s} (s-1,\lambda,\delta+2)} (\cdot)}
}
\nonumber
\\
 & \leq &
   c\,
   \Big(
   \| g - g^{(0)} \|_{C^{k,\mathbf{s} (s-1,\lambda,\delta'+1)} (\cdot)}
   \|  d^\ast (\varPhi \otimes I) g^{(0)} \|_{C^{k,\mathbf{s} (s-1,\lambda,\delta-\delta'+1)} (\cdot)}
\nonumber
\\
 &  &
 +\
   \| g^{(0)} \|_{C^{k,\mathbf{s} (s-1,\lambda,\delta-\delta'+2)} (\cdot)}
   \| d^\ast (\varPhi \otimes I) (g - g^{(0)}) \|_{C^{k,\mathbf{s} (s-1,\lambda,\delta')} (\cdot)}
\nonumber
\\
 &  &
 +\
   \| g - g^{(0)} \|_{C^{k,\mathbf{s} (s-1,\lambda,\delta'+1)} (\cdot)}
   \| d^\ast (\varPhi \otimes I) (g - g^{(0)})
   \|_{C^{k,\mathbf{s} (s-1,\lambda,\delta-\delta'+1)} (\cdot)}
   \Big)
\nonumber
\\
 & \leq &
   c\,
   \Big(
   \| g - g^{(0)} \|_{C^{k,\mathbf{s} (s-1,\lambda,\delta'+1)} (\cdot)}
   \| g^{(0)} \|_{C^{k,\mathbf{s} (s-1,\lambda,\delta-\delta'+2)} (\cdot)}
\nonumber
\\
 &  &
 +\
   \| g^{(0)} \|_{C^{k,\mathbf{s} (s-1,\lambda,\delta-\delta'+2)} (\cdot)}
   \| g - g^{(0)} \|_{C^{k,\mathbf{s} (s-1,\lambda,\delta'+1)} (\cdot)}
\nonumber
\\
 &  &
 +\
   \| g - g^{(0)} \|_{C^{k,\mathbf{s} (s-1,\lambda,\delta'+1)} (\cdot)}
   \| g - g^{(0)} \|_{C^{k,\mathbf{s} (s-1,\lambda,\delta-\delta'+2)} (\cdot)}
   \Big)
\nonumber
\\
 & \leq &
   c\,
   \Big(
   \| g - g^{(0)} \|_{C^{k,\mathbf{s} (s-1,\lambda,\delta'+1)} (\cdot)}
   \| g^{(0)} \|_{C^{k,\mathbf{s} (s-1,\lambda,\delta+1)} (\cdot)}
\nonumber
\\
 &  &
 +\
   \| g^{(0)} \|_{C^{k,\mathbf{s} (s-1,\lambda,\delta+1)} (\cdot)}
   \| g - g^{(0)} \|_{C^{k,\mathbf{s} (s-1,\lambda,\delta'+1)} (\cdot)}
\nonumber
\\
 &  &
 +\
   \| g - g^{(0)} \|_{C^{k,\mathbf{s} (s-1,\lambda,\delta'+1)} (\cdot)}
   \| g - g^{(0)} \|_{C^{k,\mathbf{s} (s-1,\lambda,\delta+1)} (\cdot)}
   \Big)
\nonumber
\\
\end{eqnarray}
with $c$ a constant independent of $g^{(0)}$ and $g$,
   the last inequality being due to the fact that $\delta - \delta' + 2 \leq \delta + 1$ if
   and only if $1 \leq \delta'$.
By abuse of notation we omit the domain and target bundles in designations.
Choosing $\delta' = \delta$ we deduce that if
   $g_\nu \to g^{(0)}$ in the space $C^{k,\mathbf{s} (s-1,\lambda,\delta+1)} (\cdot)$
then
   $\mathcal{Q} g_\nu \to \mathcal{Q} g^{(0)}$ in the space
   $C^{k,\mathbf{s} (s-1,\lambda,\delta+2)} (\cdot)$.
Moreover Lemma \ref{l.diff.oper} implies that
   the sequence $d \mathcal{Q} g_\nu$ converges to $d \mathcal{Q} g^{(0)}$ in the space
   $C^{k-1,\mathbf{s} (s-1,\lambda,\delta+3)} (\cdot)$.
Thus the nonlinear mappings (\ref{eq.map.G.tilde}) are continuous, as desired.

If moreover $\delta > 1$ then there is a real $\delta'$ such that
   $1 \leq \delta' < \delta$.
According to Lemma \ref{l.mathfrak.compact}, if $S$ is a bounded set in
$$
   \mathcal{F}^{k-1,\mathbf{s} (s,\lambda,\lambda',\delta+1)} (\cdot)
 = C^{k,\mathbf{s} (s,\lambda,\delta+1)} (\cdot) \cap
   C^{k-1,\mathbf{s} (s,\lambda',\delta+1)} (\cdot),
$$
then there is a sequence $\{ g_\nu \} \subset S$ converging in the space
$$
   \mathcal{F}^{k,\mathbf{s} (s-1,\lambda,\lambda',\delta'+1)} (\cdot)
 = C^{k+1,\mathbf{s} (s-1,\lambda,\delta'+1)} (\cdot) \cap
   C^{k,\mathbf{s} (s-1,\lambda',\delta'+1)} (\cdot)
$$
to a limit $g^{(0)}$.
Estimate (\ref{eq.G.est.0}) yields
$$
\begin{array}{rcl}
   \| \mathcal{Q} g_\nu - \mathcal{Q} g^{(0)} \|_{C^{k+1,\mathbf{s} (s-1,\lambda,\delta+2)} (\cdot)}
 & \to
 & 0,
\\
   \| \mathcal{Q} g_\nu - \mathcal{Q} g^{(0)} \|_{C^{k,\mathbf{s} (s-1,\lambda',\delta+2)} (\cdot)}
 & \to
 & 0,
\end{array}
$$
as $\nu \to \infty$.
On summing up we see that the sequence $\mathcal{Q} g_\nu$ converges to $\mathcal{Q} g^{(0)}$
in the norm of
   $\mathcal{F}^{k,\mathbf{s} (s-1,\lambda,\lambda',\delta+2)} (\cdot)$.
Hence, the mapping $\mathcal{Q}$ of (\ref{eq.map.G.tilde.mathfrak}) is compact.
Finally, Lemma \ref{l.diff.oper} implies that the mapping $\mathbf{D}^2$ of (\ref{eq.map.G.tilde.mathfrak})
is compact, too, because it factors continuously through $\mathcal{Q}$.
\end{proof}

\begin{lemma}
\label{l.sol.psi.n}
Let
   $s \geq 1$ and $k \geq 2$ be integers,
   $0 < \lambda <1$,
and let
   $0 < \delta < n$ be different from $n-2$ and $n-1$.
If
   $g_0 \in  R^{k-1,\mathbf{s} (s,\lambda,\delta+1)} (\overline{\mathcal{C}_T} , \varLambda^2) \cap
             \mathcal{S}_d$,
then any solution
   $g \in  R^{k-1,\mathbf{s} (s,\lambda,\delta+1)} (\overline{\mathcal{C}_T} , \varLambda^2)$
to the equation
\begin{equation}
\label{eq.heat.pseudo.n}
   g + \varPsi_\mu \mathbf{D}^2 g = g_0
\end{equation}
automatically satisfies $dg\, (\cdot,t) = 0$ in $\mathbb{R}^n$ for each $t \in [0,T]$.
\end{lemma}

\begin{proof}
Indeed, if $n = 2$, then $dg = 0$ because $d^2 = 0$.
If $n \geq 3$, it follows from (\ref{eq.deRham}) and (\ref{eq.commute}) that
\begin{eqnarray*}
   d \left( g + \varPsi_\mu \mathbf{D}^2 g \right)
 & = &
   dg + \varPsi_\mu d \mathbf{D}^2 g
\\
 & = &
   dg
\\
 & = &
   d g_0,
\end{eqnarray*}
and so $dg  = 0$.
\end{proof}

Our next result interprets  Lemma \ref{l.reduce.psi} within the context of (nonlinear)
Navier-Stokes equations.

\begin{lemma}
\label{l.reduce.psi.n}
Suppose that
   $s \geq 1$ and $k \geq 2$ are integers,
   $0 < \lambda <1$,
and
   $0 < \delta < n$ is different from $n-2$ and $n-1$.
Let moreover $F = (f,u_0)^T$ be an arbitrary pair of
$$
   C^{k,\mathbf{s} (s-1,\lambda,\delta)} (\overline{\mathcal{C}_T} , \varLambda^1)
 \times
   C^{2s+k,\lambda,\delta} (\mathbb{R}^n, \varLambda^1) \cap \mathcal{S}_{d^\ast}.
$$

1)
If
$
   U
 \! = \! (u,p)^T
 \in
   C^{k-1,\mathbf{s} (s,\lambda,\delta)} (\overline{\mathcal{C}_T} ,\varLambda^1) \cap \mathcal{S}_{d^\ast}
 \times
   C^{k-1,\mathbf{s} (s-1,\lambda,\delta-1)} (\overline{\mathcal{C}_T})
$
satisfies
   $du \in R^{k-1,\mathbf{s} (s,\lambda,\delta+1)} (\overline{\mathcal{C}_T} ,\varLambda^2)$
and
\begin{equation}
\label{eq.NS.lin.psi.n}
   \mathcal{A} U = F
\end{equation}
then $g = du$ is a solution to equation (\ref{eq.heat.pseudo.n}) with
   $g_0 = \varPsi_{\mu,0} du_0 + \varPsi_\mu df$.

2)
Conversely, if
$
   g
 \in
   R^{k-1,\mathbf{s} (s,\lambda,\delta+1)} (\overline{\mathcal{C}_T} ,\varLambda^2)
$
is a solution to equation (\ref{eq.heat.pseudo.n}) with
$
   g_0
 = \varPsi_{\mu,0} du_0 + \varPsi_\mu df
 \in R^{k-1,\mathbf{s} (s,\lambda,\delta+1)} (\overline{\mathcal{C}_T} ,\varLambda^2)
$
then the pair
$$
\begin{array}{rcl}
   u
 & =
 & d^\ast (\varPhi \otimes I) g,
\\
   p
 & =
 & d^\ast (\varPhi \otimes I) \left( f - H_\mu u - \mathbf{D}^1 u \right)
\end{array}
$$
belongs to
$
   C^{k-1,\mathbf{s} (s,\lambda,\delta)} (\overline{\mathcal{C}_T} ,\varLambda^1) \cap \mathcal{S}_{d^\ast}
 \times
   C^{k-1,\mathbf{s} (s-1,\lambda,\delta-1)} (\overline{\mathcal{C}_T}),
$
satisfies (\ref{eq.NS.lin.psi.n}) and
   $du \in R^{k-1,\mathbf{s} (s,\lambda,\delta+1)} (\overline{\mathcal{C}_T} ,\varLambda^2)$.
\end{lemma}

\begin{proof}
1)
Let $U = (u,p)^T$ be a solution to (\ref{eq.NS.lin.psi.n}).
From (\ref{eq.commute}) it follows that
   $d H_\mu u = H_\mu du$,
and so using (\ref{eq.VW.n}) we obtain
$$
   \Big\{
   \begin{array}{rclll}
     H_\mu du (\cdot,t) +  \mathbf{D}^2 du (\cdot,t)
   & =
   & df (\cdot,t)
   & \mbox{on}
   & \mathbb{R}^n,
\\
     du (\cdot,0)
   & =
   & d u_0
   & \mbox{on}
   & \mathbb{R}^{n}
\end{array}
$$
for all $t \in [0,T]$,
   the last equality being a consequence of $d\, \gamma_0 u = \gamma_0\, du$.
It remains to apply
   Lemmata \ref{l.homo.map.t}, \ref{l.bound.heat.initial.hoelder}, \ref{l.heat.key1}
   and Theorem \ref{t.heat.key2}
to see that
   $g = du$
is of class $R^{k-1,\mathbf{s} (s,\lambda,\delta+1)} (\overline{\mathcal{C}_T})$ and satisfies
equation (\ref{eq.heat.pseudo.n}).
(Obviously, $g$ is closed in the layer.)

2)
Set
   $g_0 = \varPsi_{\mu,0} d u_0 + \varPsi_\mu df$.
Since
   $d \varPsi_{\mu,0} = \varPsi_{\mu,0} d$ and
   $d \varPsi_\mu = \varPsi_\mu d$,
it follows by Lemmata \ref{l.homo.map.t} and \ref{l.homo.map} that
$
   g_0
 \in
   R^{k-1,\mathbf{s} (s,\lambda,\delta+1)} (\overline{\mathcal{C}_T} ,\varLambda^2) \cap \mathcal{S}_d,
$
if $0 < \delta < n$.
Hence, any solution $g$ to (\ref{eq.heat.pseudo.n}) is in
   $R^{k-1,\mathbf{s} (s,\lambda,\delta+1)} (\overline{\mathcal{C}_T} ,\varLambda^2) \cap \mathcal{S}_d$
because of Lemma \ref{l.sol.psi.n}.

Now, Corollary \ref{c.deRham.Hoelder.t} implies that
   $u = d^\ast (\varPhi \otimes I) g$
is an one-form with coefficients in
   $R^{k-1,\mathbf{s} (s,\lambda,\delta)} (\overline{\mathcal{C}_T})$
satisfying
$
   du
 = g
 \in
   R^{k-1,\mathbf{s} (s,\lambda,\delta+1)} (\overline{\mathcal{C}_T} ,\varLambda^2)
$
in the layer.
Using Lemma \ref{l.bound.heat.initial.hoelder} and formula (\ref{eq.VW.n}) we see that
$$
   \Big\{
   \begin{array}{rclll}
      d \left( H_\mu u + \mathbf{D}^1 u - f \right)
   & =
   & 0
   & \mbox{in}
   & \overline{\mathcal{C}_T} ,
\\
     d \left( \gamma_0 u - u_0 \right)
   & =
   & 0
   & \mbox{on}
   & \mathbb{R}^n.
   \end{array}
$$
As $0 < \delta < n$ and
$$
   H_\mu u + \mathbf{D}^1 u - f
 \in
   C^{k-1,\mathbf{s} (s-1,\lambda,\delta)} (\overline{\mathcal{C}_T} ,\varLambda^1) \cap \mathcal{S}_d,
$$
an application of Corollary \ref{c.deRham.Hoelder.t} shows that
$
   p = d^\ast (\varPhi \otimes I) \left( f - H_\mu u - \mathbf{D}^1 u \right)
$
belongs to the space
$
   C^{k-1,\mathbf{s} (s-1,\lambda,\delta-1)} (\overline{\mathcal{C}_T})
$
and it satisfies
$$
   H_\mu u + \mathbf{D}^1 u + dp = f
$$
in $\overline{\mathcal{C}_T} $.
Finally, since
$$
\begin{array}{rcl}
   d \left( \gamma_0 u - u_0 \right)
 & =
 & 0,
\\
   d^\ast \left( \gamma_0 u - u_0 \right)
 & =
 & 0,
\end{array}
$$
we get $\gamma_0 u = u_0$ in all of $\mathbb{R}^n$ because of Corollary \ref{c.deRham.Hoelder}.
Hence, the pair $U = (u,p)^T$ satisfies (\ref{eq.NS.lin.psi.n}).
\end{proof}

We are already in a position to state an open mapping theorem for the reduced equation
(\ref{eq.heat.pseudo.n}).

\begin{theorem}
\label{t.OpenMap.psi}
Assume that
   $s$ and $k \geq 2$ are positive integers,
   $0 < \lambda < \lambda' < 1$,
and
   $n/2 < \delta < n$ is different from $n-2$ and $n-1$.
Then $\varPsi_\mu \mathbf{D}^2$ is a compact continuous selfmapping of the space
$
   R^{k-1,\mathbf{s} (s,\lambda,\lambda',\delta+1)} (\overline{\mathcal{C}_T} , \varLambda^2) \cap
   \mathcal{S}_d.
$
Moreover, the mapping
\begin{equation}
\label{eq.heat.pseudo.nn}
   I \! + \! \varPsi_\mu \mathbf{D}^2 :
   R^{k\!-\!1,\mathbf{s} (s,\lambda,\lambda',\delta\!+\!1)} (\overline{\mathcal{C}_T} , \varLambda^2) \cap
   \mathcal{S}_d
 \! \to \!
   R^{k\!-\!1,\mathbf{s} (s,\lambda,\lambda',\delta\!+\!1)} (\overline{\mathcal{C}_T} , \varLambda^2) \cap
   \mathcal{S}_d
\end{equation}
is Fredholm, injective, and open.
\end{theorem}

\begin{proof}
First we note that Lemmata \ref{l.homo.map.t}, \ref{l.map.G} and \ref{l.sol.psi.n} imply that the
operator $\varPsi_\mu \mathbf{D}^2$ maps
$
   R^{k-1,\mathbf{s} (s,\lambda,\lambda',\delta+1)} (\overline{\mathcal{C}_T} , \varLambda^2) \cap
   \mathcal{S}_d
$
continuously into the space itself.
Since we have $\mathbf{D}^2 = d \mathcal{Q}$, the continuity and compactness of the mapping
   $\varPsi_\mu \mathbf{D}^2$
follow from
   Lemma \ref{l.map.G} and
   Theorem \ref{t.heat.key2}.
We now turn to the one-to-one property of the mapping (\ref{eq.heat.pseudo.nn}).

\begin{lemma}
\label{l.unique.psi.n}
Let
   $s \geq 1$ and $k \geq 2$ be integers,
   $0 < \lambda < 1$,
and
   $n/2 < \delta < n$ be different from $n-2$ and $n-1$.
If
$
   g_0
 \in
   R^{k-1,\mathbf{s} (s,\lambda,\delta)} (\overline{\mathcal{C}_T} , \varLambda^2) \cap
   \mathcal{S}_d
$
then equation (\ref{eq.heat.pseudo.n}) has no more than one solution in
$
   R^{k-1,\mathbf{s} (s,\lambda,\delta+1)} (\overline{\mathcal{C}_T} , \varLambda^2).
$
\end{lemma}

\begin{proof}
Suppose that
$
   g', g''
 \in
   R^{k-1,\mathbf{s} (s,\lambda,\delta+1)} (\overline{\mathcal{C}_T} , \varLambda^2)
$
are two solutions to (\ref{eq.heat.pseudo.n}).
By Lemma \ref{l.sol.psi.n}, they satisfy
   $d g' (\cdot,t) = 0$ and
   $d g'' (\cdot,t) = 0$
in $\mathbb{R}^n$ for each $t \in [0,T]$.

Since
$
   g_0
 \in
   R^{k-1,\mathbf{s} (s,\lambda,\delta)} (\overline{\mathcal{C}_T} , \varLambda^2) \cap
   \mathcal{S}_d
$
we see by Corollaries \ref{c.deRham.Hoelder} and \ref{c.deRham.Hoelder.t} that there
are unique forms
$$
\begin{array}{rcccl}
   u_0
 & =
 & d^\ast \varPhi \left( \gamma_0 g_0 \right)
 & \in
 & C^{2s+k,\lambda,\delta} (\mathbb{R}^n, \varLambda^1),
\\
   f
 & =
 & d^\ast (\varPhi \otimes I) H_\mu g_0
 & \in
 & C^{k-1,\mathbf{s} (s-1,\lambda,\delta)} (\overline{\mathcal{C}_T} , \varLambda^1),
\end{array}
$$
such that
$$
\begin{array}{rcl}
   d u_0
 & =
 & \gamma_0 g_0,
\\
   df
 & =
 & H_\mu g_0
\end{array}
$$
and
$
   g_0 = \varPsi_{\mu,0} d u_0 + \varPsi_\mu df
$
in $\overline{\mathcal{C}_T}$.
Lemma \ref{l.reduce.psi} implies that the pairs
   $U' = (u',p')^T$ and
   $U'' = (u'',p'')^T$
with
$$
\begin{array}{rclcrcl}
   u'
 & =
 & d^\ast (\varPhi \otimes I) g',
 &
 & u''
 & =
 & d^\ast (\varPhi \otimes I) g'',
\\
   p'
 & =
 & d^\ast (\varPhi \otimes I) \left( f - H_\mu u' - \mathbf{D}^1 u' \right),
 &
 & p''
 & =
 & d^\ast (\varPhi \otimes I) \left( f - L_\mu u'' - \mathbf{D}^1 u'' \right)
\end{array}
$$
belong to
$
   C^{k-2,\mathbf{s} (s,\lambda,\delta)} (\overline{\mathcal{C}_T} , \varLambda^1)
 \times
   C^{k,\mathbf{s} (s-1,\lambda,\delta-1)} (\overline{\mathcal{C}_T})
$
and satisfy
   $\mathcal{A} U' = F$ and
   $\mathcal{A} U'' = F$
in the layer, where $F = (f,u_0)^T$.
By the uniqueness of Theorem \ref{t.NS.unique}, we get $U' = U''$.

In particular, $u' = u''$, and so
$
   d^\ast (\varPhi \otimes I) g' = d^\ast (\varPhi \otimes I) g''.
$
Since both $g'$ and $g''$ belong to
$
   R^{k-1,\mathbf{s} (s,\lambda,\delta+1)} (\overline{\mathcal{C}_T} , \varLambda^2) \cap
   \mathcal{S}_d,
$
it follows that
$$
\begin{array}{rcrcl}
   \varDelta \left( d^\ast (\varPhi \otimes I) g' - d^\ast (\varPhi \otimes I) g'' \right)
 & =
 & d^\ast \left( g' - g'' \right)
 & =
 & 0,
\\
 &
 & d \left( g' - g'' \right)
 & =
 & 0
\end{array}
$$
in $\overline{\mathcal{C}_T} $.
Now Corollary \ref{c.deRham.Hoelder.t} yields $g' = g''$, as desired.
\end{proof}

Lemma \ref{l.unique.psi.n} implies immediately that the mapping in (\ref{eq.heat.pseudo.nn}) is
actually one-to-one.

An easy calculation shows that the Frech\'et derivative of the map $I + \varPsi_\mu \mathbf{D}^2$
at an arbitrary point
$$
   g ^{(0)}
 \in
   \mathcal{F}^{k-1,\mathbf{s} (s,\lambda,\lambda',\delta+1)} (\overline{\mathcal{C}_T} ,\varLambda^2) \cap
   \mathcal{S}_d
 \subset
   \mathcal{F}^{k,\mathbf{s} (s-1,\lambda,\lambda',\delta+1)} (\overline{\mathcal{C}_T} ,\varLambda^2) \cap
   \mathcal{S}_d
$$
is given by
$$
   (I + \varPsi_\mu \mathbf{D}^2)'_{g^{(0)}} g
 = (I + \varPsi_\mu W_0) g.
$$
Here, the mapping $W_0$ is constructed by means of
   $g^{(0)}$ and
   $v^{(1)} = d^\ast (\varPhi \otimes I) g^{(0)}$,
the latter one-form belonging to
$$
   \mathcal{F}^{k-1,\mathbf{s} (s,\lambda,\lambda',\delta)} (\overline{\mathcal{C}_T} ,\varLambda^1)
 \subset
   \mathcal{F}^{k,\mathbf{s} (s-1,\lambda,\lambda',\delta)} (\overline{\mathcal{C}_T} ,\varLambda^1).
$$
Now Corollary \ref{c.invertible.psi} shows that the Frech\'et derivative
$
   {(I + \varPsi_\mu \mathbf{D}^2)}'_{g^{(0)}}
$
is a continuously invertible selfmapping of the space
$
   R^{k-1,\mathbf{s} (s,\lambda,\lambda',\delta+1)} (\overline{\mathcal{C}_T} ,\varLambda^2) \cap
   \mathcal{S}_d,
$
for each $g^{(0)}$ as above.

Finally, as  the Frech\'et derivative is a continuously invertible linear operator at
each point $g ^{(0)}$ of
   $R^{k-1,\mathbf{s} (s,\lambda,\lambda',\delta+1)} (\overline{\mathcal{C}_T} ,\varLambda^2))$,
the openness of (\ref{eq.heat.pseudo.nn}) follows from the implicit mapping theorem in
Banach spaces,
   see for instance Theorem 5.2.3 of \cite[p.~101]{Ham82}.
\end{proof}

Theorem \ref{t.OpenMap.psi} suggests two directions for the development of the topic.
First, one can use the standard fixed point techniques including mapping degree theory
to handle operator equation (\ref{eq.heat.pseudo.n}).
Second, one can take into account the properties of the so-called clopen set.

\begin{corollary}
\label{c.clopen}
Assume $n \geq 2$.
Let
   $s \geq 1$ and $k \geq 2$ be integers,
   $0 < \lambda < \lambda' < 1$,
and let
   $n/2 < \delta < n$ be different from $n-2$ and $n-1$.
If the range of mapping (\ref{eq.heat.pseudo.nn}) is closed then it coincides with the
whole destination space.
\end{corollary}

\begin{proof}
Since the destination space is convex, it is connected.
As is known, the only clopen (closed and open) sets in a connected topological vector space
are the empty set and the space itself.
Hence, the range of the mapping $I + \varPsi_\mu \mathbf{D}^2$ is closed if and only if it
coincides with the whole destination space.
\end{proof}

When combined with Lemma \ref{l.reduce.psi.n}, Theorem \ref{t.OpenMap.psi} implies that
the Navier-Stokes equations induce an open noncoercive mapping in the function spaces under
consideration.

\begin{corollary}
\label{c.open.NS.short}
Suppose
   $n \geq 2$.
Let
   $s \geq 1$ and $k \geq 2$ be integers,
   $0 < \lambda < \lambda' < 1$,
and let
   $n/2 < \delta < n$ be different from $n-2$ and $n-1$.
Then, for any pair $U^{(0)} = (u^{(0)},p^{(0)})^T$ of
$
   \mathcal{F}^{k-1,\mathbf{s} (s,\lambda,\lambda',\delta)} (\overline{\mathcal{C}_T}, \varLambda^1) \cap
   \mathcal{S}_{d^\ast}
 \times
   \mathcal{F}^{k-1,\mathbf{s} (s-1,\lambda,\lambda',\delta-1)} (\overline{\mathcal{C}_T}),
$
such that
$$
\begin{array}{rcl}
   d u^{(0)}
 & \in
 & R^{k-1,\mathbf{s} (s,\lambda,\lambda',\delta+1)} (\overline{\mathcal{C}_T}, \varLambda^2),
\\
   \mathcal{A} U^{(0)}
 & \in
 & \mathcal{F}^{k-1,\mathbf{s} (s,\lambda,\lambda',\delta)} (\overline{\mathcal{C}_T}, \varLambda^1)
 \times
   C^{2s+k+1,\lambda,\delta} (\mathbb{R}^n,\varLambda^1) \cap \mathcal{S}_{d^\ast},
\end{array}
$$
there is $\varepsilon > 0$ with the property that for all data $F = (f,u_0)^T$ of the product
$
   \mathcal{F}^{k,\mathbf{s} (s,\lambda,\lambda',\delta)} (\overline{\mathcal{C}_T} ,\varLambda^1)
 \times
   C^{2s+k+1,\lambda,\delta} (\mathbb{R}^n,\varLambda^1) \cap \mathcal{S}_{d^\ast}
$
satisfying the estimate
\begin{equation}
\label{eq.NS.open.est}
   \| d (f - (H_\mu + \mathbf{D}^1) u^{(0)})
   \|_{\mathcal{F}^{k-1,\mathbf{s} (s,\lambda,\lambda',\delta+1)} (\cdot)}
 + \| d (u_0 - \gamma_0 u^{(0)})
   \|_{C^{2s+k,\lambda,\delta+1} (\cdot)}
 < \varepsilon
\end{equation}
the nonlinear equation
   $\mathcal{A} U = F$
has a unique solution $U = (u,p)^T$ in
$$
   \mathcal{F}^{k-1,\mathbf{s} (s,\lambda,\lambda',\delta)} (\overline{\mathcal{C}_T} , \varLambda^1) \cap
   \mathcal{S}_{d^\ast}
 \times
   \mathcal{F}^{k-1,\mathbf{s} (s-1,\lambda,\lambda',\delta-1)} (\overline{\mathcal{C}_T})
$$
with
$
   du \in R^{k-1,\mathbf{s} (s,\lambda,\lambda',\delta+1)} (\overline{\mathcal{C}_T} , \varLambda^2).
$
\end{corollary}

The proof of the corollary actually shows that
$
   \| u - u^{(0)} \|_{\mathcal{F}^{k-1,\mathbf{s} (s,\lambda,\lambda',\delta)} (\cdot)}
 \leq
   c\, \varepsilon,
$
where $c$ is a constant depending only on certain norms of the potentials
   $d^\ast (\varPhi \otimes I)$ and
   $\varPsi_{\mu}$,
   $\varPsi_{\mu,0}$,
and the inverse operator
   $(I + \varPsi_\mu W_0)^{-1}$,
but not on the data $f$ and $u_0$.
Here, the operator $W_0$ is constructed by means of the forms
   $g^{(0)} = d u ^{(0)}$ and
   $v^{(1)} = u^{(0)}$.

\begin{proof}
By Lemma \ref{l.heat.key1}, the volume parabolic potential $\varPsi_\mu$ induces bounded linear
operators
$$
\begin{array}{rcl}
   C^{k+1,\mathbf{s} (s,\lambda,\delta+1)} (\overline{\mathcal{C}_T})
 & \to
 & C^{k+1,\mathbf{s} (s,\lambda,\delta+1)} (\overline{\mathcal{C}_T}) \cap \mathcal{D}_{H_\mu},
\\
   C^{k,\mathbf{s} (s,\lambda',\delta+1)} (\overline{\mathcal{C}_T})
 & \to
 & C^{k,\mathbf{s} (s,\lambda',\delta+1)} (\overline{\mathcal{C}_T}) \cap \mathcal{D}_{H_\mu},
\end{array}
$$
and hence a bounded linear operator
\begin{equation}
\label{eq.oper.psi}
   \mathcal{F}^{k,\mathbf{s} (s,\lambda,\lambda',\delta+1)} (\overline{\mathcal{C}_T})
 \to
   \mathcal{F}^{k,\mathbf{s} (s,\lambda,\lambda',\delta+1)} (\overline{\mathcal{C}_T}) \cap
   \mathcal{D}_{H_\mu}.
\end{equation}

Similarly, Lemma \ref{l.heat.key1} implies that the Poisson parabolic potential $\varPsi_{\mu,0}$
induces a bounded linear operator
$$
   C^{2s+k+1,\lambda,\delta+1} (\mathbb{R}^n)
 \to
   C^{k+1,\mathbf{s} (s,\lambda,\delta+1)} (\overline{\mathcal{C}_T}) \cap \mathcal{D}_{H_\mu}.
$$
On the other hand, by Theorem \ref{t.emb.hoelder}, the space
   $C^{2s+k+1,\lambda,\delta+1} (\mathbb{R}^n)$
is continuously embedded into the space
   $C^{2s+k,\lambda',\delta+1} (\mathbb{R}^n)$.
Hence it follows that the potential $\varPsi_{\mu,0}$ induces a bounded linear operator
$$
   C^{2s+k+1,\lambda,\delta+1} (\mathbb{R}^n)
 \to
   C^{k,\mathbf{s} (s,\lambda',\delta+1)} (\overline{\mathcal{C}_T}) \cap \mathcal{D}_{H_\mu},
$$
and so a bounded linear operator
\begin{equation}
\label{eq.oper.psi0}
   C^{2s+k+1,\lambda,\delta+1} (\mathbb{R}^n)
 \to
   \mathcal{F}^{k,\mathbf{s} (s,\lambda,\lambda',\delta+1)} (\overline{\mathcal{C}_T}) \cap
   \mathcal{D}_{H_\mu}.
\end{equation}

We now apply Lemma \ref{l.reduce.psi.n} to see that the form
   $g^{(0)} = d u^{(0)}$
is a solution to the equation
$$
   (I + \varPsi_\mu \mathbf{D}^2) g^{(0)} = g_0^{(0)}
$$
with the right-hand side
   $g_0^{(0)} := \varPsi_{\mu,0} d \gamma_0 u^{(0)} + \varPsi_\mu d (H_\mu + \mathbf{D}^1) u^{(0)}$
belonging to the space
   $R^{k-1,\mathbf{s} (s,\lambda,\lambda',\delta+1)} (\overline{\mathcal{C}_T},\varLambda^2)$.
Set
$$
   g_0 = \varPsi_{\mu,0} d u_0 + \varPsi_\mu df,
$$
which belongs to
   $\mathcal{F}^{k-1,\mathbf{s} (s,\lambda,\lambda',\delta+1)} (\overline{\mathcal{C}_T},\varLambda^2)$
by Lemma  \ref{l.heat.key1}.
An immediate calculations shows that
\begin{eqnarray*}
\lefteqn{
   \| g_0 - g_0^{(0)}
   \|_{\mathcal{F}^{k-1,\mathbf{s} (s,\lambda,\lambda',\delta+1)} (\overline{\mathcal{C}_T},\varLambda^2)}
}
\\
 & \leq &
   \| \varPsi_{\mu,0} \| \,
   \| d (u_0 - \gamma_0 u^{(0)})
   \|_{C^{2s+k,\lambda,\delta+1} (\mathbb{R}^n,\varLambda^2)}
\\
 & + &
   \| \varPsi_\mu \| \,
   \| df - d (H_\mu + \mathbf{D}^1) u^{(0)}
   \|_{\mathcal{F}^{k-1,\mathbf{s} (s,\lambda,\lambda',\delta+1)} (\overline{\mathcal{C}_T},\varLambda^2)}
\end{eqnarray*}
where
   $\| \varPsi_\mu \|$ and
   $\| \varPsi_{\mu,0} \|$
are the norms of bounded linear operators
   (\ref{eq.oper.psi}) and
   (\ref{eq.oper.psi0}),
respectively.

If $\varepsilon > 0$ in the estimate (\ref{eq.NS.open.est}) is small enough, then
Theorem \ref{t.OpenMap.psi} shows that there is a solution
$
   g
 \in
   R^{k-1,\mathbf{s} (s,\lambda,\lambda',\delta+1)} (\overline{\mathcal{C}_T},\varLambda^2) \cap
   \mathcal{S}_d
$
to the equation
$
   g + \varPsi_\mu \mathbf{D}^2 g = g_0.
$
Finally, the pair
$$
\begin{array}{rcl}
   u
 & =
 & d^\ast (\varPhi \otimes I) g,
\\
   p
 & =
 & d^\ast (\varPhi \otimes I) \left( f - H_\mu u - \mathbf{D}^1 u \right)
\end{array}
$$
belongs to
$
   \mathcal{F}^{k-1,\mathbf{s} (s,\lambda,\lambda',\delta)} (\overline{\mathcal{C}_T},\varLambda^1) \cap
   \mathcal{S}_{d^\ast}
 \times
   \mathcal{F}^{k-1,\mathbf{s} (s-1,\lambda,\lambda',\delta-1)} (\overline{\mathcal{C}_T})
$
and satisfies the nonlinear equation $\mathcal{A} U = F$,
   which is due to Lemma \ref{l.reduce.psi.n}.
Moreover, $du = g$ belongs to
$
   R^{k-1,\mathbf{s} (s,\lambda,\lambda',\delta+1)} (\overline{\mathcal{C}_T},\varLambda^2),
$
as desired.
\end{proof}

In the strict sense of the word Corollary \ref{c.open.NS.short} is not an open mapping theorem, for neither
the domain nor the target space has been fixed for the Navier-Stokes equations.
In order to strengthen the assertion we turn to the framework of metric spaces.
For this purpose, denote by $\mathcal{D}_{\mathcal{A}}$ the set of all pairs $U = (u,p)^T$ in the product
$$
   \mathcal{F}^{k-1,\mathbf{s} (s,\lambda,\lambda',\delta)} (\overline{\mathcal{C}_T}, \varLambda^1) \cap
   \mathcal{S}_{d^\ast}
 \times
   \mathcal{F}^{k-1,\mathbf{s} (s-1,\lambda,\lambda',\delta-1)} (\overline{\mathcal{C}_T}),
$$
such that
$$
\begin{array}{rcl}
   du
 & \in
 & R^{k-1,\mathbf{s} (s,\lambda,\lambda',\delta+1)} (\overline{\mathcal{C}_T}, \varLambda^2),
\\
   \mathcal{A} U
 & \in
 & \mathcal{F}^{k,\mathbf{s} (s,\lambda,\lambda',\delta)} (\overline{\mathcal{C}_T}, \varLambda^1)
 \times
   C^{2s+k+1,\lambda,\delta} (\mathbb{R}^n,\varLambda^1) \cap \mathcal{S}_{d^\ast}.
\end{array}
$$
Since $\mathcal{A}$ is nonlinear, the set $\mathcal{D}_{\mathcal{A}}$ fails to bear a vector space structure.
We topologise it under the metric
$$
   d (U,V) := \| U - V \| + \| du - dv \| + \| \mathcal{A} U - \mathcal{A} V \|,
$$
where
$$
\begin{array}{rcl}
   \| U - V \|
 & = &
   \| U - V \|_{\mathcal{F}^{k-1,\mathbf{s} (s,\lambda,\lambda',\delta)} (\overline{\mathcal{C}_T}, \varLambda^1)
         \times \mathcal{F}^{k-1,\mathbf{s} (s-1,\lambda,\lambda',\delta-1)} (\overline{\mathcal{C}_T})},
\\
   \| du - dv \|
 & = &
   \| du - dv
   \|_{\mathcal{F}^{k-1,\mathbf{s} (s,\lambda,\lambda',\delta+1)} (\overline{\mathcal{C}_T}, \varLambda^2)},
\\
   \| \mathcal{A} U - \mathcal{A} V \|
 & = &
   \| \mathcal{A} U - \mathcal{A} V
   \|_{\mathcal{F}^{k,\mathbf{s} (s,\lambda,\lambda',\delta)} (\overline{\mathcal{C}_T}, \varLambda^1)
\times C^{2s+k+1,\lambda,\delta} (\mathbb{R}^n,\varLambda^1)}
\end{array}
$$
for $U = (u,p)^T$ and $V = (v,q)^T$ in $\mathcal{D}_{\mathcal{A}}$.

Since the weighted H\"{o}lder spaces are complete, we conclude immediately that the metric space
$\mathcal{D}_{\mathcal{A}}$ is complete, too.
Indeed, choose a Cauchy sequence $U_\nu = (u_\nu,p_\nu)^T$ in $\mathcal{D}_{\mathcal{A}}$.
By the above, the sequences
   $\{ U_\nu \}$,
   $\{ d u_\nu \}$ and
   $\{ \mathcal{A} U_\nu \}$
converge in the Banach spaces
$$
\begin{array}{cc}
   \mathcal{F}^{k-1,\mathbf{s} (s,\lambda,\lambda',\delta)} (\overline{\mathcal{C}_T}, \varLambda^1)
   \cap \mathcal{S}_{d^\ast} \!
 \times
   \mathcal{F}^{k\!-\!1,\mathbf{s} (s-1,\lambda,\lambda',\delta-1)} (\overline{\mathcal{C}_T}),
 \! & \!
   \mathcal{F}^{k-1,\mathbf{s} (s,\lambda,\lambda',\delta+1)} (\overline{\mathcal{C}_T}, \varLambda^2),
\\
   \mathcal{F}^{k,\mathbf{s} (s,\lambda,\lambda',\delta)} (\overline{\mathcal{C}_T}, \varLambda^1)
 \times
   C^{2s+k+1,\lambda,\delta} (\mathbb{R}^n,\varLambda^1) \cap \mathcal{S}_{d^\ast},
 \! & \!
\end{array}
$$
respectively.
Let $U = (u,p)^T$ stand for the limit of $\{ U_\nu \}$.
A familiar argument using embedding theorems and the continuity of the operators $d^1$ and $\mathcal{A}$
in their domains shows that the limits of the sequences
   $\{ d u_\nu \}$ and
   $\{ \mathcal{A} U_\nu \}$
just amount to $du$ and $\mathcal{A} U$, respectively.
Thus,
$$
\begin{array}{rcl}
   du
 & \in
 & \mathcal{F}^{k-1,\mathbf{s} (s,\lambda,\lambda',\delta+1)} (\overline{\mathcal{C}_T}, \varLambda^2),
\\
   \mathcal{A} U
 & \in
 & \mathcal{F}^{k,\mathbf{s} (s,\lambda,\lambda',\delta)} (\overline{\mathcal{C}_T}, \varLambda^1)
 \times
   C^{2s+k+1,\lambda,\delta} (\mathbb{R}^n,\varLambda^1) \cap \mathcal{S}_{d^\ast}.
\end{array}
$$
Moreover, if $h \in H_{\leq m}$ is a harmonic polynomial of suitable degree in $\mathbb{R}^n$, then
$$
   \int_{\mathbb{R}^n} du (x) h (x)\, dx
 = \lim_{\nu \to \infty} \int_{\mathbb{R}^n} d u_\nu (x) h (x)\, dx
 = 0
$$
because
   $d u_\nu \in R^{k-1,\mathbf{s} (s,\lambda,\lambda',\delta+1)} (\overline{\mathcal{C}_T}, \varLambda^2)$
for all $\nu = 1, 2, \ldots$.
Hence it follows that
   $du \in R^{k-1,\mathbf{s} (s,\lambda,\lambda',\delta+1)} (\overline{\mathcal{C}_T}, \varLambda^2)$
and
   $U \in \mathcal{D}_{\mathcal{A}}$,
which was to be proved.

\begin{corollary}
\label{c.open.NS.short.formal.BIG}
Suppose
   $n \geq 2$,
   $s$ and $k \geq 2$ are positive integers,
   $0 \! < \! \lambda \! < \! \lambda' \! < \! 1$,
and
   $n/2 < \delta < n$ is different from $n-2$ and $n-1$.
Then the mapping
\begin{equation}
\label{eq.metric.A.BIG}
   \mathcal{A}:
   \mathcal{D}_{\mathcal{A}}
 \to
   \mathcal{F}^{k, \mathbf{s} (s,\lambda,\lambda',\delta)} (\overline{\mathcal{C}_T}, \varLambda^1)
 \times
   C^{2s+k+1,\lambda,\delta} (\mathbb{R}^n,\varLambda^1) \cap \mathcal{S}_{d^\ast}
\end{equation}
is continuous, injective and its range is open.
\end{corollary}

\begin{proof}
Indeed, the continuity of mapping (\ref{eq.metric.A.BIG}) follows from the very construction of the
space $\mathcal{D}_{\mathcal{A}}$.
By Theorem \ref{t.NS.unique}, it is injective.
Fix an element $U^{(0)}$ of $\mathcal{D}_{\mathcal{A}}$.
Denote by $\| d^1 \|$ the norm of the exterior derivative which maps
$$
   \mathcal{F}^{k, \mathbf{s} (s,\lambda,\lambda',\delta)} (\overline{\mathcal{C}_T}, \varLambda^1)
 \times
   C^{2s+k+1,\lambda,\delta} (\mathbb{R}^n,\varLambda^1) \cap \mathcal{S}_{d^\ast}
$$
continuously into
$$
   \mathcal{F}^{k-1, \mathbf{s} (s,\lambda,\lambda',\delta+1)} (\overline{\mathcal{C}_T}, \varLambda^2)
 \times
   C^{2s+k,\lambda,\delta+1} (\mathbb{R}^n,\varLambda^2) \cap \mathcal{S}_d.
$$
Then, for any pair $F = (f,u_0)^T$ of
$$
   \mathcal{F}^{k, \mathbf{s} (s,\lambda,\lambda',\delta)} (\overline{\mathcal{C}_T}, \varLambda^1)
 \times
   C^{2s+k+1,\lambda,\delta} (\mathbb{R}^n,\varLambda^1) \cap \mathcal{S}_{d^\ast}
$$
satisfying
$$
   \| \mathcal{A} U^{(0)} - F
   \|_{\mathcal{F}^{k, \mathbf{s} (s,\lambda,\lambda',\delta)} (\overline{\mathcal{C}_T}, \varLambda^1)
       \times C^{2s+k+1,\lambda,\delta} (\mathbb{R}^n,\varLambda^1)}
 <
   \frac{\varepsilon}{\| d^1 \|},
$$
we get estimate (\ref{eq.NS.open.est}) with the constant $\varepsilon$ as in Corollary \ref{c.open.NS.short}.
Thus, the statement follows immediately from Corollary \ref{c.open.NS.short}.
\end{proof}

\begin{remark}
\label{r.scale.n}
For
   $n \geq 3$ and
   $n/2 < \delta < n-1$
we need not work with the sophisticated scale
$
   R^{k-1,\mathbf{s} (s,\lambda,\lambda',\delta+1)} (\overline{\mathcal{C}_T} ,\varLambda^2) \cap
   \mathcal{S}_d
$
of function spaces.
We may simply use the scale
$
   \mathcal{F}^{k-1,\mathbf{s} (s,\lambda,\lambda',\delta+1)} (\overline{\mathcal{C}_T} ,\varLambda^2) \cap
   \mathcal{S}_d.
$
instead, which makes the exposition more transparent.
\end{remark}

We finish the section by mentioning a familiar example by P. Fatou (1922).
He constructed a holomorphic mapping $f (z)$ of $\mathbb{C}^2$ whose Jacobi matrix $f' (z)$ has a constant
determinant different from zero.
The mapping $f$ is a homeomorphism onto the image, however, the image of $f$ leaves out a closed subset of
$\mathbb{C}^2$ with nonempty interior.
This shows that nonlinear Fredholm mappings may behave rather intricately.

\bigskip

\textit{Acknowledgments\,}
The first author gratefully acknowledges the financial support of the grant of the
Russian Federation Government for scientific research under the supervision of
leading scientist at the Siberian Federal University, contract No 14.Y26.31.0006.
and the grant of President of Russian Federation for leading scientific schools
  NSh-9149.2016.1

\end{document}